\documentclass[bibalpha]{amsart}
\usepackage{amssymb}
\usepackage{textcomp}
\usepackage[mathscr]{euscript}
\usepackage{pb-diagram}
\usepackage{enumitem}
\usepackage{comment}

\newtheorem{theorem}{Theorem}[section]
\newtheorem{qst}[theorem]{Question}
\newtheorem{thm}[theorem]{Theorem}
\newtheorem{prop}[theorem]{Proposition}

\newtheorem*{Claim}{Claim}
\newtheorem{fact}[theorem]{Fact}
\newtheorem{cor}[theorem]{Corollary}

\newtheorem{lemma}[theorem]{Lemma}

\newtheorem{conj}{Conjecture}
\newtheorem*{thmA}{Theorem A}

\newtheorem*{thmB}{Theorem B}
\newtheorem*{thmC}{Theorem C}
\newtheorem*{thmD}{Theorem D}
\newtheorem*{thmE}{Theorem E}
\newtheorem*{thmF}{Theorem F}

\newtheorem*{thmH}{Fact 2.3 (Hieronymi's Theorem)}

\theoremstyle{definition}

\theoremstyle{remark}

\newcommand{\az}{\alpha \Z}
\newcommand{\rvec}{\R_{\mathrm{vec}}}
\newcommand{\bvec}{\R_{\mathrm{bvec}}}
\newcommand{\lord}{L_{\mathrm{og}}}
\newcommand{\lvec}{L_{\mathrm{ovec}}}
\newcommand{\tord}{T_{\mathrm{og}}}
\newcommand{\tvec}{T_{\mathrm{ovec}}}
\newcommand{\dimH}{\dim_{\mathrm{H}}}
\newcommand{\pr}{\operatorname{Pr}}

\ProvideTextCommandDefault{\cprime}{(U+042C)}

\newcommand{\gr}{{\operatorname{gr}}}
\newcommand{\st}{\operatorname{st}}

\newcommand{\minf}{\mathbf{Inf}}
\newcommand{\mfin}{\mathbf{Fin}}
\newcommand{\om}{\mfin_a/\minf_a}
\newcommand{\slam}{(\mathscr{S},\lambda^{\mathbb{Z}})}

\newcommand{\Dim}{\operatorname{Dim}}

\newcommand{\cl}{\operatorname{Cl}}
\newcommand{\bd}{\operatorname{Bd}}

\newcommand{\Sh}[1]{\ensuremath{\mathscr{#1}^{\mathrm{Sh}}}}
\newcommand{\Sq}[1]{\ensuremath{\mathscr{#1}^{\square}}}
\newcommand{\nip}{\mathrm{NIP}}
\newcommand{\ntp}{\mathrm{NTP}_2}
\newcommand{\tp}{\mathrm{TP}_2}

\newcommand{\cal}[1]{\ensuremath{\mathcal{#1}}}
\newcommand{\Cal}[1]{\ensuremath{\mathcal{#1}}}
\newcommand{\Sa}[1]{\ensuremath{\mathscr{#1}}}

\newcommand{\ralg}{\R_{\mathrm{alg}}}

\newcommand{\Z}{\mathbb{Z}}
\newcommand{\N}{\mathbb{N}}
\newcommand{\C}{\mathbb{C}}
\newcommand{\Q}{\mathbb{Q}}
\newcommand{\R}{\mathbb{R}}
\newcommand{\DSig}{D_{\Sigma}}

\begin{document}
\title[Quotients and expansions]{externally definable quotients and nip expansions of the real ordered additive group}

\author{Erik Walsberg}
\address{Department of Mathematics, Statistics, and Computer Science\\
Department of Mathematics\\University of California, Irvine, 340 Rowland Hall (Bldg.\# 400),
Irvine, CA 92697-3875}
\email{ewalsber@uci.edu}
\urladdr{http://www.math.illinois.edu/\textasciitilde erikw}

\date{\today}

\maketitle

\begin{abstract}
Let $\Sa R$ be an $\nip$ expansion of $(\R,<,+)$ by closed subsets of $\R^n$ and continuous functions $f : \R^m \to \R^n$.
Then $\Sa R$ is generically locally o-minimal.
It follows that if $X \subseteq \R^n$ is definable in $\Sa R$ then the $C^k$-points of $X$ are dense in $X$ for any $k \geq 0$.
This follows from a more general theorem on $\nip$ expansions of locally compact groups, which itself follows from a result on quotients of definable sets by equivalence relations which are externally definable and $\bigwedge$-definable.
We also show that $\Sa R$ is strongly dependent if and only if $\Sa R$ is either o-minimal or $(\R,<,+,\az)$-minimal for some $\alpha > 0$.
\end{abstract}

\noindent A highly saturated structure $\Sa M$ is $\nip$ if whenever $(I,<)$ is an indiscernible sequence in $\Sa M$ and $X$ is an $\Sa M$-definable (possibly with parameters) set, then $X \cap I$ is a finite union of $<$-convex sets.
So we might view $\nip$ structures as ``weakly o-minimal on indiscernible sequences".
We therefore hope that definable sets in $\nip$ expansions of $(\R,<,+)$ behave similarly to definable sets in o-minimal expansions of $(\R,<,+)$.
One result in this direction is the theorem of Simon~\cite[Corollary 3.7]{Simon-dp} that a dp-rank one expansion of $(\R,<,+)$ is o-minimal.
In this paper we show that $\nip$ expansions of $(\R,<,+)$ by closed sets are very similar to o-minimal expansions. \newline

\noindent Throughout all structures are first order.
When we say that something is definable in a structure we mean that is definable possibly with parameters from that structure.
Two structures on a common domain $M$ are \textbf{interdefinable} if they define the same subsets of all $M^n$.
We regard interdefinable structures as the same.
Let $\Sa M$ be a structure with domain $M$.
The structure induced on $A \subseteq M^m$ by $\Sa M$ is the structure with domain $A$ whose primitive $n$-ary relations are all sets of the form $X \cap A^{n}$ for $\Sa M$-definable $X \subseteq M^{mn}$.
The structure induced on $A$ eliminates quantifiers if every subset of $A^n$ definable in the induced structure is of the form $Y \cap A^n$ for $\Sa M$-definable $Y \subseteq M^{mn}$.
(We will commonly encounter this situation.) \newline

\noindent Suppose $M^n$ is equipped with a topology for all $n \geq 1$.
In this paper we will always equip $M^n$ with the product topology when $n \geq 2$, but the basic definitions are naturally formulated in the more general context.
We say that $\Sa M$ is \textbf{noiseless} if every definable subset of every $M^n$ either has interior or is nowhere dense.
We say that $\Sa M$ is \textbf{noiseless in one variable} if every definable subset of $M$ either has interior or is nowhere dense.
We say that $\Sa M$ is \textbf{strongly noiseless} if the induced structure on any definable $Y \subseteq M^n$ is noiseless.
Equivalently $\Sa M$ is strongly noiseless if whenever $X,Y$ are definable subsets of $M^n$ then $X$ is either nowhere dense in $Y$ or has interior in $Y$.
We say that $\Sa M$ is \textbf{strongly noiseless in one variable} if whenever $X,Y \subseteq M$ are definable then $X$ is either nowhere dense in $Y$ or has interior in $Y$. 
(These definitions are essentially due to Chris Miller.)
\newline

\noindent Strong noiselessness is a typical and important property of structures whose definable sets are ``tame topological" objects.
Algebraically closed fields, (weakly) o-minimal expansions of dense linear orders~\cite{MMS-weak}, P-minimal expansions of $p$-adically closed fields (in particular the field of $p$-adic numbers)~\cite{CCL}, C-minimal expansions of dense C-relations (in particular algebraically closed valued fields)~\cite{HM-Cmin}, and unstable dp-minimal expansions of fields~\cite{Johnson-top,SW-tame} are all strongly noiseless with respect to canonical topologies.
There are many model-theoretically tame structures which are noisey with respect to a canonical topology.
In particular there are many noisey $\nip$ expansions of $(\R,<,+)$ such as $(\R,<,+,\Q)$, $(\R,<,+,\times,\ralg)$ where $\ralg$ is the set of real algebraic numbers, and $(\R,<,+,\times,U)$ where $U$ is the set of complex roots of unity~\cite{GH-Dependent}. \newline

\noindent The \textbf{open core} $\Sa M^\circ$ of $\Sa M$ is the structure on $M$ whose primitive $n$-ary relations are all closures of $\Sa M$-definable subsets of $M^n$.
If $\Sa M$ defines a basis for the topology on each $M^n$ then $\Sa M^\circ$ is the reduct of $\Sa M$ generated by all closed $\Sa M$-definable sets.
If the topology on $M^n$ is Hausdorff, $X \subseteq M^m$ is $\Sa M$-definable and either open or closed, and $f : X \to M^n$ is continuous and $\Sa M$-definable, then $f$ is definable in $\Sa M^\circ$.
(Hausdorffness ensures that the graph of $f$ is closed in $X \times M^n$.)
We say $\Sa M$ is \textbf{generated by closed sets} or is an \textbf{expansion by closed sets} if it is interdefinable with $\Sa M^\circ$.
A subset of a topological space is \textbf{constructible} if it is a boolean combination of closed sets.
The following result of Miller and Dougherty~\cite{MD-boolean} shows that if $\Sa M$ defines a basis for the topology on each $M^n$ then $\Sa M^\circ$ is interdefinable with the reduct of $\Sa M$ generated by all constructible $\Sa M$-definable sets.

\begin{fact}
\label{fact:MD}
Suppose $X$ is an $\Sa M$-definable set, $\Sa M$ defines a basis for a topology on $X$, and $Y \subseteq X$ is constructible and definable.
Then $Y$ is a boolean combination of definable closed sets.
\end{fact}

\noindent It has been observed that $\Sa M^\circ$ is typically noiseless when $\Sa M$ is model-theoretically tame and $\Sa M$ is typically interdefinable with $\Sa M^\circ$ when $\Sa M$ is strongly noiseless (the latter usually follows from some kind of cell decomposition).
It is conjectured that $\Sa R^\circ$ is noiseless if $\Sa R$ expands $(\R,<,+,\times)$ and does not define the set of integers.
It is also conjectured that $\Sa R^\circ$ is strongly noiseless when $\Sa R$ expands $(\R,<,+)$ and does not interpret the monadic second order theory of one successor.
(We discuss these conjectures in Section~\ref{section:main-conjecture}.)
Theorem A generalizes a special case.

\begin{thmA}
Suppose $G$ is a group and $\Sa G$ is a first order expansion of $G$ such that $\Sa G$ defines
\begin{enumerate}
    \item a basis for a locally compact Hausdorff group topology on $G$,
    \item a family $\Cal K$ of compact subsets of $G$ such that any compact subset of $G$ is contained in some element of $\Cal K$.
\end{enumerate}
If $\Sa G$ is $\nip$ then $\Sa G^\circ$ is strongly noiseless.
\end{thmA}

\noindent Note that $(2)$ is superfluous when $G$ is compact.
We do not know how strong the conclusion of Theorem A is in general.
It implies that $\Sa G^\circ$-definable functions are generically continuous, see Proposition~\ref{prop:generic-cont}. 
We will see that Theorem A has very strong consequences over $(\R,<,+)$.
We first describe the proof of Theorem A and then discuss its consequences over $(\R,<,+)$.
Theorem A is a consequence of Theorem B, which we believe to be of independent interest.

\begin{thmB}
Suppose $\Sa N$ is a highly saturated $\nip$ structure and $X$ is an $\Sa N$-definable set.
Let $E$ be an equivalence relation on $X$ which is both $\bigwedge$-definable and externally definable.
Suppose that the Shelah expansion $\Sh N$ of $\Sa N$ defines a basis for the logic uniformity on $X/E$.
Then the structure induced on $X/E$ by $\Sh N$ is strongly noiseless.
\end{thmB}

\noindent Suppose $\Sa N$ is a highly saturated o-minimal expansion of field and $G$ is a definably compact group.
Then $G^{00}$ is known to be externally definable and the structure induced on $G/G^{00}$ by $\Sh N$ is interpretable in a disjoint union of finitely many o-minimal structures.
In particular the induced structure on $G/G^{00}$ is strongly noiseless.
We view Theorem B as a broad generalization of this fact.
As an application of Theorem B we show that if $\Sa N$ is a highly saturated strongly dependent structure, $G$ is a definably amenable group, and $G/G^{00}$ is a simple centerless Lie group, then the structure induced on $G/G^{00}$ by $\Sh N$ is bi-interpretable with an o-minimal expansion of the real field.
We prove this theorem and discuss an interesting possible generalization in Section~\ref{section:pillay} below.\newline

\noindent We now describe the proof of Theorem A from Theorem B.
Let $\Sa N$ be a highly saturated elementary extension of $\Sa G$.
We let $\mfin$ be the subgroup of ``finite" elements of $\Sa N$ and $\minf$ be the subgroup of ``infinitesimal" elements of $\Sa N$.
We show that $\minf$ is $\bigwedge$-definable and $\mfin,\minf$ are both externally definable.
Using familiar ideas from nonstandard analysis we see that $\mfin/\minf$ can be identified with $G$ and the quotient map $\mfin \to \mfin/\minf$ can be identified with the usual standard part map.
We also see that the logic topology agrees with the group topology on $G$.
Finally, we show that $\Sa G^\circ$ is a reduct of the structure induced on $G$ by $\Sh N$. 
Our proof does not rely on the group structure on $G$, only on the induced uniform structure.
For this reason our proof goes through in the more general setting of locally compact Hausdorff uniform spaces, see Section~\ref{section:loc-compact}.
The general result is reasonably sharp, see Section~\ref{section:counterexample}.\newline

\noindent We now describe consequences of Theorem A over the reals.
Suppose $\Sa R$ is an expansion of $(\R,<)$.
We say that $\Sa R$ is \textbf{locally o-minimal} if for every $\Sa R$-definable $X \subseteq \R$ and $a \in X$ there is an open interval $I$ containing $a$ such that $I \cap X$ is definable in $(\R,<)$.
By \cite[Corollary 3.4]{TV-local} $\Sa R$ is locally o-minimal if and only if the induced structure on any bounded interval is o-minimal.
We say that $\Sa R$ is \textbf{generically locally o-minimal} if for every $\Sa R$-definable $X \subseteq \R$ there is a dense definable open $V \subseteq X$ such that for all $a \in V$ there is an open interval $I$ containing $a$ such that $I \cap X$ is definable in $(\R,<)$.
It is known that an expansion of $(\R,<,+)$ is strongly noiseless if and only if it is generically locally o-minimal (see Theorem~\ref{thm:equiv} below).
So Theorem C follows from Theorem B.

\begin{thmC}
Suppose $\Sa R$ is an expansion of $(\R,<,+)$.
If $\Sa R$ is $\nip$ then $\Sa R^\circ$ is generically locally o-minimal.
\end{thmC}

\noindent Theorem C fails over $(\R,<)$.
Section~\ref{section:counterexample} shows that if $f : [0,1] \to [0,1]$ is the classical Cantor function, aka ``devil's staircase", then $(\R,<,f)$ is dp-minimal and noisey.
\newline

\noindent
We say that $\Sa N$ is $\pmb{\Sa M}$-\textbf{minimal} if $\Sa N$ is an expansion of $\Sa M$ and every $\Sa N$-definable subset of $M$ is $\Sa M$-definable.
(If $\Sa M = (\R,<)$ then $\Sa M$-minimality is o-minimality.)
In Section~\ref{section:strongly-dependent} we prove Theorem D by combining Theorem C, results of Dolich and Goodrick~\cite{DG} on strongly dependent expansions of ordered abelian groups, work of Kawakami, Takeuchi, Tanaka, and Tsuboi~\cite{KTTT} on locally o-minimal structures, and a recent result of B\`{e}s and Choffrut on $(\R,<,+,\Z)$~\cite{bes-choffrut}.

\begin{thmD}
Suppose $\Sa R$ is an expansion of $(\R,<,+)$.
The following are equivalent.
\begin{enumerate}
\item $\Sa R$ is a strongly dependent expansion by closed sets.
\item $\Sa R$ is strongly dependent and noiseless.
\item $\Sa R$ is either o-minimal or $(\R,<,+,\az)$-minimal for some $\alpha > 0$.
\item $\Sa R$ is either o-minimal or interdefinable with $(\R,<,+,\Cal B,\az)$ for some collection $\Cal B$ of bounded sets such that $(\R,<,+,\Cal B)$ is o-minimal.
\item $\Sa R$ is either o-minimal or locally o-minimal and interdefinable with $(\Sa S,\az)$ for some o-minimal expansion $\Sa S$ of $(\R,<,+)$. 
\end{enumerate}
In each case above $\alpha$ is unique up to rational multiples.
Suppose $\Sa R$ is $(\R,<,+,\az)$-minimal.
Then every definable subset of $\R^n$ is a finite union of sets of the form $\bigcup_{b \in B} b + A$ for $(\az,<,+)$-definable $B \subseteq (\az)^n$ and $\Sa R$-definable $A \subseteq [0,\alpha)^n$.
It follows that $\Sa R$ is bi-interpretable with the disjoint union of the induced structure on $[0,\alpha)$ (which is o-minimal) and $(\Z,<,+)$.
\end{thmD}

\noindent
Informally, Theorem D shows that a strongly dependent expansion by closed sets which is not o-minimal has a canonical decomposition into an ``integer part", $(\az,<,+)$, and an o-minimal ``fractional part", the induced structure on $[0,\alpha)$.
\newline

\noindent Theorem D shows that a strongly dependent expansions of $(\R,<,+)$ by closed sets is almost o-minimal.
In Section~\ref{section:generic-local} we apply the general theory of expansions of $(\R,<,+)$ to show that generically locally o-minimal expansions of $(\R,<,+)$ enjoy many of the good properties of o-minimal structures.
Generic local o-minimality implies definable selection, generic $C^k$-smoothness of definable functions, and yields a theory of dimension for definable sets well suited to computations.
We use these tools to obtain a dichotomy between ``linear" and ``field-type" expansions.

\begin{thmE}
Suppose $\Sa R$ is a generically locally o-minimal expansion of $(\R,<,+)$.
Then the following are equivalent:
\begin{enumerate}
    \item there is a nonempty open interval $I$ and continuous definable $\oplus,\otimes : I^2 \to I$ such that $(I,<,\oplus,\otimes)$ is an ordered field isomorphic to $(\R,<,+,\times)$,
    \item there is a definable field $(Z,\oplus,\otimes)$ such that $\dim Z \geq 1$,
    \item there is a definable family $(X_a)_{a \in B}$ of one-dimensional subsets of $\R^m$ such that $\dim B \geq 2$ and $X_a \cap X_b$ is zero-dimensional for distinct $a,b \in B$,
    \item there is a definable function $f : U \to \R^n$, $U$ a definable open subset of $\R^m$, such that $f$ is not locally affine on a dense open subset of $U$,
    \item there is a definable subset $X$ of $\R^m$ such that the set of $p \in X$ such that $U \cap X = H \cap X$ for some open neighbourhood $U$ of $p$ and affine subspace $H$ of $\R^m$ is not dense in $X$.
 \end{enumerate}
\end{thmE}

\noindent There are locally o-minimal expansions of $(\R,<,+)$ which define infinite fields and do not satisfy $(1)$ above, see Fact~\ref{fact:ktt-converse}.
We conjecture that an $\nip$ expansion of $(\R,<,+)$ by closed sets which defines an infinite field satisfies $(1)$, see Conjecture~\ref{conj:field}.
We prove several special cases of this conjecture, in particular the strongly dependent case.\newline

\noindent 
The straightforward analogues of Theorems C and D fail for archimedean structures.
In Section~\ref{section:completion} we treat the correct analogue.
Suppose $\Sa R$ expands a dense archimedean ordered abelian group $(R,<,+)$, which we take to be a substructure of $(\R,<,+)$.
Laskowski and Steinhorn~\cite{LasStein} show that if $\Sa R$ is o-minimal then $\Sa R$ is an elementary substructure of a unique o-minimal expansion $\Sa S$ of $(\R,<,+)$.
In this case the structure induced on $R$ by $\Sa S$ is interdefinable with the Shelah expansion $\Sh R$ of $\Sa R$.
It follows from work of Wencel~\cite{Wencel-1} that if $\Sa R$ is weakly o-minimal then there is an o-minimal expansion $\Sq R$ of $(\R,<,+)$ such that the structure induced on $R$ by $\Sq R$ is interdefinable with $\Sh R$.
Theorem F generalizes these results.

\begin{thmF}
Let $R$ be a dense subgroup of $(\R,+)$, $\Sa R$ be an $\nip$ expansion of $(R,<,+)$, and $\Sa R \prec \Sa N$ be highly saturated.
Let $\mfin$ be the convex hull of $R$ in $N$ and $\minf$ be the set of $a \in N$ such that $|a| < b$ for all $b \in R,b > 0$.
Identify $\mfin/\minf$ with $\R$ and let $\st : \mfin \to \R$ be the quotient map.
As $\mfin$ and $\minf$ are $\Sh N$-definable we regard $\R$ as an imaginary sort of $\Sh N$.
Then the following structures are interdefinable.
\begin{enumerate}
    \item The structure $\Sq R$ on $\R$ with an $n$-ary relation symbol defining the closure in $\R^n$ of every subset of $R^n$ which is externally definable in $\Sa R$.
    \item The structure on $\R$ with an $n$-ary relation defining $\st(\mfin^n \cap X)$ for every $\Sa N$-definable $X \subseteq N^n$.
    \item The open core of the structure induced on $\R$ by $\Sh N$.
\end{enumerate}
Furthermore $\Sq R$ is generically locally o-minimal and if $\Sa R$ is strongly dependent then $\Sq R$ is  either o-minimal or
$(\R,<,+,\az)$-minimal for some $\alpha > 0$.
The structure induced on $R$ by $\Sq R$ is a reduct of $\Sh R$ and if $\Sa R$ is strongly dependent and noiseless then the structure induced on $R$ by $\Sq R$ eliminates quantifiers and is interdefinable with $\Sh R$.
\end{thmF}

\noindent The final two sections are somewhat speculative.
In Section~\ref{section:pillay} we discuss some questions arising out of work on definable groups in o-minimal structures.
In Section~\ref{section:modular} we discuss how our results could relate to a possible notion of modularity or one-basedness for $\nip$ structures.
(It is a known open problem to define a good notion of modularity for $\nip$ structures.)
We also give an example, perhaps suprising, of an $\nip$ structure $\Sa M$ such that $\Sa M$ does not interpret an infinite field but the Shelah expansion $\Sh M$ of $\Sa M$ interprets $(\R,<,+,\times)$.

\section{notation and conventions}
\label{section:conventions}
\noindent
Throughout $m,n,k,d$ are natural numbers, $i,j$ are integers, and $s,t,\lambda,\alpha$ are real numbers.
We let $\R_{>0}$ be the set of positive real numbers.
We consider $\R^0$ to be a singleton.
Given a subset $X$ of $A \times B$ and $a \in A$ we let $X_a$ be $\{ b \in B : (a,b) \in X \}$.
We let $\gr(f) \subseteq A \times B$ be the graph of a function $f : A \to B$. \newline

\noindent We say that a family of $\Cal C$ of sets is \textbf{subdefinable} if it is a subfamily of a definable family of sets.
A subdefinable basis for a topology on a definable set $X$ is a subdefinable family of sets forming a basis for a topology on $X$. \newline

\noindent We let $\rvec$ be the ordered vector space $(\R,<,+,(t \mapsto \lambda t)_{\lambda \in \R})$ of real numbers. \newline

\noindent We let $\cl(X)$ be the closure and $\bd(X)$ be the boundary of a subset $X$ of a topological space. \newline

\noindent A \textbf{Cantor subset} of $\R$ is a nowhere dense compact subset of $\R$ without isolated points, equivalently, a subset of $\R$ which is homeomorphic to the classical middle thirds Cantor set. \newline

\noindent We let $\dim X$ be the topological dimension of a subset $X$ of $\R^n$ and ``dimension" without modification always means ``topological dimension".
There are several notions of topological dimension which need not agree on general topological spaces but do agree on separable metric spaces.
So in particular there is a canonical notion of topological dimension for subsets of Euclidean space.
We refer to Engelking~\cite{Engelking} for definitions and results on topological dimension.
Recall that $X \subseteq \R^n$ is zero-dimensional if and only if $X$ is nonempty and totally disconnected.
By convention $\dim X = -1$ if and only if $X$ is empty. \newline

\noindent We will sometimes work in a multi-sorted setting.
Suppose $L$ is a language with $S$ the set of sorts and $\Sa M$ is an 
$L$-structure. 
Then we let $M$ denote the $S$-indexed family $(M_s)_{s \in S}$ of underlying sets of the sorts of $\Sa M$. 
If $x = (x_j)_{j\in J}$ is a tuple of variables, we let $M^x = \prod_{j\in J} M_{s(x_j)} $ where $M_{s(x_j)}$ is the sort of the variable $x_j$. 
If $\varphi(x,y)$ is an $L$-formula and $b \in M^y$, we let $\varphi(M^y, b)$ be the set defined by $\varphi(x,b)$. \newline

\noindent Suppose $\Sa N$ is a highly saturated structure.
A subset $X$ of $N^x$ is $\bigwedge$-definable if it is the intersection of a small collection of $\Sa N$-definable sets.
Such sets are often said to be ``type-definable".

\section{Background on Expansions of $(\R,<,+)$}
\noindent \textbf{Throughout this section $\Sa R$ is an expansion of $(\R,<,+)$ and ``definable" without modification means ``$\Sa R$-definable"}.
In Section~\ref{section:examples} we describe some important examples of expansions of $(\R,<,+)$.
In Section~\ref{section:main-conjecture} we discuss the core conjectures on expansions of $(\R,<,+)$.
In Section~\ref{section:equiv} we give various equivalent definitions of local and generic local o-minimality.
In particular we show that an expansion of $(\R,<,+)$ is generically locally o-minimal if and only if it is strongly noiseless.
These results are well-known to experts.
In Section~\ref{section:d-min} we give some background on d-minimal expansions.
These results are already essentially known to experts, but our approach is new.
In particular we introduce the ``Pillay rank".\newline

\noindent The general study of expansions of $(\R,<,+)$ was envisioned by Chris Miller.

\subsection{Examples of expansions}
\label{section:examples}
\noindent The basic example of an expansion which is locally o-minimal and not o-minimal is $(\R,<,+,\Z)$.
(Miller~\cite{ivp} and Weispfenning~\cite{weis} independently showed that $(\R,<,+,\Z)$ has quantifier elimination in a natural expanded language, local o-minimality follows.)
Marker and Steinhorn showed in unpublished work that $(\R,<,+,\sin)$ is locally o-minimal, see \cite{TV-local}.
Fact~\ref{fact:ktt-converse} generalizes both of these examples.
Fact~\ref{fact:ktt-converse} is due to Kawakami, Takeuchi, Tanaka, and Tsuboi, it is a special case of \cite[Theorem 18]{KTTT}.
Given a real number $\alpha > 0$ we let $+_\alpha : [0,\alpha)^2 \to [0,\alpha)$ be given by $t +_\alpha t' = t + t'$ when $t + t' < \alpha$ and $t +_\alpha t' = t + t' - \alpha$ otherwise.

\begin{fact}
\label{fact:ktt-converse}
Fix $\alpha > 0$.
Suppose $\Sa I$ is an o-minimal expansion of $([0,\alpha),<,+_\alpha)$ and $\Sa D$ is an arbitrary first order expansion of $(\az,<,+)$.
Then there is a first order expansion $\Sa S$ of $(\R,<,+)$ such that a subset of $\R^n$ is $\Sa S$-definable if and only if it is a finite union of sets of the form 
$$ \bigcup_{b \in B} b + A. $$
for $\Sa I$-definable $A \subseteq [0,\alpha)^n$ and $\Sa D$-definable $B \subseteq (\az)^n$.
This $\Sa S$ is locally o-minimal and
is bi-interpretable with the disjoint union of $\Sa I$ and $\Sa D$.
\end{fact}

\noindent If $\alpha = 1$, $\Sa I$ is $([0,1),<,+_1)$, and $\Sa D$ is $(\Z,<,+)$, then $\Sa S$ is interdefinable with $(\R,<,+,\Z)$.
It follows that any $(\R,<,+,\Z)$-definable subset of $\R^n$ is a finite union of sets of the the form $\bigcup_{b \in B} b + A$ for $(\R,<,+)$-definable $A \subseteq [0,1)^n$ and $(\Z,<,+)$-definable $B \subseteq \Z^n$.
This description of definable sets previously appeared in work by computer scientists \cite{BBL-semenov, BFL-decomp}.
\newline

\noindent
If $\alpha = \pi$, $\Sa I$ is $([0,\pi), <, +_\pi, \sin|_{[0,\pi)})$, and $\Sa D$ is $(\pi\Z,<,+)$, then $\Sa S$ is interdefinable with $(\R,<,+,\sin)$.
More generally, suppose that $f : \R \to \R^n$ is analytic and $\alpha > 0$ is such that $f(t + \alpha) = f(t)$ for all $t$.
Letting $\Sa I$ be $([0,\alpha),<,+_\alpha, f|_{[0,\alpha)})$ and $\Sa D$ be $(\az,<,+)$, we easily see that $(\R,<,+,f)$ is interdefinable with $\Sa S$.
O-minimality of $\R_{\mathrm{an}}$ (see \cite{vdd-sub}) implies $\Sa I$ is o-minimal, so $(\R,<,+,f)$ is locally o-minimal.
\newline

\noindent Fact~\ref{fact:ktt-converse} shows in particular that if $\Sa D$ is an $\nip$ expansion of $(\Z,<,+)$ then the expansion of $(\R,<,+,\Z)$ by all $\Sa D$-definable subsets of all $\Z^n$ is $\nip$ and locally o-minimal.
Note that this structure is generated by closed sets.
If $\Sa D$ is $(\az,<,+)$ then $\Sa S$ is strongly dependent as a disjoint union of two strongly dependent structures is strongly dependent.\newline

\noindent Fact~\ref{fact:ktt-converse} also shows that the expansion of $(\R,<,+)$ by \textit{all} subsets of \textit{all} $\Z^n$ is locally o-minimal.
This is a special case of a general phenomenon uncovered by Miller and Friedman~\cite{FM-Sparse}: geometric tameness properties for definable sets weaker then o-minimality do not imply \textit{any} model-theoretic tameness. \newline

\noindent It is easy to see that an expansion of $(\R,<,+,\times)$ is locally o-minimal if and only if it is o-minimal.
The basic example of an $\nip$ expansion by closed sets which is not o-minimal is $(\R,<,+,\times,\lambda^{\Z})$ where $\lambda$ is a positive real other than $1$ and 
$$ \lambda^\Z := \{ \lambda^i : i \in \Z \}. $$
This structure was first studied by van den Dries~\cite{vdd-Powers2}.
It is $\nip$ by \cite[Theorem 6.5]{GH-Dependent}.
Fact~\ref{fact:increasing-sequence} follows by combining a theorem of Hieronymi and Miller~\cite{HM} with a metric result of Garc\'{\i}a, Hare, and Mendivil~\cite[Proposition 3.1]{GHM} (see also Fraser and Yu \cite[Theorem 6.1]{FraserYu}).

\begin{fact}
\label{fact:increasing-sequence}
Suppose $(s_n)_{n \in \N}$ is an eventually increasing sequence of positive real numbers such that $(s_{n+1} - s_n)_{n \in \N}$ is also eventually increasing.
Let $S$ be $\{ s_n : n \in \N \}$.
If $(\R,<,+,\times,S)$ does not define the set of integers then there is $\lambda > 1$ such that $s_n \geq \lambda^n$ for sufficiently large $n$.
\end{fact}

\noindent The assumption that $(s_{n+1} - s_n)_{n \in \N}$ is eventually increasing is necessary by Thamrongthanyalak~\cite{Athipat}.
It is a theorem of Miller and Speissegger that an expansion of $(\R,<,+,\times)$ by closed sets is either o-minimal or defines an infinite closed and discrete subset of $\R$, see Fact~\ref{fact:miller-speissegger}.
So $(\R,<,+,\times,\lambda^\Z)$ is arguably the most natural example of a non o-minimal tame expansion of $(\R,<,+,\times)$ by a closed set. \newline

\noindent The expansion $(\R,<,+,\times,\lambda^\Z)$ is a special case of an interesting family of expansions.
Let $\Sa S$ be an o-minimal expansion of $(\R,<,+,\times)$.
We say that $\Sa S$ has \textbf{rational exponents} if the function $\R \to \R$ given by $t \mapsto t^r$ is only definable when $r \in \Q$.
Hieronymi~\cite{discrete} showed that $(\R,<,+,\times,\lambda^\Z,\eta^\Z)$ defines the set of integers for any $\lambda,\eta > 1$ such that $\log_{\lambda}\eta \notin \Q$. 
So if $\Sa S$ does not have rational exponents then $(\Sa S, \lambda^\Z)$ defines the set of integers.
Miller and Speissegger showed that if $\Sa S$ has rational exponents then $(\Sa S,\lambda^\Z)$ admits quantifier elimination in a natural expanded language~\cite[Section 8.6]{Miller-tame}.
Tychonievich studied $(\Sa S, \lambda^\Z)$ in his thesis~\cite{Tychon-thesis}.
It follows from \cite[Theorem 4.1.2, Corollary 4.1.7]{Tychon-thesis} and work of Chernikov and Simon~\cite[Corollary 2.6]{CS-I} that $(\Sa S,\lambda^\Z)$ is $\nip$ when $\Sa S$ has rational exponents.
Let $\mathbf{e}$ and $\mathbf{s}$ be the restrictions of the exponential function and $\sin$ to $[0,2\pi]$, respectively.
Then $(\R,<,+,\times,\mathbf{e},\mathbf{s})$ is o-minimal and has rational exponents by van den Dries~\cite{vdd-sub}.
It is observed in \cite[Section 3.4]{Miller-tame} that $(\R,<,+,\times,\textbf{e},\textbf{s},\lambda^\Z)$ defines the logarithmic spiral
$$  \{ (e^t \sin(\alpha t), e^t \cos(\alpha t) ) : t \in \mathbb{R} \} $$
where $\lambda = e^{2\pi\alpha}$.
So the expansion of $(\R,<,+,\times)$ by a logarithmic spiral is an $\nip$ expansion by a closed set which is not locally o-minimal. \newline

\noindent It should be noted that $(\Sa S, \lambda^\Z)$ satisfies a stronger condition then generic local o-minimality.
If $\Sa S$ has rational exponents then $(\Sa S, \lambda^\Z)$ is \textbf{d-minimal}: every unary definable set in every model of the theory of $(\Sa S,\lambda^\Z)$ is a union of an open set together with finitely many discrete sets.
In particular every nowhere dense definable subset of $\R$ has finite Cantor rank.
Question~\ref{qst:cantor-rank} is open, we expect it to have a positive answer.
(However, we expect that natural examples of $\nip$ expansions of $(\R,<,+)$ have d-minimal open core.)

\begin{qst}
\label{qst:cantor-rank}
Is there an $\nip$ expansion of $(\R,<,+)$ which defines a nowhere dense subset of $\R$ with infinite Cantor rank?
Is there an $\nip$ expansion of $(\R,<,+)$ which defines an uncountable nowhere dense subset of $\R$?
\end{qst}

\noindent Finally, it is worth pointing out that there are well-behaved expansions of $(\R,<,+)$ which are noiseless and not strongly noiseless.
Theorem~\ref{thm:equiv} shows that a noiseless expansion is strongly noiseless if and only if it does not define a Cantor subset of $\R$.
Friedman, Kurdyka, Miller, and Speissegger~\cite{FKMS} gave the first example of a Cantor subset $K$ of $\R$ such that $(\R,<,+,\times,K)$ is noiseless.
Hieronymi~\cite{H-TameCantor} shows that if $\Sa S$ is an o-minimal expansion of $(\R,<,+,\times)$ such that every $\Sa S$-definable function $\R \to \R$ is eventually bounded above by some compositional iterate of the exponential (every known o-minimal expansion of the real field satisfies this condition), then there is a Cantor subset $K$ of $\R$ such that $(\Sa S,K)$ is model-theoretically tame and noiseless.

\subsection{The core conjectures}
\label{section:main-conjecture} A subset of $\R^k$ is \textbf{$\omega$-orderable} if it is definable and is either finite or admits a definable ordering with order type $\omega$.
One should think of ``$\omega$-orderable" as ``definably countable".
A \textbf{dense $\omega$-order} is an $\omega$-orderable subset of $\R$ which is dense in some nonempty open interval.
The presence or absence of a dense $\omega$-order has remarkably strong consequences.
Hieronymi's Theorem is equivalent to the main theorem of \cite{discrete}.
The left to right implication follows as any countable subset of $\R^n$ is $\omega$-orderable in $(\R,<,+,\times,\Z)$.

\begin{thmH}
Suppose $\Sa R$ expands $(\R,<,+,\times)$.
Then there is a dense $\omega$-order if and only if the set $\Z$ of integers is definable.
\end{thmH}

\noindent
Fact~\ref{fact:HW} is proven in Hieronymi and Walsberg~\cite{HW-continuous}.

\begin{fact}
\label{fact:HW}
Suppose that $\Sa R$ defines at least one of the following:
\begin{enumerate}
    \item an unbounded continuous function $I \to \R^n$ on a bounded interval $I$,
    \item a non-affine $C^2$-function $U \to \R^n$ on a connected open subset $U$ of $\R^m$,
    \item the function $[0,1] \to \R$ given by $x \mapsto \lambda x$ for uncountably many $\lambda \in \R$.
\end{enumerate}
Then there is a dense $\omega$-order if and only if every bounded Borel subset of every $\R^n$ is definable.
\end{fact}

\noindent
Fact~\ref{fact:HW2} is due to Hieronymi and Walsberg~\cite{HW-Monadic}.
We consider the two-sorted first order structure $(\Cal P(\N),\N,\in,s)$ where $\Cal P(\N)$ is the power set of $\N$ and $s : \N \to \N$ is the usual successor function.
We identify the (first order) theory of $(\Cal P(\N),\N,\in,s)$ with the monadic second order theory of $(\N,s)$.

\begin{fact}
\label{fact:HW2}
If $\Sa R$ admits a dense $\omega$-order then $\Sa R$ defines an isomorphic copy of $(\Cal P(\N),\N,\in,s)$.
\end{fact}

\noindent Fact~\ref{fact:HW2} is sharp in that $(\Cal P(\N),\N,\in,s)$ defines isomorphic copies of first order expansions of $(\R,<,+)$ which admit dense $\omega$-orders.
Recall that the theory of $(\Cal P(\N),\N,\in,s)$ is decidable by a theorem of B{\"u}chi~\cite{Buchi}.\newline

\noindent
We say that an expansion of $(\R,<,+)$ is \textbf{type A} if there are no dense $\omega$-orders.
Conjecture~\ref{conj:main} is the main conjecture on first order expansions of $(\R,<,+)$.

\begin{conj}
\label{conj:main}
If $\Sa R$ is type A then $\Sa R^\circ$ is noiseless.
In particular if $\Sa R$ expands $(\R,<,+,\times)$ and does not define $\Z$ then $\Sa R^\circ$ is noiseless.
\end{conj}

\noindent
We recall another theorem from \cite{HW-Monadic}.

\begin{fact}
\label{fact:HW-cantor}
If $\Sa R$ defines a Cantor subset of $\R$ then $\Sa R$ defines an isomorphic copy of $(\Cal P(\N),\N,\in,s)$.
\end{fact}

\noindent
Fact~\ref{fact:HW-cantor} is sharp as $(\Cal P(\N),\N,\in,s)$ defines an isomorphic copy of $(\R,<,+,K)$ where $K$ is the classical middle-thirds Cantor set \cite[Theorem 5]{BRW}.
In Theorem~\ref{thm:equiv} we give a proof of the fact, well known to experts, that an expansion of $(\R,<,+)$ is generically locally o-minimal if and only if it is noiseless and does not define a Cantor subset of $\R$.
So Conjecture~\ref{conj:main-B} is a special case of Conjecture~\ref{conj:main}.

\begin{conj}
\label{conj:main-B}
If $\Sa R$ does not define an isomorphic copy of $(\Cal P(\N),\N,\in,s)$ then $\Sa R^\circ$ is generically locally o-minimal.
\end{conj}

\noindent
We now describe several known cases of Conjecture~\ref{conj:main}.
The following theorem of Miller and Speissegger~\cite{MS99} is a special case of Conjecture~\ref{conj:main}.

\begin{fact}
\label{fact:miller-speissegger}
Suppose $\Sa R$ expands $(\R,<,+,\times)$.
Then $\Sa R$ does not define an infinite closed discrete subset of $\R$ if and only if $\Sa R^\circ$ is o-minimal.
\end{fact}

\noindent It follows from the proof of \cite[Theorem A]{HW-Monadic} that if $\Sa R$ admits a dense $\omega$-order then $\Sa R$ defines an infinite bounded discrete subset of $\R$.
So Fact~\ref{fact:HW-case} is another special case of Conjecture~\ref{conj:main}.
Fact~\ref{fact:HW-case} will be proven in forthcoming joint work with Hieronymi.

\begin{fact}
\label{fact:HW-case}
Every $\Sa R$-definable nowhere dense subset of $\R$ is closed and discrete if and only if $\Sa R^\circ$ is locally o-minimal.
\end{fact}

\noindent The following theorem of Fornasiero, Hieronymi, and Walsberg~\cite[Theorem D]{FHW-Compact} shows in particular that if $\Sa R$ is type A then every subset of $\R^m$ which is existentially definable in $\Sa R^\circ$ either has interior or is nowhere dense.

\begin{fact}
\label{fact:con-image}
Suppose that $\Sa R$ is type A.
Let $X$ be a definable constructible subset of $\R^n$.
If $f : X \to \R^m$ is definable and continuous then $f(X)$ either has interior or is nowhere dense.
\end{fact}

\noindent Theorem~\ref{thm:fmhw} is known to experts but has not previously been stated in full generality.
It shows that a type A expansion by zero-dimensional closed subsets of $\R^n$ and closed subsets of $\R$ is noiseless.

\begin{theorem}
\label{thm:fmhw}
Suppose $\Sa S$ is an o-minimal expansion of $(\R,<,+)$.
Let $\Cal C$ be a collection of subsets of Euclidean space such that the closure of each $X \in \Cal C$ is zero-dimensional and let $\Cal D$ be a collection of subsets of $\R$ such that each $C \in \Cal D$ is not dense and co-dense in any nonempty open interval.
If $(\Sa S,\Cal C, \Cal D)$ is type A then $(\Sa S,\Cal C, \Cal D)$ is noiseless.
\end{theorem}

\noindent
Theorem~\ref{thm:fmhw} is an application of the next two facts.
The first, a theorem of Friedman and Miller~\cite{FM-Sparse}, requires some notation.
If $\Sa S$ is a first order expansion of $(\R,<,+)$ and $E$ is a subset of $\R$ then $(\Sa S, E)^\sharp$ is the expansion of $\Sa S$ by \textit{all} subsets of \textit{all} cartesian powers of $E$.
Note that if $E$ is infinite then $(\Sa S, E)^\sharp$ defines an isomorphic copy of $(\Z,+,\times)$ and if $E$ is uncountable then $(\Sa S, E)^\sharp$ defines an isomorphic copy of the standard model of second order arthimetic.
So Fact~\ref{fact:FM} may be a surprise.

\begin{fact}
\label{fact:FM}
Let $\Sa S$ be an o-minimal expansion of $(\R,<,+)$ and $E$ be a subset of $\R$.
If $f(E^n)$ is nowhere dense for every $\Sa S$-definable $f : \R^n \to \R$ then $(\Sa S, E)^\sharp$ is noiseless.
\end{fact}

\noindent
Miller and Friedman only show that $(\Sa S,E)^\sharp$ is noiseless in one variable.
Fact~\ref{fact:noiseless} below shows that $(\Sa S, E)^\sharp$ is noiseless.
Fact~\ref{fact:HW-dimension} follows from \cite[Proposition 5.6]{FHW-Compact}.

\begin{fact}
\label{fact:HW-dimension}
Suppose $\Sa R$ is type A.
Let $X$ be a constructible definable subset of $\R^n$ and $f : X \to \R^m$ be continuous and definable.
If $X$ is zero-dimensional then $f(X)$ is nowhere dense.
\end{fact}

\noindent
We may now prove Theorem~\ref{thm:fmhw}.

\begin{proof}
Given $1 \leq k \leq n$ we let $\pi^n_k : \R^n \to \R$ be the map given by
$$ \pi^n_k(x_1,\ldots,x_n) =  x_k.$$
We let $\Sa R$ be $(\Sa S, \Cal C, \Cal D)$.
We suppose that $\Sa R$ is type A.
Note that $\Sa R$ is noiseless if and only if $(\Sa S, \Cal C', \Cal D')$ is noiseless for any finite subcollection $\Cal C'$ of $\Cal C$ and finite subcollection $\Cal D'$ of $\Cal D$.
So we assume $\Cal C$ and $\Cal D$ are finite.
Suppose that $\Sa R$ is type A.
We now define an auxiliary collection $\Cal E$ of subsets of $\R$ consisting of:
\begin{itemize}
    \item $\pi^n_k( \cl(X) )$ for all $1 \leq k \leq n$ and $X \in \Cal C$ such that $X \subseteq \R^n$,
    \item $\bd(X)$ for all $X \in \Cal D$.
\end{itemize}
If $X \in \Cal D$ then $\bd(X)$ is nowhere dense as $X$ is not dense and co-dense in any open interval.
If $X \in \Cal C$, then $\cl(X)$ is zero-dimensional by assumption, so each $\pi^n_k(\cl(X))$ is nowhere dense by Fact~\ref{fact:HW-dimension}.
So every element of $\Cal E$ is nowhere dense.
Note also that $\Cal E$ is finite as $\Cal C$ and $\Cal D$ are finite, so $\bigcup \Cal E$ is nowhere dense.
Let $E$ be the closure of $\bigcup \Cal E$, so $E$ is closed, nowhere dense, and definable.\newline

\noindent We show that $E$ satisfies the condition of Fact~\ref{fact:FM}.
Let $f : \R^n \to \R$ be $\Sa S$-definable.
We show that $f(E^n)$ is nowhere dense.
Applying o-minimal cell decomposition we obtain a finite partition $\Cal F$ of $\R^n$ into $\Sa S$-definable cells such that the restriction of $f$ to each element of $\Cal F$ is continuous.
Fix $Z \in \Cal F$.
As $E$ is closed and zero-dimensional, $E^n$ is also closed and zero-dimensional.
So $Z \cap E^n$ is zero-dimensional.
It easily follows from the definition of a cell that any cell is constructible, so $Z$ is constructible.
So $Z \cap E^n$ is constructible.
An application of Fact~\ref{fact:HW-dimension} now shows that $f(Z \cap E^n)$ is nowhere dense.
So $f(E^n)$ is a finite union of nowhere dense sets and is hence nowhere dense.
Fact~\ref{fact:FM} shows that $(\Sa S, E)^\sharp$ is noiseless.\newline

\noindent It now suffices to show that every element of $\Cal C$ and $\Cal D$ are $(\Sa S,E)^\sharp$-definable.
Every element of $\Cal E$ is a subset of $E$ and is thus $(\Sa S,E)^\sharp$-definable.
Fix $X \in \Cal C$ and suppose $X \subseteq \R^n$.
Then
$$ X \subseteq \prod_{k = 1}^{n} \pi^n_k(\cl(X)) \subseteq E^n $$
so $X$ is $(\Sa S, E)^\sharp$-definable.
Now suppose $X \in \Cal D$.
Let $U$ be the interior of $X$.
Then
$$ X \setminus U \subseteq \bd(X) \subseteq E $$
so $X \setminus U$ is $(\Sa S, E)^\sharp$-definable.
We show that $U$ is $(\Sa S, E)^\sharp$-definable.
Let $A$ be the set of $(a,b) \in \bd(X)^2$ such that $a < b$ and the open interval with endpoints $a,b$ is contained in $U$.
Then $A \subseteq E^2$ and so $A$ is $(\Sa S, E)^\sharp$-definable.
Finally
$$ U = \bigcup_{(a,b) \in A} \{ t \in \R : a < t < b\} $$
so $U$ is $(\Sa S,E)^\sharp$-definable.
\end{proof}

\subsection{Local and generic local o-minimality}
\label{section:equiv}
\noindent Theorem~\ref{thm:equiv-1} gives several equivalent definitions of local o-minimality for expansions of $(\R,<,+)$.
Locally o-minimal structures are studied in \cite{TV-local,KTTT}.
In particular the equivalence of $(2)$ and $(3)$ below was proven in \cite{TV-local}.

\begin{theorem} \label{thm:equiv-1}
The following are equivalent:
\begin{enumerate}
    \item $\Sa R$ is locally o-minimal,
    \item for every definable $X \subseteq \R$ and $a \in \R$ there is an open interval $I$ containing $a$ such that $I \cap X$ is $(\R,<)$-definable,
    \item the structure induced by $\Sa R$ on any bounded interval $I$ is o-minimal,
    \item the expansion of $(\R,<,+)$ by all bounded $\Sa R$-definable sets is o-minimal,   
    \item $\Sa R$ is noiseless and does not define a bounded discrete subset of $D$ of $\R_{>0}$ such that $\cl(D) = D \cup \{0\}$,
    \item $\Sa R$ is noiseless and every nowhere dense definable subset of $\R$ is closed and discrete.
\end{enumerate}
\end{theorem}

\noindent We will need four results for the proof of Theorem~\ref{thm:equiv-1}.
The first is an easy fact about subsets of $\R$ whose verification we leave to the reader.

\begin{fact}
\label{fact:boundary}
Suppose $I$ is an open interval and $X$ is a subset of $I$.
Then $X$ is a finite union of open intervals and singletons if and only if $\bd(X) \cap I$ is finite.
\end{fact}

\noindent We say that $\Sa R$ is \textbf{o-minimal at infinity} if for every definable $X \subseteq \R$ there is $t > 0$ such that $(t,\infty)$ is either contained in or disjoint from $X$.
Fact~\ref{fact:at-infinity} is a special case of a theorem of Belegradek, Verbovskiy, and Wagner~\cite[Theorem 19]{coset-minimal}.

\begin{fact}
\label{fact:at-infinity}
The expansion of $(\R,<,+)$ by all bounded subsets of all $\R^n$ is o-minimal at infinity.
\end{fact}

\noindent Fact~\ref{fact:noiseless} is due to Miller~\cite[Theorem 3.2]{Miller-tame}.

\begin{fact}
\label{fact:noiseless}
$\Sa R$ is noiseless if and only if it is noiseless is one variable.
\end{fact}



\noindent We now prove Theorem~\ref{thm:equiv-1}.

\begin{proof}
\noindent $(2) \Rightarrow (3):$ Let $I$ be a bounded interval.
After replacing $I$ with its closure if necessary we suppose $I$ is closed.
Fix a definable subset $X$ of $I$.
We show that $X$ is definable in $(\R,<)$.
For every $a \in I$ let $J_a$ be an open interval containing $a$ such that $J_a \cap X$ is definable in $(\R,<)$.
As $I$ is compact there is a finite $A \subseteq I$ such that $(J_a)_{a \in A}$ covers $I$.
Then 
$$ X = \bigcup_{a \in A} J_a \cap X $$
so $X$ is definable in $(\R,<)$. \newline

\noindent $(3) \Rightarrow (4):$
Let $\Cal B$ be the collection of all bounded definable sets.
Suppose $X \subseteq \R$ is definable in $(\R,<,+,\Cal B)$.
Fact~\ref{fact:at-infinity} yields a $t > 0$ such that $X \setminus [-t,t]$ is definable in $(\R,<)$.
The induced structure on $[-t,t]$ is o-minimal, so  $X \cap [-t,t]$ is definable in $(\R,<)$.
So $X$ is $(\R,<)$-definable. \newline

\noindent $(4) \Rightarrow (5):$ Suppose that $\Sa R$ is noisey.
Let $X$ be a definable subset of $\R^n$ which is dense and co-dense in a nonempty definable open subset $U$ of $\R^n$.
We may suppose that $U$ is bounded.
Then $U \cap X$ is bounded and definable and $(\R,<,+,U \cap X)$ is not o-minimal, contradiction.
Suppose that $D$ is a bounded discrete subset of $\R_{>0}$ such that $\cl(D) = D \cup \{0\}$.
Then $\bd(D) = D \cup \{0\}$ is infinite, an application of Fact~\ref{fact:boundary} shows that $(\R,<,+,D)$ is not o-minimal. \newline



\noindent $(5) \Rightarrow (6):$ Suppose that $X$ is a nowhere dense definable subset of $\R$ which is not closed and discrete.
Fix a bounded open interval $I$ such that $I \cap X$ is infinite.
Note that each connected component of $I \setminus \cl(X)$ is a nonempty open interval.
We let $D$ be the set of lengths of connected components of $I \setminus \cl(X)$.
Then $D$ is infinite as $\cl(X)$ is infinite and bounded as $I$ is bounded.
As $I$ is bounded there are only finitely many connected components of length $> t$ for any $t > 0$.
Thus $D \setminus (0,t)$ is finite for any $t > 0$.
It follows that $D$ is discrete and $\cl(D) = D \cup \{0\}$.  \newline

\noindent $(6) \Rightarrow (2):$ Fix a definable subset $X$ of $\R$ and $a \in X$.
Then $\bd(X)$ is nowhere dense as $X$ is nowhere dense and co-dense. 
So $\bd(X)$ is closed and discrete. 
Let $I$ be a bounded open interval containing $a$ such that $I \cap \bd(X)$ is finite.
Fact~\ref{fact:boundary} shows that $I \cap X$ is definable in $(\R,<)$. \newline

\noindent It is immediate that $(2)$ implies $(1)$.
We show that $(1)$ implies $(6)$.
Let $X \subseteq \R$ be definable and somewhere dense and $I$ be a nonempty open interval in which $X$ is dense.
Fix $p \in I \cap X$ and a subinterval $J$ of $I$ such that $J \cap X$ is $(\R,<)$-definable.
As $X$ is dense in $J$, $J \cap X$ is cofinite, so $X$ has interior.
Now suppose that $Y \subseteq \R$ is nowhere dense and definable.
It suffices to show that every $p \in \cl(X)$ is isolated.
Let $I$ be a open interval containing $p$ such that $I \cap \cl(X)$ is $(\R,<)$-definable.
As $\cl(X)$ is nowhere dense, $I \cap \cl(X)$ is finite hence $p$ is isolated in $I \cap \cl(X)$.
\end{proof}

\noindent Theorem~\ref{thm:equiv} gives conditions that  are equivalent to generic local o-minimality.

\begin{theorem}
\label{thm:equiv}
The following are equivalent:
\begin{enumerate}
    \item $\Sa R$ is generically locally o-minimal,
    \item every definable subset of $\R$ either has interior or contains an isolated point,
    \item $\Sa R$ is strongly noiseless,
    \item the structure induced on $[0,1]$ by $\Sa R$ is strongly noiseless,
    \item $\Sa R$ is strongly noiseless in one variable,
    \item $\Sa R$ is noiseless in one variable and does not define a Cantor subset of $\R$.
\end{enumerate}
\end{theorem}

\noindent We need another result of Miller.
Fact~\ref{fact:c0-smooth} follows from \cite[Proposition 3.4]{Miller-tame}.

\begin{fact}
\label{fact:c0-smooth}
Suppose every $\Sa R$-definable subset of $\R$ either has interior or contains an isolated point.
Let $X$ be a nonempty $\Sa R$-definable subset of $\R^n$.
Then there is a dense open subset $V$ of $X$ such that for every $p \in V$ there is $0 \leq d \leq n$, a coordinate projection $\pi : \R^n \to \R^d$, and an open neighbourhood $W$ of $p$ such that $\pi(X \cap W)$ is open and $\pi$ induces a homeomorphism $X \cap W \to \pi(X \cap W)$.
\end{fact}

\noindent We now prove Theorem~\ref{thm:equiv}.

\begin{proof}
We first show that $(1)$ and $(2)$ are equivalent.
Let $X \subseteq \R$ be definable.
Suppose $\Sa R$ is generically locally o-minimal.
Then there is a $p \in X$ and an open interval $I$ containing $p$ such that $I \cap X$ is $(\R,<)$ definable.
So $I \cap X$ either has interior or is finite and hence contains an isolated point.
Suppose $(2)$ holds.
Let $U$ be the interior of $X$ and $D$ be the set of isolated points of $X$.
Note that $U \cup D$ is open in $X$.
It suffices to show that $U \cap D$ is dense in $X$.
Suppose otherwise.
Then there is $p \in X$ and an open interval $I$ containing $p$ such that $I$ is disjoint from $U \cap D$.
Then $I \cap X$ does not have interior and does not contain an isolated point, contradiction.\newline

\noindent $(2) \Rightarrow (3)$:
Note that $(2)$ implies that $\Sa R$ is noiseless in one variable, so $\Sa R$ is noiseless by Fact~\ref{fact:noiseless}.
Let $X,Y$ be definable subsets of $\R^n$.
Suppose that $X$ is somewhere dense in $Y$.
Fix a nonempty definable open subset $W$ of $Y$ such that $X$ is dense in $W \cap Y$.
After replacing $Y$ with $W \cap Y$ if necessary we suppose $X$ is dense in $Y$.
Applying Fact~\ref{fact:c0-smooth} we obtain $0 \leq d \leq n$, definable open $U \subseteq \R^n$, $V \subseteq \R^d$, and a coordinate projection $\pi : \R^n \to \R^d$ such that $\pi$ restricts to a homeomorphism $U \cap Y \to V$.
Then $\pi(X \cap U)$ is dense in $V$.
So $\pi(X \cap U)$ has interior in $V$ as $\Sa R$ is noiseless.
It follows that $X$ has interior in $U \cap Y$ and thus has interior in $Y$.\newline

\noindent It is clear that $(3)$ implies $(4)$.
We show that $(4)$ implies $(5)$.
Suppose $\Sa R$ is not strongly noiseless in one variable.
Let $X,Y$ be definable subsets of $\R$ such that $X$ is somewhere dense in $Y$ and has empty interior in $Y$.
Let $p$ be an element of $X$ and $I$ be an open interval containing $p$ such that $I \cap X$ is dense in $I \cap Y$.
Then $(X - p + \frac{1}{2}) \cap [0,1]$ is somewhere dense in $(Y - p + \frac{1}{2}) \cap [0,1]$ but $(X - p + \frac{1}{2}) \cap [0,1]$ has empty interior in $( Y - p + \frac{1}{2} ) \cap [0,1]$.
So the structure induced on $[0,1]$ by $\Sa R$ is not strongly noiseless. \newline

\noindent $(5) \Rightarrow (6):$ It suffices to suppose there is a definable Cantor set and show that $(5)$ fails.
Let $Y$ be a definable Cantor set.
Observe that each connected component of $\R \setminus Y$ is an interval.
Let $X$ be the set of endpoints of connected components of $\R \setminus Y$.
Note $X$ is a definable subset of $Y$.
It is easy to see that $X$ is dense and co-dense in $Y$.\newline

\noindent $(6) \Rightarrow (2):$
Let $X$ be a definable subset of $\R$.
We show that $X$ either has interior or contains an isolated point.
If $X$ is somewhere dense then $X$ has interior.
Suppose $X$ is nowhere dense.
It suffices to show that $\cl(X)$ has an isolated point.
After replacing $X$ by its closure we suppose $X$ is closed nowhere dense.
As $X$ is totally disconnected there is a bounded open interval $I$ such that $I \cap X \neq \emptyset$ and $\bd(I) \cap X = \emptyset$.
So $I \cap X$ is nonempty, closed, bounded, and nowhere dense.
As $I \cap X$ is not a Cantor set, it has an isolated point.
\end{proof}

\subsection{D-minimality and Pillay rank}
\label{section:d-min}
\noindent 
At several points we will need to assume d-minimality.
We introduce a useful tool in this setting, the Pillasy rank.
We say that $\Sa R$ satisfies the \textbf{closed chain condition} if there is no sequence $(X_k)_{k \in \N}$ of nonempty definable subsets of some $\R^n$ such that $X_{k+1}$ is a nowhere dense subset of $X_k$ for all $k$.
It is easy to see that $\Sa R$ fails the closed chain condition if and only if there is a sequence $\{ X_k \}_{k \in \N}$ of nonempty closed definable subsets of some $\R^n$ such that $X_{k+1}$ has empty interior in $X_k$ for all $k$.

\begin{thm}
\label{thm:d-min}
Suppose $\Sa R$ is d-minimal.
Then $\Sa R$ satisfies the closed chain condition.
\end{thm}

\noindent Theorem~\ref{thm:d-min} requires Fact~\ref{fact:cb} which is elementary and left to the reader.
We let $\mathrm{Cb}(X)$ be the Cantor-Bendixson rank of a subset $X$ of $\R$.

\begin{fact}
\label{fact:cb}
Suppose $X$ is a nonempty subset of $\R$ with finite Cantor-Bendixson rank and $Y$ is a subset of $X$.
If $\mathrm{Cb}(X) = \mathrm{Cb}(Y)$ then $Y$ has interior in $X$.
\end{fact}

\noindent We introduce some notation for the proof of Theorem~\ref{thm:d-min}.
Let $Y$ be a definable subset of $\R$.
We let:
\begin{enumerate}
    \item $\Omega(Y) = -1$ when $Y = \emptyset$,
    \item $\Omega(Y) = \mathrm{Cb}(Y)$ when $Y$ is nonempty and nowhere dense, and
    \item $\Omega(Y) = \omega$ when $Y$ has interior.
\end{enumerate}
If $Z \subseteq \R^n$ is definable then by d-minimality $\Omega(X_a)$ takes only finitely many values as $a$ ranges over $\R^{n-1}$ and $\{ a \in \R^{n-1} : \Omega(X_a) = \eta \}$ is definable for all $\eta \in \N \cup \{-1,\omega \}$.\newline

\noindent Theorem~\ref{thm:d-min} also requires Lemma~\ref{lem:k-w} below.
Lemma~\ref{lem:k-w} holds more generally for generically locally o-minimal expansions of $(\R,<,+)$, so we prove it in Section~\ref{section:generic-local}.
We now prove Theorem~\ref{thm:d-min}.

\begin{proof}
Suppose towards a contradiction that $(X_k)_{k \in \N}$ is a sequence of nonempty definable subsets of $\R^n$ such that $X_{k+1}$ is a nowhere dense subset of $X_k$ for all $k$.
We apply induction on $n$. \newline

\noindent Suppose $n = 1$.
Note that $X_k$ is nowhere dense in $\R$ when $k \geq 1$.
By d-minimality $X_k$ has finite Cantor-Bendixson rank for all $k \geq 1$.
Fact~\ref{fact:cb} shows that $\mathrm{Cb}(X_{k+1}) < \mathrm{Cb}(X_k)$ for all $k \geq 1$.
Contradiction. \newline

\noindent Suppose $n \geq 2$ and let $\pi : \R^{n} \to \R^{n - 1}$ be the projection away from the first coordinate.
By induction it suffices to show that for every $k \in \N$ there is $m > k$ such that $\pi(X_m)$ is nowhere dense in $\pi(X_k)$.
For the sake of simplicity we only treat the case when $k = 0$ and set $X := X_0$.
\newline

\noindent
For all $m \geq 1$ and $\eta \in \N \cup \{-1,\omega\}$ we let
$$ B^\eta_m := \{ a \in \pi(X) : \Omega((X_m)_a) = \eta \}. $$
Note that for each $m$ there are only finitely many $\eta$ such that $B^\eta_m$ is nonempty.
For each $m \geq 1$ we let $\zeta(m)$ be the maximal $\eta \in \N \cup \{-1,\omega \}$ such that $B^\eta_m$ is somewhere dense in $\pi(X)$.
Observe that $\pi(X_m)$ is nowhere dense in $\pi(X)$ if and only if $\zeta(m) = -1$.
So it suffices to show that $\zeta(m+1) < \zeta(m)$ when $\zeta(m) \geq 0$. \newline

\noindent As $(X_m)_{m \geq 1}$ is decreasing, $(\Omega((X_m)_a)_{m \geq 1}$ is decreasing for all $a \in \pi(X)$.
It follows that $(\zeta(m))_{m \geq 1}$ is decreasing.
Fix $m$ and suppose towards a contradiction that $\zeta(m) \geq 0$ and $\zeta(m+1) = \zeta(m)$.
Applying strong noiselessness let $U$ be a nonempty definable open subset of $\pi(X)$ contained in $B^{\zeta(m)}_{m+1}$.
So $\Omega((X_{m+1})_a) = \zeta(m)$ for all $a \in U$.
Hence $\Omega((X_{m+1})_a) = \Omega((X_m)_a)$ for all $a \in U$.
Applying Fact~\ref{fact:cb} we see that $(X_{m+1})_a$ has interior in $(X_m)_a$ for all $a \in U$.
Lemma~\ref{lem:k-w} now shows that $X_{m+1}$ has interior in $X_m$.
Contradiction.
\end{proof}

\noindent Suppose that $\Sa R$ is d-minimal.
Let $\prec_n$ be the partial order on definable subsets of $\R^n$ where $X \prec_n Y$ when $X$ is a nowhere dense subset of $Y$.
Theorem~\ref{thm:d-min} shows that each $\prec_n$ is well-founded.
We refer to the ordinal rank associated to $\prec_n$ as \textbf{Pillay rank} as it was defined (in a more general setting) by Pillay~\cite{pillay-top}.
We denote the Pillay rank of a definable subset $X$ of $\R^n$ by $\pr(X)$.
If $X$ is a definable subset of $\R^n$, then:
\begin{enumerate}
    \item $\pr(X) = -1$ if and only if $X$ is empty,
    \item If $\delta$ is an ordinal or $-1$, then $\pr(X) \geq \delta + 1$ if and only if there is a nowhere dense definable $X' \subseteq X$ such that $\Pr(X') \geq \delta$,
    \item If $\delta$ is a limit ordinal then $\pr(X) \geq \delta$ if and only if $\pr(X) \geq \eta$ for all ordinals $\eta < \delta$,
    \item If $X \neq \emptyset$ then $\pr(X)$ is the supremum of all ordinals $\delta$ such that $\pr(X) \geq \delta$.
\end{enumerate}
\noindent We now give several facts about Pillay rank, still assuming that $\Sa R$ is d-minimal.
We will not make use of these facts so we do not prove them.
Suppose $Y$ is a definable subset of $\R^n$.
Then $\pr(Y) = 0$ if and only if $Y$ is discrete and nonempty.
If $\dim Y = 0$ then $\pr(Y) = \mathrm{Cb}(Y)$.
If $\Sa R$ is locally o-minimal then $\pr(Y) = \dim Y$.
If $\Sa R$ is d-minimal and not locally o-minimal then $\pr(Y) = \Omega(Y)$ when $Y \subseteq \R$.\newline

\noindent
Fact~\ref{fact:d-min-constructible} is well-known and may be proven via an easy induction on Pillay rank.

\begin{fact}
\label{fact:d-min-constructible}
If $\Sa R$ is d-minimal then every definable set is constructible.
\end{fact}

\noindent Fact~\ref{fact:d-min-cntble} is also well-known, we include a proof for the sake of completeness.

\begin{fact}
\label{fact:d-min-cntble}
Suppose $\Sa R$ is d-minimal, $X$ is a definable subset of $\R^n$, and $\dim X = 0$.
Then $X$ is countable.
\end{fact}

\begin{proof}
Given $1 \leq k \leq n$ let $\pi_k : \R^n \to \R$ be the projection onto the $k$th coordinate.
Fact~\ref{fact:HW-dimension} shows that each $\pi_k(X)$ is nowhere dense.
By d-minimality each $\pi_k(X)$ is countable.
So $X$ is countable as $X$ is a subset of $\pi_1(X) \times \ldots \times \pi_k(X)$.
\end{proof}

\section{NIP}
\noindent We recall relevant background on $\nip$ structures.
Let $\Sa M$ be a possibly multi-sorted first order structure and $\Sa N$ be a highly saturated elementary extension of $\Sa M$. \newline

\noindent We first discuss externally definable sets.
A subset $X$ of $M^x$ is \textbf{externally definable} (in $\Sa M$) if there is an $\Sa N$-definable subset $Y$ of $N^x$ such that $X = M^x \cap Y$.
It is a saturation exercise to see that the collection of externally definable sets does not depend on choice of $\Sa N$.
We first record a useful observation whose verification we leave to the reader.

\begin{fact}
\label{fact:convex}
Suppose $\Sa M$ is an expansion of a linear order.
Then every convex subset of $M$ is externally definable in $\Sa M$.
\end{fact}

\noindent
The \textbf{Shelah expansion} $\Sh M$ of $\Sa M$ is the expansion of $\Sa M$ by all externally definable subsets of all $M^x$.
If $\Sa M$ is one-sorted then $\Sh M$ is the structure induced on $M$ by $\Sa N$.
The following theorem is due to Shelah~\cite{Shelah-external}.
Chernikov and Simon~\cite[Corollary 1.10]{CS-I} give a simpler proof.

\begin{fact}
\label{fact:shelah}
Suppose $\Sa M$ is $\nip$.
Then every $\Sh M$-definable set is externally definable in $\Sa M$.
It follows that $\Sh M$ is $\nip$ when $\Sa M$ is $\nip$ and $\Sh M$ is strongly dependent when $\Sa M$ is strongly dependent.
\end{fact}

\noindent
The one-sorted case of Fact~\ref{fact:shelah} asserts that the structure induced on $M$ by $\Sa N$ eliminates quantifiers.
We record another theorem of Chernikov and Simon~\cite[Corollary 9]{CS-II}.
The right to left implication is a saturation exercise and does not require $\nip$.

\begin{fact}
\label{fact:cs}
Suppose $\Sa M$ is $\nip$.
Let $X$ be a subset of $M^x$.
Then $X$ is externally definable in $\Sa M$ if and only if there is an $\Sa M$-definable family of $(X_a)_{a \in M^y}$ of subsets of $M^x$ such that for every finite $A \subseteq X$ there is $a \in M^y$ such that $A \subseteq X_a \subseteq X$.
\end{fact}

\noindent
We will need the following final easy fact whose verification we leave to the reader.

\begin{fact}
\label{fact:extension-external}
Suppose $Y \subseteq N^x$ is externally definable in $\Sa N$.
Then $Y \cap M^x$ is externally definable in $\Sa M$.
\end{fact}

\noindent We will need to use honest definitions at one point.
Given an $\Sa M$-definable $Z \subseteq M^x$ we let $Z'$ be the subset of $N^x$ defined by any formula defining $Z$. 
Suppose $X$ is an externally definable subset of $M^x$.
Then an $\Sa N$-definable subset $Y$ of $N^x$ is an \textbf{honest definition} of $X$ if $Y \cap M^x = X$ and $Y \subseteq Z'$ for every $\Sa M$-definable $Z \subseteq M^x$ such that $X \subseteq Z$.
Taking complements, we see that if $Y$ is an honest definition of $X$ then $Y \cap Z' = \emptyset$ for every $\Sa M$-definable $Z \subseteq M^x$ such that $X \cap Z = \emptyset$.
The following theorem is due to Chernikov and Simon~\cite[Proposition 1.6]{CS-II}.

\begin{fact}
\label{fact:Honest}
Suppose $\Sa M$ is $\mathrm{NIP}$.
Then every externally definable subset of $M^x$ has an honest definition.
\end{fact}

\noindent It is an exercise to show that Fact~\ref{fact:shelah} follows from Fact~\ref{fact:Honest}. \newline

\noindent We will use dp-ranks in a couple places.
We refer to \cite[Chapter 4]{Simon-Book} for the definition and background information.
A structure is dp-minimal if it has dp-rank at most one.  Onshuus and Usvyatsov~\cite[Observation 3.8]{OnUs} show that the Shelah expansion of a dp-minimal structure is dp-minimal.
Their proof yields Fact~\ref{fact:onus}.

\begin{fact}
\label{fact:onus}
If $\Sa M$ is $\nip$, then the dp-rank of $\Sh M$ agrees with the dp-rank of $\Sa M$.
\end{fact}

\noindent Fact~\ref{fact:subadditive} is proven by Kaplan, Onshuus, and Usvyatsov~\cite{dp-rank-additive}.

\begin{fact}
\label{fact:subadditive}
Suppose $\Sa M$ is $\nip$.
Suppose $X,Y$ are $\Sa M$-definable sets of dp-rank $m,n$, respectively.
Then the dp-rank of $X \times Y$ is at most $m + n$.
\end{fact}

\section{Uniform structures and $\bigwedge$-definable equivalence relations}

\subsection{Uniform structures}
\label{section:uniform-structures} 
We describe background from the classical theory of uniform structures.
We regard a basis for a uniform structure, defined below, as a generalization of a pseudo-metric, and regard a uniform structure as a generalization of a pseudo-metric considered up to uniform equivalence.
One can consider a uniform structure to be an ``approximate equivalence relation".
As $\bigwedge$-definable equivalence relations are approximated by definable relations, there is a natural connection between $\bigwedge$-definable equivalence relations and uniform structures.\newline

\noindent
Let $X$ be a set and declare $\Delta := \{ (x,x) : x \in X\}$
Given $A,B \subseteq X^2$ we declare
$$ A^{-1} := \{ (x',x) : (x,x') \in A\} $$
and
$$ A \circ B := \{ (x,x') : \exists y\: (x,y) \in A, (y,x') \in B \}. $$
We say that $A \subseteq X^2$ is \textbf{symmetric} if $A^{-1} = A$.
A \textbf{basis for a uniform structure} on $X$ is a collection $\Cal B$ of subsets of $X^2$ satisfying the following conditions.
\begin{enumerate}
    \item Every element of $\Cal B$ contains $\Delta$.
    \item Every element of $\Cal B$ is symmetric.
    \item For all $U,U' \in \Cal B$ there is $V \in \Cal B$ such that $V \subseteq U \cap U'$.
    \item For every $U \in \Cal B$ there is $V \in \Cal B$ such that $V \circ V \subseteq U$.
\end{enumerate}
Suppose $\Cal B$ is a basis for uniform structure on $X$.
The \textbf{uniform structure} $\overline{\Cal B}$ on $X$ generated by $\Cal B$ is the collection of all subsets of $X^2$ which contain some element of $\Cal B$.
Suppose $\Cal B$ is a basis for a uniform structure on $X$.
We associate a topology on $X$ to $\Cal B$ by declaring a subset $A$ of $X$ to be open if and only if for every $a \in A$ there is $U \in \Cal B$ such that $U_a \subseteq A$.
This topology only depends on $\overline{\Cal B}$.
Note that $(U_a)_{U \in \Cal B}$ is a neighbourhood basis for $a \in X$.
(In general $U_a$ need not be an open neighbourhood of $a$.)
We say that a uniform structure is Hausdorff if the associated topology is Hausdorff.
It is easy to see that $\overline{\Cal B}$ is Hausdorff if and only if $\bigcap \Cal B = \Delta$. \newline

\noindent We leave the following lemma as a exercise.

\begin{lemma}
\label{fact:uniform-basic}
Suppose $\Cal B$ is a basis for a uniform structure on $X$.
Let $\Cal C$ be a collection of subsets of $X^2$ such that
\begin{enumerate}
    \item Every $V \in \Cal C$ is symmetric,
    \item Every $V \in \Cal C$ contains some $U \in \Cal B$,
    \item $\Cal B$ is a subfamily of $\Cal C$.
\end{enumerate}
Then $\Cal C$ is a basis for $\overline{\Cal B}$.
\end{lemma}

\noindent
We recall the main examples of uniform structures.
If $(X,d)$ is a pseudo-metric space then the collection of sets of the form $\{ (x,x') \in X^2 : d(x,x') < t \}$ for $t > 0$ is a basis for a uniform structure on $X$ which induces the $d$-topology.
If $G$ is a topological group and $\Cal U$ is a neighbourhood basis for the identity then the collection $\{ (g,g') \in G^2 : g^{-1}g' \in U\}$, $U$ ranging over $\Cal U$, is a basis for a uniform structure on $G$ which induces the group topology.
Suppose $\tau$ is a compact Hausdorff topology on $X$ and equipp $X^2$ with the associated product topology.
Then the collection of all symmetric open subsets of $X^2$ containing $\Delta$ forms a basis for a uniform structure on $X$.
This is the unique uniform structure on $X$ for which the associated topology is $\tau$.
If $E$ is an equivalence relation on $X$ then $\{E\}$ is a basis for a uniform structure on $X$.
Finally, the discrete uniform structure on $X$ is the uniform structure with basis $\{ \Delta \}$. \newline

\noindent
Suppose $\Cal C$ is a basis for a uniform structure on a set $Y$.
Let $f$ be a function $X \to Y$.
Then $f$ is uniformly continuous if for every $U \in \Cal C$ there is $V \in \Cal B$ such that for all $(a,b) \in V$ we have $(f(a),f(b)) \in U$.
If $f$ is uniformly continuous then $f$ is a continuous map between the topologies associated to $\Cal B$ and $\Cal C$.
We say that $f$ is a \textbf{uniform equivalence} if $f$ is bijective and $f,f^{-1}$ are both uniformly continuous.\newline

\noindent One can associate a quotient Hausdorff uniform structure to a general uniform structure in a canonical way.
It is easy to see that $E := \bigcap \Cal B$ is an equivalence relation on $X$.
Let $\pi$ be the quotient map $X \to X/E$.
Given $U \in \Cal B$ we let $\pi(U)$ to be the set of $(y,y') \in (X/E)^2$ such that there are $x \in \pi^{-1}(\{y\}), x' \in \pi^{-1}(\{y'\})$ satisfying $(x,x') \in U$. 
Let $\pi(\Cal B) := (\pi(U))_{U \in \Cal B}$.
Then $\pi(\Cal B)$ is a basis for a Hausdorff uniform structure on $X/E$.
Note that $E$ and the uniform structure on $X/E$ depend only on $\overline{\Cal B}$.\newline

\noindent Suppose $Y$ is a subset of $X$.
Then 
$$ \Cal B|_Y := \{ U \cap Y^2 : U \in \Cal B \} $$
is a basis for a uniform structure on $Y$ which we refer to as the induced uniform structure on $Y$.\newline

\noindent We describe the product uniform structure on $X^n$.
Given $U \in \Cal B$ we let $U_n$ be the set of $(x,x') \in X^n \times X^n$ such that $(x_k,x'_k) \in U$ for all $1 \leq k \leq n$ where $x = (x_1,\ldots,x_n)$ and $x' = (x'_1,\ldots,x'_n)$.
Then $\Cal B_n := \{ U_n : U \in \Cal B\}$ is a basis for a uniform structure on $X^n$ which induces the product topology on $X^n$.\newline

\noindent Lemma~\ref{lem:uniform} will be crucial below.
Given a subset $A$ of $X$ and $U \in \Cal B$ we let $A[U]$ be $\bigcup_{a \in A} U_a$.
Note that $A$ lies in the interior of $A[U]$.

\begin{lemma}
\label{lem:uniform}
Suppose $\Cal W$ is a finite collection of nonempty open subsets of $X$ and $A \subseteq X$ is not dense in any $W \in \Cal W$.
Then there is $U \in \Cal B$ such that $W \setminus A[U]$ has interior for all $W \in \Cal W$.
\end{lemma}

\begin{proof}
Fix $W \in \Cal W$.
We show there is $U_W \in \Cal B$ such that $W \setminus A[U_W]$ has interior.
Fix an element $b$ of $W$ that does not lie in the closure of $A$.
So there is $V \in \Cal B$ such that $V_b$ is disjoint from $A$.
Let $U_W \in \Cal B$ be such that $U_W \circ U_W \subseteq V$.
We show that $A[U_W]$ and $(U_W)_b$ are disjoint.
Suppose $p$ lies in both $A[U_W]$ and $(U_W)_b$.
As $p \in A[U_W]$ there is $a \in A$ such that $p \in (U_W)_a$.
Then $(a,p) \in U_W$ and $(p,b) \in U_W$.
So $(a,b) \in V$.
Contradiction.
So $b$ lies in the interior of $W \setminus A[U_W]$. \newline

\noindent Now let $U \in \Cal B$ satisfy $U \subseteq \bigcap_{W \in \Cal W} U_W$.
Then $A[U] \subseteq A[U_W]$ for all $W \in \Cal W$ so $W \setminus A[U]$ has interior for all $W \in \Cal W$.
\end{proof}

\noindent We now suppose that $\Sa M$ is a (possibly multisorted) structure.
Given a definable set $X$, a (sub)definable basis for a uniform structure on $X$ is a (sub)definable family of sets forming a basis for a uniform structure on $X$.
If $\Cal B$ is a subdefinable basis for a uniform structure on $X$ then $\Cal B_n$ is a subdefinable basis for the product uniform structure on $X^n$ and $\Cal B_n|_Y$ is a subdefinable basis for the induced uniform structure on a definable $Y \subseteq X$.

\begin{lemma}
\label{lem:basis-ext-def}
Suppose $\Cal B$ is a subdefinable basis for a uniform structure on an $\Sa M$-definable set $X$.
Then $E := \bigcap \Cal B$ is externally definable in $\Sa M$.
\end{lemma}

\begin{proof}
Suppose $\Cal B$ is a subdefinable basis for $E$.
Let $\phi(x,y)$ be a formula such that for every $U \in \Cal B$ there is $a \in M^x$ such that $\phi(a,M^y) = U$.
It suffices to show that the following partial $x$-type is consistent
$$ \{ \phi(x,e) : e \in E \} \cup \{ \forall y [\phi(x,y) \to \phi(a,y)] : \phi(a,M^y) \in \Cal C \}.  $$
So it suffices to show that for any $U_1,\ldots,U_n \in \Cal B$ there is $a \in M^x$ such that 
$$ E \subseteq \phi(a,M^y) \subseteq \bigcap_{k = 1}^{n} U_k .$$
As $\Cal B$ is a basis for the uniform structure associated to $E$ there is $V \in \Cal B$ such that $V \subseteq \bigcap_{k = 1}^{n} U_k$. 
Let $a \in M^x$ satisfy $\phi(a,M^y) = V$.
\end{proof}

\begin{prop}
\label{prop:subdef-to-def}
Suppose $\Sa M$ is $\nip$.
Suppose $X$ is an $\Sa M$-definable set and $\Cal B$ is a subdefinable basis for a uniform structure on $X$.
Then $\Sh M$ defines a basis for $\overline{\Cal B}$.
\end{prop}

\begin{proof}
Let $(B_a)_{a \in M^x}$ be a definable family of sets such that $\Cal B$ is a subfamily of $(B_a)_{a \in M^x}$.
Let $\Sa N$ be a highly saturated elementary extension of $\Sa M$.
Given an $\Sa M$-definable set $Y$ we let $Y'$ be the $\Sa N$-definable set defined by the same formula as $Y$.
We also let $(B'_a)_{a \in N^x}$ be the family of sets defined by the same formula as $(B_a)_{a \in M^x}$.
Note that $\Cal B' := \{ U' : U \in \Cal B \}$ is a subdefinable (in $\Sa N$) basis for a uniform structure on $X'$.
Let $F := \bigcap \Cal B'$.
Lemma~\ref{lem:basis-ext-def} shows that $F$ is definable in $\Sh N$.
An application of saturation shows that if $Y$ is an $\Sa M$-definable subset of $X^2$ then $F \subseteq Y'$ if and only if $Y$ contains some $U \in \Cal B$.
Let $Z$ be the set of $a \in N^x$ such that $F \subseteq B'_a$.
Then $Z$ is definable in $\Sh N$.
Facts~\ref{fact:shelah} and \ref{fact:extension-external} together show that $Q := Z \cap M^x$ is definable in $\Sh M$.
Then $Q$ is the set of $a \in M^x$ such that $B_a$ contains some $U \in \Cal B$.
In particular $\Cal B$ is a subfamily of $(B_a)_{a \in Q}$. \newline

\noindent Let $C_a = B_a \cap B_a^{-1}$ for all $a \in Q$.
Note that each $C_a$ is symmetric and contains some $U \in \Cal B$.
Note also that $\Cal B$ is a subfamily of $(C_a)_{a \in Q}$.
An application of Lemma~\ref{fact:uniform-basic} shows that $(C_a)_{a \in Q}$ is a basis for $\overline{\Cal B}$.
\end{proof}

\noindent Corollary~\ref{cor:subdef} follows immediately from Proposition~\ref{prop:subdef-to-def}.

\begin{cor}
\label{cor:subdef}
Suppose $\Sa M$ is $\nip$, $X$ is an $\Sa M$-definable set, and there is a subdefinable basis for a uniform structure on $X$.
Let $\Sa X$ be the structure induced on $X$ by $\Sa M$.
Then the open core $\Sa X^\circ$ of $\Sa X$ is a reduct of the structure induced on $X$ by $\Sh M$.
\end{cor}


\subsection{$\bigwedge$-definable equivalence relations}
We now describe how uniform structures arise in abstract model theory.
\textbf{In this section $\Sa N$ is a highly saturated structure and $X$ is an $\Sa N$-definable set.}
There is a  a canonical correspondence between
\begin{itemize}
    \item $\bigwedge$-definable equivalence relations $E$ on $X$, and
    \item uniform structures on $X$ which admit a basis $\Cal B$ consisting of a small family of definable sets.
\end{itemize}
This correspondence is well-known to experts and used implicitly throughout the literature.
If $\Cal B$ is a small collection of definable subsets of $X^2$ forming a basis for a uniform structure on $X$ then $E$ is $\bigcap \Cal B$.
Note that $E$ is $\bigwedge$-definable and only depends on $\overline{\Cal B}$.
Now suppose $E$ is a $\bigwedge$-definable equivalence relation on $X$.
Let $\Cal C$ be a small collection of definable subsets of $X$ such that $\bigcap \Cal C = E$.
After replacing $\Cal C$ with
$$ \{ C_1 \cap \ldots \cap C_n \cap C^{-1}_1 \cap \ldots \cap C^{-1}_n : C_1,\ldots,C_n \in \Cal C \} $$
if necessary we may suppose that every element of $\Cal C$ is symmetric and that for all $C,C' \in \Cal C$ there is $B \in \Cal C$ such that $B \subseteq C \cap C'$.
It is now an exercise in saturation to show that $\Cal C$ is a basis for a uniform structure on $X$ and that this uniform structure does not depend on choice of $\Cal C$.
It is also a saturation exercise to see that the induced uniform structure on $X/E$ is the discrete uniform structure if and only if $E$ is definable.
Finally, the quotient $X/E$ is small if and only if the induced topology on $X/E$ is compact.
In this case the topology is known as the \textbf{logic topology} on $X/E$.
A subset of $X/E$ is closed in the logic topology if and only if it is the image of a $\bigwedge$-definable subset of $X$ under the quotient map.
In general we refer to the topology on $X/E$ as the logic topology.

\begin{prop}
\label{prop:image}
Let $E$ be a $\bigwedge$-definable equivalence relation on $X$ and $\pi$ be the quotient map $X \to X/E$.
Then $\pi(Y)$ is closed for any definable $Y \subseteq X$.
\end{prop}

\begin{proof}
Let $\Cal B$ be a small collection of definable subsets of $X^2$ which forms a basis for the uniform structure on $X$.
Let $Y$ be a definable subset of $X$.
Suppose $p \in X/E$ lies in the closure of $\pi(Y)$.
Fix $p' \in X$ such that $\pi(p') = p$.
For every $U \in \Cal B$ there is $r_U \in Y$ such that $(p',r_U) \in U$.
As $\Cal B$ is small and $\bigcap \Cal B = E$ an application of saturation  yields an $r \in Y$ such that $(p',r) \in E$.
So $\pi(r) = p$ and $p \in \pi(Y)$.
\end{proof}

\noindent We also equipp $(X/E)^n$ with a logic topology.
Let $E_n$ be the equivalence relation on $X^n$ where $(a_1,\ldots,a_n)$ and $(b_1,
\ldots,b_n)$ are $E_n$-equivalent if $(a_k,b_k) \in E$ for all $1 \leq k \leq n$.
We identify $(X/E)^n$ and $(X^n/E_n)$.
Observe that $E_n$ is a $\bigwedge$-definable equivalence relation on $X^n$.
Observe that if $\Cal B$ is a small collection of definable sets forming a basis for the uniform structure associated to $E$ then $\Cal B_n$ is a small collection definable sets forming a basis for the uniform structure associated to $E_n$.
So the uniform structure on $(X/E)^n$ is simply the product uniform structure. \newline

\noindent Let $E$ be a $\bigwedge$-definable equivalence relation on $X$.
A subdefinable (in $\Sa N$) basis for $E$ is a subdefinable (in $\Sa N$) basis for the uniform structure on $X$ associated to $E$.
Lemma~\ref{lem:basis} is a saturation exercise which we leave to the reader.

\begin{lemma}
\label{lem:basis}
Suppose $\Cal C$ is a small subdefinable collection of subsets of $X^2$.
Then $\Cal C$ is a subdefinable basis for $E$ if and only if
\begin{enumerate}
    \item $\bigcap \Cal C = E$,
    \item each $U \in \Cal C$ is symmetric, and
    \item for all $U,U' \in \Cal C$ there is $V \in \Cal C$ such that $V \subseteq U \cap U'$.
\end{enumerate}
If $\Cal C$ satisfies $(1)$ and $(3)$ then $\{ U \cap U^{-1} : U \in \Cal C\}$ is a subdefinable basis for $E$.
\end{lemma}

\begin{prop}
\label{prop:subdef-to-def-cor}
Suppose there is a subdefinable (in $\Sa N$) basis for $E$.
Then $\Sh N$ defines a basis for the uniform structure on $X/E$.
\end{prop}

\begin{proof}
Lemma~\ref{lem:basis-ext-def} shows that $E$ is externally definable and Proposition~\ref{prop:subdef-to-def} shows that $\Sh N$ defines a basis $\Cal B$ for the uniform structure on $X$.
Then $\pi(\Cal B)$ is an $\Sh N$-definable basis for the uniform structure on $X/E$.
\end{proof}

\noindent If the answer to Question~\ref{qst:basis} is positive then Proposition~\ref{prop:subdef-to-def-cor} shows that if $\Sa N$ is $\nip$ and $E$ is externally definable then $\Sh N$ defines a basis for the uniform structure on $X/E$.

\begin{qst}
\label{qst:basis}
If $\Sa N$ is $\nip$ does every equivalence relation which is both $\bigwedge$-definable and externally definable admit a subdefinable basis?
\end{qst}

\noindent It is possible that Question~\ref{qst:basis} fails in general but holds under more robust model-theoretic assumptions on $\Sa N$ such as distality.
At present we do not even know if Question~\ref{qst:basis} has a positive answer when $\Sa N$ is o-minimal.
Propositions~\ref{prop:basis-groups} gives an answer in a very special case.

\begin{prop}
\label{prop:basis-groups}
Suppose $\Sa N$ is $\nip$.
Suppose $X$ is a definable group, $\Cal H$ is a small subdefinable collection of definable subgroups of $X$, and $H := \bigcap \Cal H$.
Let $E$ be equivalence modulo $H$.
Then $E$ admits a subdefinable basis.
\end{prop}

\begin{proof}
Applying the Baldwin-Saxl theorem~\cite[Theorem 2.13]{Simon-Book} we obtain $n$ such that for any finite $\Cal H' \subseteq \Cal H$ there is $\Cal H'' \subseteq \Cal H'$ such that $|\Cal H''| \leq n$ and $\bigcap \Cal H'' = \bigcap \Cal H'$.
Lemma~\ref{lem:basis} shows that the collection of sets of the form $ \{ (g,h) \in X: g h^{-1} \in \bigcap \Cal H'\}$ for $\cal H' \subseteq \Cal H, |\Cal H'| \leq n$
is a subdefinable basis for $E$.
\end{proof}

\section{Really strong Baire category theorem}
\noindent \textbf{Throughout this section $\Sa M$ is a possibly multisorted first order structure, ``definable" without modification means ``$\Sa M$-definable", and $X$ is a definable set.} \newline

\noindent We prove a result on subdefinable families of nowhere dense subsets of uniform spaces in $\ntp$ structures.
The $\nip$ case is crucial for the proof of Theorem B. \newline 

\noindent We say that a family $(X_t)_{t > 0}$ of sets is increasing if $X_s \subseteq X_t$ whenever $s \leq t$.
The classical Baire category theorem is easily seen to be equivalent to the following: If $(Z,d)$ is a complete metric space, and $(X_t)_{t > 0}$ is an increasing family of nowhere dense subsets of $Z$, then $\bigcup_{t > 0} X_t$ has empty interior.
It is shown in \cite[Theorem D]{FHW-Compact} that if $\Sa R$ is a type A expansion of $(\R,<,+)$ and $(X_t)_{t > 0}$ is an increasing definable family of nowhere dense subsets of $\R^n$ then $\bigcup_{t > 0} X_t$ is nowhere dense.
This result is known as the ``strong Baire category theorem".
Dolich and Goodrick \cite[Proposition 2.16]{DG} show that if $\Sa R$ is a strong (in particular strongly dependent) expansion of a dense ordered abelian group $(R,<,+)$, $I \subseteq R$ is an open interval, and $(X_a)_{a \in I}$ is an increasing definable family of discrete subsets of $R$, then $\bigcup_{a \in I} X_a$ is nowhere dense.
A family $\Cal A$ of sets is \textbf{directed} if for all $A,A' \in \Cal A$ there is $B \in \Cal A$ such that $A \cup A' \subseteq B$.
It follows from \cite[Lemma 3.5]{SW-tame} that if $\Sa M$ is a dp-minimal expansion of a divisible ordered abelian group or a non strongly minimal dp-minimal expansion of a field (equipped with the Johnson topology) and $\Cal A$ is a directed definable family of nowhere dense subsets of $M^n$ then $\bigcup \Cal A$ is nowhere dense.\newline

\noindent Theorem~\ref{thm:abstract-sbct} generalizes the $\ntp$ case of the strong Baire category theorem and the latter two results described in the preceeding paragraph.


\begin{thm}
\label{thm:abstract-sbct}
Suppose $\Cal B$ is a subdefinable basis for a uniform structure on $X$, $W$ is a nonempty open subset of $X$, and $\Cal A$ is a definable family of subsets of $X$ such that for every finite collection $\Cal U$ of nonempty open subsets of $W$ there is $A \in \Cal A$ such that $A$ is nowhere dense in $W$ and intersects each $U \in \Cal U$.
Then $\Sa M$ is $\mathrm{TP}_2$.
\end{thm}

\noindent 
Let $[n]$ denote $\{0,\ldots,\}$.
Let $[n]^{= m}$ be the set of functions $\{0,\ldots,m\} \to \{0,\ldots,n\}$ which we identify with $\{0,\ldots,n\}$-valued sequences of length $m$.
Let $[n]^{\leq m}$ be the union of all $[n]^{=k}$ where $0 \leq k \leq m$.
Given $\sigma \in [n]^{\leq m}$ and $\eta \in [n]^{\leq k}$ we let $|\sigma|$ be the length of $\sigma$, and let $\sigma^\frown \eta \in [n]^{\leq m + k}$ be the usual concatenation of $\sigma$ and $\eta$.

\begin{proof}
Let $R$ be the relation on $W \times (\Cal A \times \Cal B \times \Cal B)$ where $(a,(A,U,U')) \in R$ if and only if $a \in A[U] \setminus A[U']$.
We show that $R$ has $\mathrm{TP}_2$.
Fixing $n \geq 1$ we construct elements $(A_i :  i \in [n])$ of $\Cal A$, elements $(U^i_j : i,j \in [n] )$ of $\Cal B$, and nonempty open subsets $(W_\sigma :\sigma \in [n -1]^{\leq n} )$ of $W$ such that:
\begin{enumerate}
\item $U^i_{j + 1} \subsetneq U^i_{j}$ and $A_i[U^i_{j}] \setminus A_i[U^i_{j + 1}]$ has interior for all $i \in [n]$ and $j \in [n-1]$,
\item $W_{\sigma \frown i},  W_{\sigma \frown j}$ are pairwise disjoint subsets of $W_\sigma$ for all $\sigma \in [n -1]^{< n}$ and $i,j \in [n-1]$,
\item and if $p \in W_{\sigma}$ then $p \in A_i[U^i_{j}] \setminus A_i[U^i_{j + 1}]$ if and only if $\sigma(i) = j$.
\end{enumerate}
This shows that $R$ has $\mathrm{TP}_2$. \newline

\noindent Let $A_0 \in \Cal A$ be nowhere dense in $W$.
Applying Lemma~\ref{lem:uniform} inductively let $U^0_0,\ldots,U^0_n$ be elements of $\Cal B$ such that $U^0_{j +1} \subsetneq U^0_{j}$ and $A_0[U^0_{j}] \setminus A_0[U^0_{j + 1}]$ has interior in $W$ for all $j \in [n-1]$.
For all $j \in [n-1]$ let $W_j$ be a nonempty open subset of $W$ contained in $A_0[U^0_{j}] \setminus A_0[U^0_{j + 1}]$. \newline

\noindent Fix $1 \leq k \leq n - 1$.
Suppose we have constructed $(A_i : i \in [k])$, $(U^i_j : i \in [k], j \in [n])$, and $(W_{\sigma} : \sigma \in [n-1]^{\leq k})$.
Let $A_{k+1}$ be an element of $\Cal A$ which is nowhere dense in $W$ and intersects every element of $(W_{\sigma} : \sigma \in [n-1]^{= k})$.
Applying Lemma~\ref{lem:uniform} inductively let $U^{k+1}_0,\ldots ,U^{k+1}_n \in \Cal B$ be such that $U^{k+1}_{j + 1} \subsetneq U^{k+1}_{j}$  and $A_{k+1}[U^{k+1}_{j}] \setminus A_{k+1}[U^{k+1}_{j + 1}]$ has interior in each element of $(W_{\sigma} : \sigma \in [n - 1]^{=k})$ for all $j \in [n]$.
For each $\sigma \in [n - 1]^{=k}$ and $j \in [n-1]$ we let $W_{\sigma^{\frown}j}$ be a nonempty open subset of $W_\sigma$ contained in $A_{k+1}[U^i_{j}]\setminus A_{k+1}[U^i_{j + 1}]$.
\end{proof}



\noindent The statement of Proposition~\ref{prop:sbct-cor} is closer to previous forms of strong Baire category.

\begin{prop}
\label{prop:sbct-cor}
Suppose $\Sa M$ is $\mathrm{NTP}_2$.
Suppose $\Cal B$ is a subdefinable basis for a uniform structure on $X$.
Let $Y$ be a definable subset of $X$, $\Cal A$ be a subdefinable directed family of subsets of $X$, and $W$ be a nonempty open subset of $Y$.
If each element of $\Cal A$ is nowhere dense in $W$ then $\bigcup \Cal A$ is nowhere dense in $W$.
\end{prop}

\begin{proof}
We only treat the case when $X = Y$.
The more general case follows by replacing $X$ with $Y$ and $\Cal B$ with $\Cal B|_Y$.
Suppose $\bigcup \Cal A$ is somewhere dense in $W$.
Let $V$ be a nonempty open subset of $Y$ in which $\bigcup \Cal A$ is dense.
Suppose $\Cal U$ is a finite collection of nonempty open subsets of $V$.
Pick for each $U \in \Cal U$ an $A_U \in \Cal A$ such that $U \cap A_U \neq \emptyset$.
As $\Cal A$ is directed there is  $A \in \Cal A$ such that $A_U \subseteq A$ for all $U \in \Cal U$.
This $A$ intersects each element of $\Cal U$.
An application of Theorem~\ref{thm:abstract-sbct} shows that some element of $\Cal A$ is somewhere dense in $W$.
\end{proof}

\noindent
Corollary~\ref{cor:sbct-cor-cor} follows immediately from Proposition~\ref{prop:sbct-cor}.

\begin{cor}
\label{cor:sbct-cor-cor}
Suppose $\Sa M$ is $\ntp$.
Suppose $\Cal B$ is a subdefinable basis for a uniform structure on $X$ with no isolated points.
Let $\Cal A$ be a subdefinable directed family of discrete subsets of $X$.
Then $\bigcup \Cal A$ is nowhere dense.
In particular if $\Cal A$ is a subdefinable directed family of finite subsets of $X$ then $\bigcup \Cal A$ is nowhere dense.
\end{cor}

\noindent We give three applications of these results before proving Theorem B.

\subsection{Shelah expansions}
Proposition~\ref{prop:shelah-noise} naturally goes through in greater generality, we leave that to the reader.

\begin{prop}
\label{prop:shelah-noise}
Suppose $\Sa M$ is $\nip$ and one-sorted.
Suppose there is a subdefinable basis for a uniform structure on $M$.
If $\Sa M$ is (strongly) noiseless then $\Sh M$ is (strongly) noiseless.
\end{prop}

\begin{proof}
We suppose that $\Sa M$ is strongly noiseless and show that $\Sh M$ is strongly noiseless, the same argument shows that if $\Sa M$ is noiseless then $\Sh M$ is noiseless.
Applying Proposition~\ref{prop:subdef-to-def} we suppose that $\Cal B$ is an $\Sh M$-definable basis for the uniform structure on $M$.
Suppose $X,Y$ are $\Sh M$-definable subsets of $M^n$ such that $X$ is somewhere dense in $Y$ and $X$ has empty interior in $Y$.
Let $U$ be a nonempty open subset of $Y$ in which $X$ is dense.
As $\Cal B$ is an $\Sa M$-definable basis we may suppose $U$ is $\Sh M$-definable.
After replacing $Y$ with $U \cap Y$ we suppose $X$ is dense in $Y$.
We apply Theorem~\ref{thm:abstract-sbct} to $Y$ and $\Cal B^n|_Y$.
Applying Facts~\ref{fact:shelah} and \ref{fact:cs} we let $(X_a)_{a \in M^m}$ be an $\Sa M$-definable family of sets such that for every finite $A \subseteq X$ there is $a \in M^m$ such that $A \subseteq X_a \subseteq X$.
As $X$ has empty interior in $Y$ and $\Sa M$ is strongly noiseless $X_a$ is nowhere dense in $Y$ when $X_a \subseteq X$.
Let $\Cal U$ be a finite collection of nonempty open subsets of $Y$.
As $X$ is dense in $Y$ there is a finite subset $A$ of $X$ which intersects each $U \in \Cal U$.
So there is $a \in M^m$ such that $A \subseteq X_a \subseteq X$, this $X_a$ intersects each $U \in \Cal U$ and is nowhere dense in $Y$.
This contradicts Theorem~\ref{thm:abstract-sbct}
\end{proof}

\subsection{Cyclic orders on $(\Z,+)$} 
We equip $\R/\Z$ with the cyclic group order $C$ where whenever $0 \leq a,b,c < 1$ then $C(a + \Z, b + \Z, c + \Z)$ holds if and only if either $a < b <c$, $b < c < a$, or $c < a < b$.
Given an injective character $\chi : (\Z,+) \to (\R/\Z,+)$ we let $C_\chi$ be pullback of $C$ by $\chi$.
Then $C_\chi$ is a cyclic group order on $(\Z,+)$.
Tran and Walsberg show that $(\Z,+,C_\chi)$ is $\nip$~\cite{TW-cyclic}.
 
\begin{prop}
\label{prop:order-val}
Let $\chi : (\Z,+) \to (\R/\Z,+)$ be an injective character.
Then $(\Z,+,C_\chi,<)$ is $\tp$.
\end{prop}

\begin{proof}
Observe that the collection of sets of the form $\{ b \in \Z : C(a,b,c) \}$, $a,c \in \Z$ forms a $(\Z,+,C_\chi)$-definable basis for a topology on $\Z$ with no isolated points.
It follows that $(\Z,+,C_\chi)$ defines a basis for a uniform structure on $\Z$ with no isolated points.
Furthermore $\{ [-n,n] : n \in \N \}$ is a $(\Z,+,C_\chi,<)$-definable directed family of finite sets with union $\Z$.
Corollary~\ref{cor:sbct-cor-cor} shows that $(\Z,+,C_\chi,<)$ is $\tp$.
\end{proof}

\noindent It is natural to ask if $(\Z,+,C_\chi,<)$ is totally wild.
Every injective character $\chi : (\Z,+) \to (\R/\Z,+)$ is of the form $\chi(k) = \alpha k + \Z$ for some $\alpha \in \R \setminus \Q$.
It is easy to see that $C_\chi$ is definable in $(\R,<,+,\Z,\az)$.
If $\alpha$ is a quadratic irrational then $(\R,<,+,\Z,\az)$ is parameter-free intepretable in $(\Cal P(\N),\N,\in,s)$ by Hieronymi~\cite{H-Twosubgroups} and hence decidable.
In particular $(\Z,+,C_\chi,<)$ does not define multiplication on $\Z$. \newline

\noindent
The next result shows in fact that an $\ntp$ expansion of $(\Z,<)$ cannot interpret $(\Z,+,C_\chi)$.
We will return to Proposition~\ref{prop:interpret} around Conjecture~\ref{conj:omega-order} below.

\begin{prop}
\label{prop:interpret}
Suppose $\Sa Z$ is an $\ntp$ expansion of $(\Z,<)$.
Let $\Sa G$ be a first order expansion of a group $G$ which admits a subdefinable basis for a non-discrete group topology on $G$.
Then $\Sa Z$ does not interpret $\Sa G$.
\end{prop}

\noindent Proposition~\ref{prop:interpret} fails for elementary extensions of $(\Z,<)$.
Suppose $(Z,<,+)$ is a proper elementary extension of $(\Z,<,+)$.
As $\Z$ is a convex subset of $Z$ an application of Fact~\ref{fact:convex} shows that $(Z,<,+,\Z)$ is $\nip$.
We declare $a + \Z \triangleleft b + \Z$ if every element of $a + \Z$ is strictly less then every element of $b + \Z$.
Then $\triangleleft$ is a dense group ordering on $Z/\Z$.

\begin{proof}
Suppose towards a contradiction that $\Sa Z$ interprets $\Sa G$.
It is easy to see that any expansion of $(\Z,<)$ eliminates imaginaries (reduce to the case when every equivalence class intersects $\N^m$, than select the lexicographic-minimal element of the intersection of each class with $\N^m$).
So we may suppose that $\Sa G$ is definable in $\Sa Z$.
Suppose $G$ is a subset of $\Z^m$.
Then $([-n,n]^m \cap G)_{n > 0}$ is a directed $\Sa Z$-definable family of finite subsets with union $G$.
As $\Sa G$ admits a subdefinable basis for a non-discrete group topology, it admits a subdefinable basis for a uniform structure with no isolated points.
So an application of Corollary~\ref{cor:sbct-cor-cor} shows that $\Sa Z$ is $\tp$.
\end{proof}

\subsection{Porosity} We give an application to the study of metric properties of definable sets in expansions of ordered fields.
Porosity (for subsets of $\R^n$) is extensively studied in metric geometry.
We discuss the natural generalization of this notion to subsets of $R^n$, $R$ an arbitrary ordered field.
There are several notions of porosity in the literature, we discuss the strongest notion we are aware of.\newline

\noindent Let $(R,<,+,\times)$ be an ordered field and $R_{>0}$ be the set of positive elements of $R$.
We equip $R^n$ with the natural $l_\infty$ $R$-valued metric.
We let $B_n(p,r)$ be the open ball with center $p \in R^n$ and radius $r \in R_{>0}$.
Let $X$ be a subset of $R^n$.
Given $\delta \in R_{>0}$ we say that $X$ is \textbf{$\delta$-porous} if for every $p \in R^n, r \in R_{>0}$ there is $q \in R^n$ such that $B_n(q,\delta r) \subseteq B_n(p,r) \setminus X$.
We say that $X$ is porous if it is $\delta$-porous for some $\delta \in R_{>0}$. 
It is easy to see that a porous subset of $R^n$ is nowhere dense.
A good example of a nowhere dense non-porous subset of $\R$ is $\{ n^{-k} : n \geq 1\}$ for any $k \geq 1$. 
A family $\Cal X$ of subsets of $R^n$ is \textbf{uniformly porous} if there is $\delta \in R_{>0}$ such that every $X \in \Cal X$ is $\delta$-porous.\newline

\noindent Suppose $\Sa R$ is a first order expansion of $(\R,<,+,\times)$.
Hieronymi and Miller~\cite{HM} show that if $\Sa R$ does not define the set of integers, and $X$ is a closed subset of $\R^n$ definable in $\Sa R$, then the topological dimension of $X$ agrees with the Assouad dimension of $X$.
(See \cite{Luukkainen} for an introduction to this important metric dimension.)
By Fact~\ref{fact:gen-dim} below a closed subset of $\R^n$ has topological dimension $< n$ if and only if it is nowhere dense.
It follows directly from the definition that the Assouad dimension of $X \subseteq \R^n$ agrees with the Assouad dimension of the closure of $X$.
Luukainen~\cite[Theorem 5.2]{Luukkainen} shows that a subset of $\R^n$ has Assouad dimension $< n$ if and only if it is porous.
So we obtain the following.

\begin{fact}
\label{fact:porous-1}
Suppose $\Sa R$ is a first order expansion of $(\R,<,+,\times)$.
If $\Sa R$ does not define the set of integers then every nowhere dense subset of $\R^n$ definable in $\Sa R$ is porous.
\end{fact}

\noindent
We prove a strengthening of Fact~\ref{fact:porous-1} for $\ntp$ expansions of ordered fields.
We will not need to apply Fact~\ref{fact:porous-1}.

\begin{prop}
\label{prop:porous-2}
Let $\Sa R$ be an $\ntp$ expansion of $(R,<,+,\times)$.
Then every $\Sa R$-definable family of nowhere dense subsets of $R^n$ is uniformly porous.
In particular every nowhere dense subset of $R^n$ definable in $\Sa R$ is porous.
If $\Sa R$ is an elementary extension of an expansion of an archimedean ordered field then every nowhere dense definable subset of $R^n$ is $q$-porous for some $q \in \Q_{>0}$.
\end{prop}

\begin{proof}
The latter two claims follows easily from the first, so we only prove the first claim.
Let $\Cal X = (X_a)_{a \in R^m}$ be a definable family of nowhere dense subsets of $R^n$.
We suppose towards a contradiction that $\Cal X$ is not uniformly porous.
Given $a \in R^m$, $p \in R^n$, and $r \in R_{>0}$ we declare
$$ Y_{a,p,r} := \frac{1}{r}([ B_n(p,r) \cap X_a] - p). $$
Then $\Cal Y := \{ Y_{a,p,r} : (a,p,r) \in R^m \times R^n \times R_{>0} \}$ is a definable family of nowhere dense subsets of $B_n(0,1)$.
Let $(U_i)_{i = 1}^{k}$ be a collection of nonempty open subsets of $B_n(0,1)$.
For each $1 \leq i \leq n$ fix $q_i \in U_i$.
Let $\varepsilon \in R_{>0}$ be  such that $B_n(q_i,\varepsilon) \subseteq U_i$ for all $1 \leq i \leq k$.
Fix $a \in R^m$ such that $X_a$ is not $\varepsilon$-porous.
Fix $p \in R^n$ and $r \in R_{>0}$ such that every open ball of radius $\varepsilon r$ contained in $B_n(p,r)$ intersects $X_a$.
It is easy to see that every open ball of radius $\varepsilon$ contained in $B_n(0,1)$ intersects $Y_{a,p,r}$.
So $Y_{a,p,r}$ intersects each $B_n(q_i,\varepsilon)$ and hence intersects each element of $\Cal U$.
An application of Theorem~\ref{thm:abstract-sbct} shows that $\Sa R$ is $\tp$, contradiction.
\end{proof}

\noindent Suppose $\Sa R$ is highly saturated, let $\mfin$ be the convex hull of $\Q$ in $R$, let $\st : \mfin \to \R$ be the usual standard part map, and (abusing notation) let $\st: \mfin^n \to \R^n$ be
$$ \st(x_1,\ldots,x_n) = (\st(x_1),\ldots,\st(x_n)). $$
Then one can show that the following are equivalent:
\begin{enumerate}
    \item Every nowhere dense $\Sa R$-definable subset of $R^n$ is $q$-porous for some $q \in \Q_{>0}$,
    \item $\st(X)$ is nowhere dense in $\R^n$ for every nowhere dense $\Sa R$-definable $X \subseteq \mfin^n$.
\end{enumerate}
It follows that if the unit interval in $\Sa R$ is compactly dominated by $\st$ then every nowhere dense $\Sa R$-definable set is $q$-porous for some $q \in \Q_{>0}$. \newline

\noindent It is natural to ask if every nowhere dense definable set in an $\nip$ expansion of an ordered field is $q$-sparse for some $q \in \Q_{>0}$.
This need not be the case.
Let $\Sa S$ be an o-minimal expansion of $(\R,<,+,\times)$, $\Sa M$ be a proper elementary extension of $\Sa S$, and $\Sa R$ be $(\Sa M, \R)$.
Note that $\R$ is discrete in $M$ and is not $q$-sparse for any $q \in \Q_{>0}$.
Such structures were first studied by van den Dries and Lewenberg~\cite{t-convexity}, it is shown in \cite[Theorem 5.2]{GH-Dependent} that $\Sa R$ is $\nip$.

\section{Proof of Theorem B} 
\noindent We now prove Theorem B.

\begin{theorem}
\label{thm:main-noise}
Suppose $\Sa N$ is highly saturated and $\nip$.
Let $E$ be a $\bigwedge$-definable equivalence relation on $X$.
Suppose one of the following:
\begin{enumerate}
    \item $E$ is externally definable and there is a subdefinable (in $\Sh N$) basis for the uniform structure on $X/E$, or
    \item there is a subdefinable (in $\Sa N$) basis for $E$.
\end{enumerate}
Then the structure induced on $X/E$ by $\Sh N$ is strongly noiseless.
\end{theorem}

\noindent The proof is very similar to that of Proposition~\ref{prop:shelah-noise}.

\begin{proof}
Lemma~\ref{lem:basis-ext-def} and Proposition~\ref{prop:subdef-to-def} shows that $(2)$ implies $(1)$, so we suppose $(1)$.
Let $\Cal B$ be a subdefinable basis for the uniform structure on $X/E$.
Then $\Cal B_n$ is a subdefinable basis for the product uniform structure on $(X/E)^n$.
Let $Y,Y'$ be $\Sh N$-definable subsets of $(X/E)^n$.
We suppose that $Y$ is somewhere dense in $Y'$ and show that $Y$ has interior in $Y'$.
Note that $\Cal B_n|_{Y'}$ is an $\Sh N$-subdefinable basis for the induced uniform structure on $Y'$.
As above, let $E_n$ be the equivalence relation on $X^n$ where $(a_1,\ldots,a_n)$ and $(b_1,\ldots,b_n)$ are $E_n$-equivalent if and only if $(a_k,b_k) \in E$ for all $1 \leq k \leq n$.
Recall that we identify $(X/E)^n$ with $X^n/E_n$ and identify the product topology on $(X/E)^n$ with the logic topology on $X^n/E_n$.
\newline


\noindent Let $\pi$ be the quotient map $X^n \to (X/E)^n$.
Suppose towards a contradiction that $W$ is a nonempty open subset of $Y'$ such that $Y$ is dense and co-dense in $W$.
Then $Z := \pi^{-1}(Y)$ is $\Sh N$-definable and hence externally definable by Fact~\ref{fact:shelah}.
Applying Fact~\ref{fact:cs} we obtain a definable family $(Z_b)_{b \in N^x}$ of subsets of $X$ such that for any finite $A \subseteq Z$ there is $b \in N^x$ such that $A \subseteq Z_b \subseteq Z$.
Let $Y_b := \pi(Z_b)$ for $b \in N^x$.
Then $(Y_b)_{b \in N^x}$ is an $\Sh N$-definable family of sets.
For any finite $A \subseteq Y$ there is $b \in N^x$ such that $A \subseteq Y_b \subseteq Z$.
Suppose $\Cal U$ is a finite collection of nonempty open subsets of $W$.
As $Y$ is dense in $W$ there is a finite subset of $Y$ which intersects each $U \in \Cal U$.
So there is $b \in N^x$ such that $Y_b \subseteq Y$ and $Y_b$ intersects each element of $\Cal U$.
Proposition~\ref{prop:image} shows that each $Y_b$ is closed in $(X/E)^n$.
If $Y_b \subseteq Y$ then $Y_b$ is closed and co-dense in $W$ and is therefore nowhere dense in $W$.
As $\Sh N$ is $\nip$ this gives a contradiction with Theorem~\ref{thm:abstract-sbct}.
\end{proof}



\section{locally compact hausdorff uniform structures}
\label{section:loc-compact}
\noindent In this section we prove a more general version of Theorem A.
\textbf{Throughout this section, $\Sa M$ is a structure, $X$ is a definable set, $\Cal B$ is a subdefinable basis for a locally compact Hausdorff uniform structure on $X$, and ``definable" without modification means ``$\Sa M$-definable".} \newline

\noindent A \textbf{subdefinable compact exhaustion} of $X$ is a subdefinable family $\Cal K$ of compact subsets of $X$ such that every compact subset of $X$ is contained in some element of $\Cal K$.
\textbf{We additionally assume that $X$ admits a subdefinable compact exhaustion $\Cal K$}.
If $X$ is compact then we take $\Cal K = \{ X \}$ and this additional assumption is superfluous. \newline

\noindent We equip $X^n$ with the product uniform structure with the subdefinable basis $\{ U_n : U \in \Cal B\}$ defined in Section~\ref{section:uniform-structures}.
Note that $\{ K^n : K \in \Cal K \}$ is a subdefinable compact exhaustion of $X^n$.\newline

\noindent Let $\Sa N$ be a highly saturated elementary extension of $\Sa M$.
Given an $\Sa M$-definable set $Z$ we let $Z'$ be the $\Sa N$-definable set defined by any formula defining $Z$.
Observe that $\Cal B' := \{ U' : U \in \Cal B \}$ is a subdefinable basis for a uniform structure on $X'$.
Let $E := \bigcap \Cal B'$.
Then $\Cal B'$ is a subdefinable basis for $E$.  \newline

\noindent We define
$$ O := \bigcup \{ K' : K \in \Cal K\}. $$

\noindent Let $\pi$ be the quotient map $X' \to X'/E$.
Abusing notation we also let $\pi$ denote the map $(X')^n \to (X'/E)^n$ given by
$$ \pi(x_1,\ldots,x_n) = (\pi(x_1),\ldots,\pi(x_n)). $$
Lemma~\ref{lem:closed-image} is an easy variation of Proposition~\ref{prop:image}.
We leave the details of the proof to the reader.

\begin{lemma}
\label{lem:closed-image}
Suppose $Y$ is an $\Sa N$-definable subset of $(X')^n$.
Then $\pi(Y \cap O^n)$ is closed in $(O/E)^n$.
\end{lemma}

\noindent We show that $O/E$ may be identified with $X$ and view the quotient map $O \to O/E$ as a standard part map.

\begin{lemma}
\label{lem:unif-lem-0}
If $a \in O$ then the $E$-class of $a$ contains exactly one element of $X$.
\end{lemma}

\noindent We recall a useful definition from general topology.
Suppose $Z$ is a topological space.
A filter base on $Z$ is a collection $\Cal F$ of nonempty subsets of $Z$ such that for all $E,F \in \Cal F$ there is $G \in \Cal F$ such that $G \subseteq E \cap F$.
A point $p \in Z$ is a cluster point of a filter base $\Cal F$ if every neighbourhood of $p$ intersects every element of $\Cal F$.
If $Z$ is compact then every filter base has a cluster point.

\begin{proof}
If $x,y \in X$ are distinct then as $X$ is Hausdorff there is $U \in \Cal B$ such that $(x,y) \notin U$.
So each $E$-class contains at most one element of $X$.\newline

\noindent Fix $a \in O$.
We show that $a$ is $E$-equivalent to some element of $X$.
Let $K \in \Cal K$ be such that $K'$ contains $a$.
We first show that $(U'_a \cap K)_{U \in \Cal B}$ is a filter base on $K$ and then show that the cluster point is $E$-equivalent to $a$.\newline

\noindent Fix $U \in \Cal B$.
As $K$ is compact there is a finite $B \subseteq K$ such that $(U_b)_{b \in B}$ covers $K$.
It follows that $(U'_b)_{b \in B}$ covers $K'$.
So there is $b \in B$ such that $a \in U'_b$.
As $U'$ is symmetric we have $b \in U'_a$.
So $U'_a \cap K$ is nonempty for every $U \in \Cal B$.
For any $U,V \in \Cal B$ there is $W \in \Cal B$ such that $W \subseteq U \cap V$ so $W'_a \subseteq U'_a \cap V'_a$ and hence
$$ W'_a \cap K \subseteq (U'_a \cap K) \cap (V'_a \cap K) .$$
So $(U'_a \cap K)_{U \in \Cal B}$ is a filter base on $K$.
As $K$ is compact there is a cluster point $p \in K$ of $(U'_a \cap K)_{U \in \Cal B}$.
We show that $a$ and $p$ are $E$-equivalent.
It suffices to fix $U \in \Cal B$ and show that $(a,p) \in U'$.
Let $V \in \Cal B$ be such that $V \circ V = U$.
Then $V_p$ intersects $V'_a \cap K$.
Suppose $q$ lies in this intersection.
Then $(p,q) \in V'$ and $(q,a) \in V'$, so $(p,q) \in U'$.
\end{proof}

\noindent Let $\rho$ be the bijection $O/E \to X$ which maps each $E$-class to the unique element of $X$ that it contains.
We consider $O/E$ as a uniform structure with basis $\pi(\Cal B'|_{O})$.

\begin{lemma}
\label{lem:unif-lem-1}
Suppose $U \in \Cal B$.
Then $a,b \in O/E$ satisfy $(a,b) \in \pi(U')$ if and only if $(\rho(a),\rho(b)) \in U$.
So $O/E$ and $X$ are bi-uniformly equivalent and hence homeomorphic.
\end{lemma}

\begin{proof}
Fix $a,b \in O/E$.
Then $\pi(\rho(a)) = a$ and $\pi(\rho(b)) = b$.
Recall that $(a,b) \in \pi(U')$ if and only if $(c,d) \in U'$ for any $c,d \in X'$ such that $\pi(c) = a$ and $\pi(d) = b$.
So $(a,b) \in \pi(U')$ if and only if $(\rho(a),\rho(b)) \in U'$, which holds if and only if $(\rho(a),\rho(b)) \in U$.
\end{proof}

\noindent In light of Lemma~\ref{lem:unif-lem-1} we identify $X$ with $O/E$ and identify the uniform structure on $X$ with that on $O/E$.
We now view $\pi : O \to X$ as a standard part map.
These identifications greatly simplify notation in the proof of Theorem~\ref{thm:main-cor}. \newline

\noindent So far we have not needed to suppose that $\Cal B$ is a subdefinable basis, only that each element of $\Cal B$ is definable.
We will now use the full assumption.
As $\Cal B'$ is subdefinable basis for $E$, Lemma~\ref{lem:basis-ext-def} shows that $E$ is externally definable in $\Sa N$.
Lemma~\ref{lem:O-ext} shows that $O$ is also externally definable in $\Sa N$ so we regard $X$ as an imaginary sort of $\Sh N$.

\begin{lemma}
\label{lem:O-ext}
The set $O$ is externally definable in $\Sa N$.
\end{lemma}

\begin{proof}
Let $\phi(x,y)$ be a formula such that for every $K \in \Cal K$ there is $a \in M^x$ such that $\phi(a,M^y) = K$.
For each $K \in \Cal K$ select an element $a_K$ of $M^x$ such that $\phi(a_K,M^y) = K$.
It suffices to show that the partial $x$-type
$$ p(x) := \{ \neg \phi(x,b) : b \in X \setminus O \} \cup \{ \forall y [ \phi(a_K, y) \to \phi(x,y)  ] : K \in \Cal K \} $$
is consistent.
Suppose $\Cal C$ is a finite subcollection of $\Cal K$.
Then $\bigcup \Cal C$ is compact.
As $\Cal K$ is a subdefinable compact exhaustion of $X$ there is an $F \in \Cal K$ containing $\bigcup \Cal C$.
Then
$$ \Sa N \models \{ \neg \phi( a_F,b) : b \in X \setminus O \} \cup \{ \forall y [ \phi(a_K,y) \to \phi(a_F,y) : K \in \Cal C \}. $$
So $p(x)$ is consistent.
\end{proof}

\noindent We let $\cl(Y)$ be the closure in $X^n$ of $Y \subseteq X^n$.
It is not a priori clear that $\cl(Y)$ is always definable when $Y$ is definable, Proposition~\ref{prop:subdef-to-def} shows that $\cl(Y)$ is definable in $\Sh M$ when $\Sa M$ is $\nip$.
We let $\Sa X$ be the structure induced on $X$ by $\Sa M$ and as usual let $\Sa X^\circ$ be the open core of the structure induced on $X$ by $\Sa M$, i.e. the structure on $X$ whose primitive $n$-ary relations are all sets of the form $\cl(Y)$ for definable $Y \subseteq X^n$.

\begin{theorem}
\label{thm:main-cor}
Suppose $\Sa M$ is $\nip$.
Then $\Sa X^\circ$ is strongly noiseless.
\end{theorem}

\begin{proof}
Theorem~\ref{thm:main-noise} shows that the structure induced on $X'/E$ by $\Sh N$ is strongly noiseless.
As we identify $X$ with $O/E \subseteq X'/E$ it follows that the structure induced on $X$ by $\Sh N$ is strongly noiseless.
It now show suffices to show that $\Sa X^\circ$ is a reduct of the structure induced on $X$ by $\Sh N$. \newline
 
\noindent We fix an $\Sa M$-definable subset $Y$ of $X^n$ and show that $\cl(Y)$ is definable in $\Sa N^{\text{Sh}}$.
We show that $\cl(Y) = \pi(Y' \cap O^n)$.
As $\pi$ is the identity on $X^n$ we have $Y \subseteq \pi(Y' \cap O^n)$.
Lemma~\ref{lem:closed-image} shows that $\pi(Y' \cap O^n)$ is closed so $\cl(Y) \subseteq \pi(Y' \cap O^n)$.
We prove the other inclusion.
Fix $p \in \pi(Y' \cap O^n)$.
We show that $U_p \cap Y \neq \emptyset$ for all $U \in \Cal B_n$.
Fix $U \in \Cal B_n$.
Let $q \in Y' \cap O^n$ satisfy $\pi(q) = p$.
Then $q \in U'_p$.
So $Y' \cap U'_p$ is nonempty.
Then $Y \cap U_p$ is nonempty as $\Sa M$ is an elementary submodel of $\Sa N$.
\end{proof}

\noindent
We will see below that strong noiselessness has very strong consequences on definable sets and functions in expansion of $(\R,<,+)$.
It is not clear how strong the assumption of strong noiselessness is in more general setttings.
As an application of Theorem~\ref{thm:main-cor} we show that $\Sa X^\circ$-definable functions are generically continuous.

\begin{prop}
\label{prop:generic-cont}
Suppose $\Sa M$ is $\nip$.
Suppose $Y \subseteq X^m$ and $f : Y \to X^n$ are definable in $\Sa X^\circ$.
Then there is a dense definable open subset of $Y$ on which $f$ is continuous.
\end{prop}

\begin{proof}
Let
$$ f(a) = (f_1(a),\ldots,f_n(a)) \quad \text{for all} \quad a \in Y. $$
Suppose that for all $1 \leq k \leq n$ there  is a dense definable open subset $U_k$ of $Y$ on which $f_k$ is continuous.
Then $\bigcap_{k = 1}^{n} U_k$ is a dense definable open subset of $Y$ on which $f$ is continuous.
So we assume $n = 1$. \newline

\noindent After applying Proposition~\ref{prop:subdef-to-def} and replacing $\Sa M$ by $\Sh M$ we may suppose that $\Cal B$ is a definable basis. \newline

\noindent For each $U \in \Cal B$ we let $D_U$ be the set of $a \in Y$ such that for all $V \in \Cal B$ there is $b \in Y$ such that $(a,b) \in V$ and $(f(a),f(b)) \notin U$.
Note that $(D_U)_{U \in \Cal B}$ is a definable family and $\bigcup_{U \in \Cal B} D_U$ is the set of points at which $f$ is discontinuous.
We suppose towards a contradiction that $\bigcup_{U \in \Cal B} D_U$ is somewhere dense in $Y$.
Observe that if $U,V \in \Cal B$ satisfy $U \subseteq V$ then $D_V \subseteq D_U$.
For any $U,U' \in \Cal B$ there is $V \in \Cal B$ such that $V \subseteq U \cap U'$, hence $D_U, D_{U'} \subseteq D_V$.
So $(D_U)_{U \in \Cal B}$ is directed.
We apply Proposition~\ref{prop:sbct-cor} to obtain $U \in \Cal B$ such that $D_U$ is somewhere dense in $Y$.
It follows by strong noiselessness that $D_U$ has nonempty interior in $Y$.
Let $W$ be a nonempty open subset of $Y$ contained in $D_U$.
As $\Cal B$ is a definable basis we may suppose $W$ is definable.\newline

\noindent For each $K \in \Cal K$ let $F_K$ be $f^{-1}(K) \cap D_U$.
Then $(F_K)_{K \in \Cal K}$ is subdefinable.
For any $K,L \in \Cal K$ there is $P \in \Cal K$ such that $K,L \subseteq P$, hence $F_{K}, F_{L} \subseteq F_P$.
So $(F_K)_{K \in \Cal K}$ is directed.
As $\bigcup_{K \in \Cal K} F_K = D_U$  Proposition~\ref{prop:sbct-cor} and strong noiselessness yield a $K \in \Cal K$ such that $F_K$  has interior in $W$. After replacing $W$ with a smaller nonempty definable open set if necessary we suppose that $W$ is contained in $F_K$.
So $f(W) \subseteq K$. \newline

\noindent Fix $V \in \Cal B$ such that $V \circ V \subseteq U$.
As $K$ is compact there is a finite $A \subseteq K$ such that $(V_a)_{a \in A}$ covers $K$.
Then $(f^{-1}(V_a))_{a \in A}$ covers $F_K$ and in particular covers $W$.
So there is $b \in A$ such that $f^{-1}(V_b)$ is somewhere dense in $W$.
Then $f^{-1}(V_b)$ has interior in $W$ by strong noiselessness.
Let $Z$ be the interior of $f^{-1}(V_b)$ in $W$, note $Z$ is definable as $\Cal B$ is a definable basis.
As $p \in D_U$ there is $q \in Z$ such that $(f(p),f(q)) \notin U$.
However as $p,q \in f^{-1}(V_b)$ we have $(f(p), b) \in V$ and $(b,f(q)) \in V$, so $(f(p), f(q)) \in U$ by choice of $V$.
Contradiction.
\end{proof}

\section{Two counterexamples}
\label{section:counterexample}
\noindent It is natural to ask if Theorem~\ref{thm:main-cor} holds when we only have a definable basis for a locally compact Hausdorff topology on $X$.
We show that this is not the case.
In both \ref{subsection:cantor} and \ref{subsection:arrow} we let $I$ be $[0,1]$.

\subsection{The Cantor function}
\label{subsection:cantor}
Let $K$ be the middle-thirds Cantor set and $f : I \to I$ be the Cantor function.
If $t \in K$ and $t = \sum_{i = 1}^{\infty} a_i 3^{-i}$ where $a_i \in \{0,2\}$ for all $i$, then $f(t) = \sum_{i = 1}^{\infty} \frac{1}{2} a_i 2^{-i}$ and if $t \notin K$ then $f(t) = \sup \{ f(s) : s \in K, s < t \}$.
See \cite{Cantor-function} for background.
Continuity of $f$ implies $(\R,<,f)$ and $(\R,<,f)^\circ$ are interdefinable. \newline

\noindent We first show that $(\R,<,f)$ is $\nip$ (in fact dp-minimal).
A subset $X$ of $\R^2$ is monotone if whenever $(s,s') \in X$ and $t \leq s, s' \leq t'$ then $(t,t') \in X$.
It is a special case of \cite[Proposition 4.2]{Simon-dp} that the expansion of $(\R,<)$ by all monotone subsets of $\R^2$ is dp-minimal.
Suppose $h : \R \to \R$ is increasing.
Then $G_h := \{(s,s') \in \R^2 : g(s) \leq s' \}$ is monotone.
It is easy to see that $h$ is definable in $(\R,<,G_h)$.
It follows that the expansion of $(\R,<)$ by all increasing functions $\R \to \R$ is dp-minimal.
As $f(0) = 0$ and $f(1) = 1$, we let $h : \R \to \R$ be given by $h(t) = 0$ when $t < 0$, $h(t) = f(t)$ when $t \in I$, and $h(t) = 1$ when $t > 1$.
Then $h$ is increasing and $(\R,<,h)$ is interdefinable with $(\R,<,f)$.
So $(\R,<,f)$ is dp-minimal. \newline

\noindent It follows from the definition of $f$ that $K$ is the set of $t \in [0,1]$ at which $f$ is not locally constant.
So $K$ is definable in $(\R,<,f)$.
Each connected component of $I \setminus K$ is an open interval, let $L$ be the set of endpoints of connected components of $I \setminus K$.
Note that $L$ is definable in $(\R,<,f)$.
As $f(K) = I$, $L$ is dense in $K$, and $f$ is continuous, we see that $f(L)$ is dense in $I$.
As $L$ is countable $f(L)$ is co-dense in $I$.
So $(\R,<,f)$ is noisey.

\subsection{The double arrow space}
\label{subsection:arrow}
Our second counterexample is definable in $(\R,<)$.
Let $X$ be $I \times \{0,1\}$.
Let $\triangleleft$ be the lexicographic order on $X$, i.e. $(s,t) \triangleleft (s',t')$ if either $s < s'$ or $s = s'$ and $t < t'$.
We equip $X$ with the associated order topology, this is known as the double arrow space or the split interval.
This topology is compact Hausdorff as $\triangleleft$ is complete and has a maximum and minimum.
Any linear order is dp-minimal \cite[Proposition 4.2]{Simon-dp}, so $(X,\triangleleft)$ is dp-minimal. \newline

\noindent Note that the collection of $\triangleleft$-open intervals is a $(X,\triangleleft)$-definable basis for the double arrow space.
Note that $\triangleleft$ is open in the product topology on $X^2$, so $(X,\triangleleft)$ and $(X,\triangleleft)^\circ$ are interdefinable.
Now $I \times \{0\}$ is the set of $a \in X$ such that $\{ b \in X : a \triangleleft b \}$ has a minimum.
(If $a = (t,0)$ then this minimum is $(t,1)$.)
So $I \times \{0\}$ and $I \times \{1\}$ are both definable in $(X,\triangleleft)$.
Finally $I \times \{0\}$ and $I \times \{1\}$ are both dense in $X$.
So $(X,\triangleleft)$ is noisey. \newline

\noindent A similar argument shows that $(K,<)$ is dp-minimal, interdefinable with $(K,<)^\circ$, and noisey, when $K$ is any Cantor subset of $\R$.

\section{Expansions of locally compact groups}
\noindent\textbf{Let $G$ be a group.
Suppose that $\Sa G$ is a first order expansion of $G$, $\Cal C$ is a subdefinable neighbourhood basis at the identity for a locally compact Hausdorff group topology on $G$, and suppose $\Cal K$ is a subdefinable compact exhaustion of $G$.}
If $G$ is compact then we take $\Cal K = \{G\}$. \newline

\noindent 
Recall that the collection $\{ (g,g') \in G^2 : g^{-1}g' \in U \}$, $U \in \Cal C$ is  a subdefinable basis for a uniform structure on $G$ inducing the group topology.
Theorem~\ref{thm:main-cor-group} is a special case of Theorem~\ref{thm:main-cor}.

\begin{theorem}
\label{thm:main-cor-group}
Suppose $\Sa G$ is $\nip$.
Then $\Sa G^\circ$ is strongly noiseless.
\end{theorem}

\noindent
Theorem~\ref{thm:main-cor-group} shows that any $\Sa G^\circ$-definable subset of $G^n$ with empty interior is topologically small.
Under the additional assumption that $G$ is second countable we show that any $\Sa G^\circ$-definable subset of $G^n$ with empty interior is measure-theoretically small.
The key tool is a theorem of Simon~\cite{Simon-measure}.
Proposition~\ref{prop:measure} is already known for expansions of $(\R,<,+)$.
By \cite[Theorem D]{FHW-Compact} any nowhere dense subset of $\R^n$ definable in a type A expansion of $(\R,<,+)$ is Lebesgue null.

\begin{prop}
\label{prop:measure}
Suppose $G$ is second countable and $\Sa G$ is $\nip$.
Then any $\Sa G^\circ$-definable subset of $G^n$ with empty interior is Haar null.
\end{prop}

\noindent
The Haar measure is unique up to rescaling, so the collection of Haar null sets does not depend on the choice of a Haar measure.
An open subset of a locally compact second countable group has non-zero Haar measure, so we see that a subset of $G^n$ definable in $\Sa G^\circ$ is Haar null if and only if it has empty interior.

\begin{proof}
Let $X$ be an $\Sa G^\circ$-definable subset of $G^n$ with empty interior and let $\cl(X)$ be the closure of $X$ in $G^n$.
Then $X$ is nowhere dense, so $\cl(X)$ is nowhere dense.
Corollary~\ref{cor:subdef} shows that $\cl(X)$ is $\Sh G$-definable.
As $\Sh G$ is $\nip$ \cite[Theorem 3.6]{Simon-measure} shows that $\cl(X)$ is Haar null.
So $X$ is Haar null.
\end{proof}

\noindent We finish this section with two examples.\newline

\noindent Suppose $\Sa R$ is a first order expansion of $(\R,<,+)$.
Then $\{ (-t,t) : t > 0 \}$ is a definable neighbourhood basis of zero and $\{ [-t,t] : t > 0 \}$ is a definable compact exhaustion of $\R$.
Theorem~\ref{thm:main-reals} follows from Theorem~\ref{thm:main-cor-group} and Theorem~\ref{thm:equiv}.

\begin{theorem}
\label{thm:main-reals}
Suppose $\Sa R$ is $\nip$.
Then $\Sa R^\circ$ is generically locally o-minimal.
\end{theorem}

\noindent We now discuss structures on the $p$-adics.
Fix a prime $p$.
Let $\overline{\Q}_p$ be the field of $p$-adic numbers.
Let $v_p$ be the $p$-adic valuation on $\Q_p$.
Let $\prec$ be the binary relation on $\Q_p$ where $a \prec b$ if $v_p(a) \geq v(p)$.
It is well-known that $\prec$ is definable in $\overline{\Q}_p$ (see for example \cite[Section 2.1]{Belair-panorama}) so $(\Q_p,+,\prec)$ is a reduct of $\overline{\Q}_p$.
Observe that the collection of sets of the form $\{ a \in \Q_p : a \prec b \},b \in \Q_p$ is both a definable neighbourhood basis at the identity and a definable compact exhaustion of $\Q_p$.

\begin{theorem}
\label{thm:p-adic}
Suppose $\Sa Q$ is a first order expansion of $(\Q_p,+,\prec)$.
If $\Sa Q$ is $\nip$ then $\Sa Q^\circ$ is strongly noiseless.
In particular an $\nip$ expansion of $\overline{\Q}_p$ has strongly noiseless open core.
\end{theorem}

\noindent
We finish this section with two questions on expansions of $\overline{\Q}_p$.
It is shown in \cite[Corollary 2.4]{DG} that if $\Sa R$ is a strongly dependent expansion of $(\R,<,+,\times)$ then $\Sa R^\circ$ is o-minimal (see Theorem~\ref{thm:loc-o-min-field}).
Recall that an expansion of $(\R,<)$ is o-minimal if and only if every definable subset of $\R$ is the union of an open set and a finite set.

\begin{qst}
\label{qst:p-adic}
Suppose $\Sa Q$ is a strongly dependent expansion of $\overline{\Q}_p$.
Must every $\Sa Q^\circ$-definable subset of $\Q_p$ be a union of an open set and a finite set?
\end{qst}

\noindent Question~\ref{qst:p-adic} fails if we only assume $\Sa Q$ is $\nip$.
Mariaule~\cite{Ma-adic} shows that $(\overline{\Q}_p,p^{\Z})$ is $\nip$.
As $p^{\Z}$ has closure $p^{\Z} \cup \{0\}$ we see that $(\overline{\Q}_p,p^{\Z})$ is interdefinable with $(\overline{\Q}_p,p^{\Z})^\circ$.
If Question~\ref{qst:p-adic} admits a positive answer then one can apply Dolich and Goodrick~\cite[Theorem 3.18]{DG-uniform} to get a weak cell decomposition for $\Sa Q^\circ$-definable sets.\newline

\noindent
It is shown in \cite{HW-continuous} that any function $\R^m \to \R^n$ definable in a noiseless expansion of $(\R,<,+)$ is differentiable on a dense open subset of $\R^m$.
Kuijpers and Leenknegt~\cite[Theorem 1.9]{padic-diff} show that any function $\Q_p \to \Q_p$ definable in a P-minimal expansion of the field of $p$-adic numbers is differentiable away from a finite set.
(Recall that P-minimal structures are $\nip$.)

\begin{conj}
\label{conj:p-diff}
Suppose $\Sa Q$ is an $\nip$ expansion of the field of $p$-adic numbers.
Any function $\Q^m_p \to \Q^n_p$ definable in $\Sa Q^\circ$ is differentiable on a dense open subset of $\Q^m_p$.
\end{conj}

\noindent Note that Proposition~\ref{prop:generic-cont} shows that any function $\Q_p^m \to \Q_p^n$ definable in $\Sa Q^\circ$ is continuous on a dense open subset of $\Q^m_p$.
Conjecture~\ref{conj:p-diff} and a positive answer to Question~\ref{qst:p-adic} would together generalize the theorem of Kuijpers and Leenknegt to continuous definable functions in strongly dependent expansions of $\overline{\Q}_p$.

\section{Generic local o-minimality}
\label{section:generic-local}
\noindent
\textbf{Throughout this section we assume $\Sa R$ is a generically locally o-minimal expansion of $(\R,<,+)$ and throughout ``definable" without modification means ``$\Sa R$-definable"}.
By Theorem~\ref{thm:main-reals} all results in this section hold for $\nip$ expansions of $(\R,<,+)$ by closed subsets of Euclidean space.
Conjecture~\ref{conj:main-B} implies that the results of this section hold for any expansion of $(\R,<,+)$ by closed subsets of Euclidean space which does not define an isomorphic copy of $(\Cal P(\N),\N,\in,s)$.\newline

\noindent We will need to recall various results from the general theory of expansions of $(\R,<,+)$, most of which are stated in terms of $\DSig$-sets.
This collection of definable sets was introduced by Dolich, Miller, and Steinhorn~\cite{DMS1}.
A subset of $\R^n$ is $\DSig$ if there is a definable family $(X_{s,t})_{s,t > 0}$ of compact subsets of $\R^n$ such that $X_{s',t} \subseteq X_{s,t}$ when $s' \leq s$, $X_{s,t'} \subseteq X_{s,t}$ when $t \leq t'$, and $X = \bigcup_{s,t > 0} X_{s,t}$.
Every constructible definable set is $\DSig$ and $\DSig$-sets are closed under finite unions, finite intersections, products, and images under continuous definable functions~\cite[1.9,1.10]{DMS1}.
(These facts hold for arbitrary expansions of $(\R,<,+)$.)\newline

\noindent
We refer to Fact~\ref{fact:selection}, proven in \cite[Proposition 6.3]{HNW}, as ``definable selection".
Definable selection can fail for noisey $\nip$ expansions such as $(\R,<,+,\Q)$ (see \cite{HNW} for this example).

\begin{fact}
\label{fact:selection}
Let $X \subseteq \R^m \times \R^n$ be $\emptyset$-definable and let $\pi : \R^m \times \R^n \to \R^m$ be the coordinate projection onto $\R^m$.
Then there is a $\emptyset$-definable function $f : \pi(X) \to \R^n$ such that 
\begin{enumerate}
    \item $(a,f(a)) \in A$ for all $a \in \pi(X)$.
    \item $f(a) = f(b)$ for all $a,b \in \pi(X)$ such that $X_a = X_b$.
\end{enumerate}
So $\Sa R$ admits definable Skolem functions and eliminates imaginaries.
\end{fact}

\noindent Let $X$ be a subset of $\R^n$ and $p \in X$.
We say that $p$ is a \textbf{$\mathbf{C^k}$-point} of $X$ if the following holds (possibly after permuting coordinates): there is $0 \leq d \leq n$, an open subset $V$ of $\R^d$, a $C^k$-function $f : V \to \R^{n-d}$, and an open neighbourhood $U$ of $p$ such that $\gr(f) = U \cap X$.
An application of the inverse function theorem shows that if $k \geq 1$ then $p$ is a $C^k$-point of $X$ if and only if there is an open neighbourhood $U$ of $p$ such that $U \cap X$ is a $C^k$-submanifold of $\R^n$.
The $C^k$-points of $X$ form a $C^k$-submanifold of $\R^n$.
It easy to see that $p$ is a $C^0$-point of $X$ if there is $0 \leq d \leq n$, a coordinate projection $\pi : \R^n \to \R^d$, an open neighbourhood $U$ of $p$, and an open subset $V$ of $\R^d$, such that $\pi$ induces a homeomorphism $U \cap X \to V$.
So Fact~\ref{fact:c0-smooth} shows that if $X$ is definable then the $C^0$-points of $X$ are dense in $X$.\newline

\noindent It is easy to see that the $C^0$-points of $X$ form an $(\R,<,+,X)$-definable set.
It is an open question whether the $C^1$-smooth points of $X$ always form an $(\R,<,+,X)$-definable set.
The $C^\infty$-smooth points of $X$ need not form an $(\R,<,+,\times,X)$-definable set by work of Le Gal and Rolin~\cite{LGR-cellular}.
Fact~\ref{fact:define-Ck-points} is proven in \cite{HW-fractal}.

\begin{fact}
\label{fact:define-Ck-points}
Let $X$ be a subset of $\R^n$ and $k \geq 2$.
Then the set of $C^k$-points of $X$ is $(\R,<,+,X)$-definable.
\end{fact}

\noindent
Let $U$ be an open subset of $\R^m$.
Miller~\cite[Theorem 3.3]{Miller-tame} shows that a function $U \to \R^m$ definable in a noiseless expansion is continuous on a dense definable open subset of $U$.
It is shown in \cite{HW-continuous} that a continuous function $U \to \R^n$ definable in a type A expansion is $C^k$ on a dense definable open subset of $U$ for any $k \geq 1$.
Theorem~\ref{thm:generic-smoothness} follows.

\begin{theorem}
\label{thm:generic-smoothness}
Suppose $U$ is a definable open subset of $\R^m$ and $f : U \to \R^n$ is a definable function.
Fix $k \geq 0$.
Then there is a dense definable open subset $V$ of $U$ such that the restriction of $f$ to $V$ is $C^k$.
\end{theorem}

\noindent
It is an open question whether a function $\R \to \R$ definable in an o-minimal expansion of $(\R,<,+,\times)$ is $C^\infty$ on a dense open subset of $\R$.
We now obtain generic $C^k$-smoothness for definable sets.
It follows from Theorem~\ref{thm:equiv} that Theorem~\ref{thm:Ck-points} fails for expansions of $(\R,<,+)$ which are not generically locally o-minimal, we leave the details to the reader.

\begin{theorem}
\label{thm:Ck-points}
Let $X$ be a definable subset of $\R^n$.
Fix $k \geq 2$.
Then the $C^k$-points of $X$ form a dense definable subset of $X$.
It follows that there is a definable open $W \subseteq \R^n$ such that $W \cap X$ is a $C^k$-submanifold of $\R^n$ and $W$ is dense in $X$.
\end{theorem}

\noindent Let $Y$ be the set of $C^k$-points of $X$.
Theorem~\ref{thm:Ck-points}, Theorem~\ref{thm:equiv-1}, and the frontier inequality for o-minimal structures may be applied to show that if $\Sa R$ is locally o-minimal then $\dim X \setminus Y < \dim X$.
This inequality fails in general when $\Sa R$ is not locally o-minimal.
Suppose $\Sa R$ is not locally o-minimal.
Applying Theorem~\ref{thm:equiv-1} we obtain a definable discrete subset $D$ of $\R_{>0}$ such that $\cl(D) = D \cup \{0\}$.
Then for any $k,m \geq 0$, the set of $C^k$-points of $\cl(D) \times \R^m$ is $D \times \R^m$ and both $D \times \R^m$ and $\{0\} \times \R^m$ are $m$-dimensional.

\begin{proof}
An application of Fact~\ref{fact:define-Ck-points} shows that the set of $C^k$-points of $X$ is definable.
Fix an open subset $U$ of $\R^n$ such that $U \cap X \neq \emptyset$.
We show that $U$ contains a $C^k$-point of $X$.
After applying Fact~\ref{fact:c0-smooth}, permuting coordinates, and shrinking $U$ if necessary we obtain $0 \leq d \leq n$, a definable open $V \subseteq \R^d$, and a definable continuous $f : V \to \R^{n-d}$ such that $\gr(f) = U \cap X$.
Applying Theorem~\ref{thm:generic-smoothness} we obtain a dense definable open subset $V'$ of $V$ such that $f$ is $C^k$ on $V'$.
Fix $p \in V'$.
Then $(p,f(p))$ is a $C^k$-point of $X$. \newline

\noindent We now prove the second claim.
Let $Y$ be the set of $C^k$-points of $X$.
Note that $Y$ is open in $X$.
Let $W$ be the union of all open boxes $B$ in $\R^n$ such that $B \cap Y \neq \emptyset$ and $B \cap X \subseteq Y$.
Then $W$ is definable and $W \cap X = Y$.
\end{proof}

\noindent
We now develop a theory of dimension for definable sets.
We make crucial use of the dimension theory for $D_\Sigma$-sets in type A expansions developed in \cite{FHW-Compact}.
We first recall a classical theorem of Menger and (independently) Uryshon~\cite[1.8.10]{Engelking}.

\begin{fact}
\label{fact:gen-dim}
Suppose $X \subseteq \R^n$.
Then $\dim X = n$ if and only if $X$ has interior.
\end{fact}

\noindent Given a nonempty subset $X$ of $\R^n$ we let $\Dim X$ be the maximal $0 \leq d \leq n$ for which there is a coordinate projection $\pi : \R^n \to \R^d$ such that $\pi(X)$ has interior and let $\Dim \emptyset = -1$.
Fact~\ref{fact:dimcl} holds more generally for noiseless expansions of $(\R,<,+)$.
It is proven in \cite[Section 7]{Miller-tame}.

\begin{fact}
\label{fact:dimcl}
Suppose $X$ is a definable subset of $\R^n$.
Then $\Dim X = \Dim \cl(X)$.
\end{fact}

\noindent
Fact~\ref{fact:dimcl-1} holds more generally for type A expansions \cite[Theorem E]{FHW-Compact}.

\begin{fact}
\label{fact:dimcl-1}
Suppose $X$ is a $\DSig$ subset of $\R^n$.
Then $\dim X = \Dim X$.
\end{fact}

\noindent Proposition~\ref{prop:dim-closure} fails for noisey expansions.
If $X$ is a dense and co-dense subset of a nonempty open interval $I$ then $X$ is zero-dimensional and $\cl(X) = \cl(I)$ is one-dimensional.

\begin{prop}
\label{prop:dim-closure}
Suppose $X$ is a definable subset of $\R^n$.
Then $\dim X = \dim \cl(X)$.
\end{prop}

\begin{proof}
If $X$ is empty then both dimensions are $-1$.
Suppose $X \neq \emptyset$.
Monotonicity of topological dimensions yields $\dim X \leq \dim \cl(X)$.
We prove the other inequality.
Suppose $\dim \cl(X) = d$.
Then $\Dim \cl(X) = d$ by Fact~\ref{fact:dimcl-1} and so $\Dim X = d$ by Fact~\ref{fact:dimcl}.
Let $\pi : \R^n \to \R^d$ be a coordinate projection such that $\pi(X)$ contains a nonempty open subset $U$ of $\R^n$.
After permuting coordinates if necessary and applying definable selection we let $f : U \to \R^{n-d}$ be definable such that $\gr(f) \subseteq X$.
After applying Theorem~\ref{thm:generic-smoothness} and shrinking $U$ if necessary we suppose $f$ is continuous.
Then $\gr(f)$ is homeomorphic to $U$ and so has topological dimension $d$.
Thus $\dim X \geq d$.
\end{proof}

\noindent Corollary~\ref{cor:inject} follows from Proposition~\ref{prop:dim-closure} and its proof.

\begin{cor}
\label{cor:inject}
Suppose $X$ is a nonempty definable subset of $\R^n$.
If $\dim X = d$ then there is a nonempty definable open $U \subseteq \R^d$ and a definable injection $f : U \to X$. 
\end{cor}

\noindent We view Theorem~\ref{thm:dim-eq} as a generalization of Fact~\ref{fact:gen-dim} for definable sets.
Theorem~\ref{thm:dim-eq} fails for general closed subsets of $\R^2$.
For example if $f : [0,1] \to [0,1]$ is the Cantor function, $K$ is the middle-thirds Cantor set, and $G := \{ (t,f(t)) : \in [0,1] \}$, then $\dim G = 0$ and $\Dim G = 1$.
It is an open question whether Theorem~\ref{thm:dim-eq} holds for noiseless expansions of $(\R,<,+)$.

\begin{theorem}
\label{thm:dim-eq}
Suppose $X$ is a definable subset of $\R^n$.
Then $\dim X = \Dim X$.
\end{theorem}

\begin{proof}
Proposition~\ref{prop:dim-closure} shows that $\dim X = \dim \cl(X)$.
As any closed definable set is $\DSig$ an application of Fact~\ref{fact:dimcl-1} shows that $\dim \cl(X)$ and $\Dim \cl(X)$ agree.
Fact~\ref{fact:dimcl} shows $\Dim \cl(X) = \Dim X$.
So $\dim X = \Dim X$.
\end{proof}

\noindent Corollary~\ref{cor:define-dim} follows easily from Theorem~\ref{thm:dim-eq}.
Corollary~\ref{cor:define-dim} may be easily applied to extend our dimension theory to other models of the theory of $\Sa R$.
We leave this to the reader.

\begin{cor}
\label{cor:define-dim}
Suppose $X$ is a definable subset of $\R^m \times \R^n$ and let $-1 \leq d \leq n$.
Then $\{ a \in \R^m : \dim X_a = d \}$ is definable.
\end{cor}

\noindent Corollary~\ref{cor:dim-union} fails if $\Sa R$ is noisey.
If $X$ is a dense and co-dense subset of a nonempty open interval $I$ then $X,I \setminus X$ are both zero-dimensional and $I$ is one-dimensional.

\begin{cor}
\label{cor:dim-union}
Suppose $X$ is a definable set and $\Cal Y$ is a finite partition of $X$ into definable sets.
Then 
$$ \dim X = \max \{ \dim Y : Y \in \Cal Y\}. $$
\end{cor}

\begin{proof}
Monotonicity of topological dimension implies that $\dim X \geq \dim Y$ for all $Y \in \Cal Y$.
We prove the other inequality.
Suppose $\dim X = d$.
By Theorem~\ref{thm:dim-eq} there is a coordinate projection $\pi : \R^n \to \R^d$ such that $\pi(X)$ has interior.
As $\{ \pi(Y) : Y \in \Cal Y\}$ covers $\pi(X)$ there is a $Y \in \Cal Y$ such that $\pi(Y)$ is somewhere dense in $\R^d$.
Then $\pi(Y)$ has interior as $\Sa R$ is noiseless.
So $\dim Y \geq d$ by Theorem~\ref{thm:dim-eq}.
\end{proof}

\begin{prop}
\label{prop:dim-smooth}
Let $X$ be a definable subset of $\R^n$.
Fix $k \geq 0$.
Then $\dim X$ is the maximal $d$ for which there is a definable open subset $U$ of $\R^n$ such that $U \cap X$ is a $d$-dimensional $C^k$-submanifold of $\R^n$.
In particular $\dim X$ is the maximal $d$ for which there is a definable open subset $U$ of $\R^n$ such that $X$ is closed in $U$ and $\dim U \cap X = d$.
\end{prop}

\begin{proof}
The second claim follows directly from the first.
If there is an open $U \subseteq \R^n$ such that $U \cap X$ is a $d$-dimensional $C^k$-submanifold then $\dim X \geq d$ by monotonicity of topological dimension.
We prove the other inequality.
Suppose $\dim X = d$.
Let $l = \max \{ k, 2 \}$.
Applying Theorem~\ref{thm:Ck-points} we obtain a definable open subset $U$ of $\R^n$ such that $U \cap X$ is a $C^l$-submanifold of $\R^n$ and $U$ is dense in $X$.
As $\cl(U \cap X) = X$ an application of Proposition~\ref{prop:dim-closure} shows that $U \cap X$ is  $d$-dimensional.
\end{proof}

\noindent Theorem~\ref{thm:space-filling} fails for noisey expansions.
If $X$ is a dense and co-dense subset of a nonempty open interval $I$ then $Y := [ \{0\} \times X] \cup [\{1\} \times (I \setminus X)] $ is zero-dimensional and the projection of $Y$ onto the second coordinate is one-dimensional.

\begin{theorem}
\label{thm:space-filling}
Let $X$ be a definable subset of $\R^m$ and suppose $f : X \to \R^n$ is definable.
Then $\dim f(X) \leq \dim X$.
In particular dimension is preserved by definable bijections.
\end{theorem}

\begin{proof}
The second claim follows easily from the first so we only prove the first claim.
Suppose $\dim f(X) = d$.
Theorem~\ref{thm:dim-eq} yields a coordinate projection $\pi : \R^n \to \R^d$ such that $\pi(f(X))$ has interior.
After replacing $f$ with $\pi \circ f$ we suppose that $f(X)$ has interior.
Let $U \subseteq X$ be a definable open subset of $\R^n$.
Applying definable selection we obtain a definable $g : U \to X$ such that $f(g(p)) = p$ for all $p \in U$.
Note that $g$ is injective.
After shrinking $U$ if necessary and applying Theorem~\ref{thm:generic-smoothness} we suppose $g$ is continuous.
Let $K$ be a compact subset of $U$ with nonempty interior.
Then $\dim K = d$.
As $g$ is a continuous injection and $K$ is compact, $g$ gives a homeomorphism $K \to g(K)$.
Then $\dim g(K) = d$ and so $\dim X \geq d$.
\end{proof}

\noindent We now proceed towards the proof of Theorem~\ref{thm:fiber}. \newline

\noindent We recall two results on the dimension of $D_\Sigma$-sets in type A expansions of $(\R,<,+)$.
If $X \subseteq \R^m \times \R^n$ is $\DSig$ then $\{ a \in \R^m : \dim X_a \geq d \}$ is always $\DSig$ \cite[Fact 2.9]{FHW-Compact} but $\{ a \in \R^m : \dim X_a = d \}$ need not be $\DSig$ (the complement of $\DSig$-set may not be $\DSig$).
Fact~\ref{fact:fiber-dsigma} holds in any type A expansion \cite[Theorem 6.4, Corollary 6.7]{FHW-Compact}.

\begin{fact}
\label{fact:fiber-dsigma}
Let $X$ be a $\DSig$-subset of $\R^m \times \R^n$ and let $\pi$ be the coordinate projection $\R^m \times \R^n \to \R^m$.
Then there is $0 \leq d \leq m$ such that
$$ \dim X \leq \{ a \in \R^m : \dim X_a \geq d \} + d. $$
If there is a $0 \leq d \leq n$ such that $\dim X_a = d$ for all $a \in \pi(X)$ then
$$ \dim X = \dim \pi(X) + d. $$
\end{fact}

\begin{theorem}
\label{thm:fiber-projection}
Let $X$ be a definable subset of $\R^m \times \R^n$.
Then 
$$ \dim X = \max_{0 \leq d \leq n} \dim \{ a \in \R^m : \dim X_a = d \} + d .$$
In particular
$$ \dim Y \times Y' = \dim Y + \dim Y' $$
for definable $Y \subseteq \R^m$, $Y' \subseteq \R^n$.
\end{theorem}

\begin{proof}
Let $\pi$ be the coordinate projection $\R^m \times \R^n \to \R^m$.
We first show that 
$$ \dim X \geq  \dim \{ a \in \R^m : \dim X_a = d \} + d \quad \text{for all} \quad 0 \leq d \leq n .$$
Fix $0 \leq d \leq n$.
After replacing $X$ by
$$[ \{ a \in \R^m : \dim X_a = d\} \times \R^{n} ] \cap X $$ 
we suppose $\dim X_a = d$ for all $a \in \pi(X)$.
So we show $\dim X$ is at least $\dim \pi(X) + d$.
We show that $\dim \cl(X)$ satisfies the same inequality and apply Proposition~\ref{prop:dim-closure}.
Note that $\dim \cl(X)_a \geq d$ for all $a \in \pi(X)$.
For each $d \leq k \leq n$ let $Z[k]$ be the set of $a \in \pi(X)$ such that $\dim \cl(X)_a = k$.
Applying Corollary~\ref{cor:dim-union} we obtain $d \leq k \leq n$ such that $\dim Z[k] = \dim \pi(X)$.
Applying Proposition~\ref{prop:dim-smooth} we obtain a definable open subset $U$ of $\R^m$ such that $U \cap Z[k]$ is constructible and has dimension $\dim \pi(X)$.
In particular $U \cap Z[k]$ is $D_\Sigma$.
It follows that $Y := [ (U \cap Z[k]) \times \R^{n}]\cap \cl(X)$ is also $D_\Sigma$.
Applying the second claim of Fact~\ref{fact:fiber-dsigma} we obtain
$$ \dim Y = \dim (U \cap Z[k]) + k = \dim \pi(X) + k \geq \dim \pi(X) + d. $$
As $Y$ is a subset of $\cl(X)$ we have $\dim \cl(X) \geq \dim \pi(X) + d$.\newline

\noindent We now show that there is $0 \leq d \leq n$ such that
$$ \dim X \leq \dim \{ a \in \R^m : \dim X_a = d \} + d .$$
Applying Proposition~\ref{prop:dim-smooth} we obtain a definable open subset $U$ of $\R^m \times \R^n$ such that $Q := U \cap X$ is constructible and has the same dimension as $X$.
So $Q$ is $\DSig$.
Applying Fact~\ref{fact:fiber-dsigma} we obtain $0 \leq e \leq n$ such that
$$ \dim Q \leq \dim \{ a \in \R^m : \dim Q_a \geq e \} + e \leq \dim \{ a \in \R^m : \dim X_a \geq e \} + e. $$
By Corollary~\ref{cor:dim-union} there is $e \leq d \leq n$ such that
$$ \dim \{ a \in \R^n : \dim X_a = d \} = \dim \{ a \in \R^n : \dim X_a \geq e\}. $$
So
$$ \dim X \leq \dim \{ a \in \R^n : \dim X_a = d \} + d. $$
\end{proof}

\noindent We now prove Theorem~\ref{thm:fiber}.

\begin{theorem}
\label{thm:fiber}
Let $X$ be a definable subset of $\R^n$ and $f$ be a definable map $X \to \R^m$.
Then
$$ \dim X = \max_{0 \leq d \leq n} \dim \{ a \in \R^m : \dim f^{-1}(\{a\}) = d \} + d .$$
\end{theorem}

\begin{proof}
Let $G$ be the set of $(a,b) \in \R^m \times \R^n$ such that $f(b) = a$.
Then $G_a = f^{-1}(\{a\})$ for all $a \in \R^m$.
Theorem~\ref{thm:fiber-projection} shows that
$$ \dim G = \max_{0 \leq d \leq n} \dim \{a \in \R^m : \dim f^{-1}(\{a\}) = d \} + d. $$
The coordinate projection $\R^m \times \R^n \to \R^n$ gives a definable bijection $G \to X$ so $\dim G = \dim X$ by Theorem~\ref{thm:space-filling}.
\end{proof}

\noindent As an application of the dimension theory we prove Lemma~\ref{lem:k-w}, which is used in the proof of Theorem~\ref{thm:d-min}.

\begin{lemma}
\label{lem:k-w}
Let $X,Y$ be nonempty definable subsets of $\R^m \times \R^n$ and let $\pi : \R^m \times \R^n \to \R^m$ be the coordinate projection.
Suppose there is a nonempty open subset $U$ of $\pi(X)$ such that $Y_a$ has interior in $X_a$ for all $a \in U$.
Then $Y$ has interior in $X$.
\end{lemma}

\begin{proof}
We first reduce to the case when $U$ is an open subset of $\R^m$.
After applying Fact~\ref{fact:c0-smooth} and Theorem~\ref{thm:equiv} and shrinking $U$ if necessary we obtain $0 \leq d \leq m$, a nonempty definable open $V \subseteq \R^d$, and a coordinate projection $\rho : \R^m \to \R^d$ such that $\rho$ restricts to a homeomorphism $U \to V$.
Let $g$ be the definable homeomorphism  $U \times \R^n \to V \times \R^n$ given by $g(a,b) = (\rho(a),b)$.
It suffices to show that $g(Y)$ has interior in $g(X)$.
So we suppose that $U$ is an open subset of $\R^m$. \newline

\noindent For all $a \in U$ let $W_a$ be the union of all open boxes $B$ in $\R^n$ such that $B \cap X_a \subseteq Y_a$ and $B \cap X_a \neq \emptyset$.
So $W_a \cap X_a$ is the interior of $Y_a$ in $X_a$ for all $a \in U$.
Let $W$ be the set of $(a,b) \in \R^{m} \times \R^{n}$ such that $b \in W_a$.
Note that $W$ is definable and each $W_a$ is nonempty and open.
So $\dim W_a = n$ for all $a \in U$.
As $\dim U = m$ Theorem~\ref{thm:fiber-projection} shows that $\dim W = m + n$.
Fact~\ref{fact:gen-dim} shows that $W$ has interior in $\R^m \times \R^n$.
As $W \cap X \subseteq Y$ we see that $Y$ has interior in $X$.
\end{proof}

\noindent We briefly describe some results on metric dimensions.
It follows from \cite[Corollary 1.5]{HM} that if $\Sa R$ expands $(\R,<,+,\times)$ and $X \subseteq \R^n$ is definable then the Assouad dimension of $X$ agrees with $\dim X$.
We prove a similar result for expansions of $\rvec$.
Fact~\ref{fact:HD} is the main theorem of \cite{FHM}.

\begin{fact}
\label{fact:HD}
Suppose $\Sa S$ is a type A expansion of $\rvec$ and let $X \subseteq \R^n$ be $\DSig$.
Then the Hausdorff dimension of $X$ agrees with $\dim X$.
\end{fact}

\noindent We let $\dimH X$ be the Hausdorff dimension of a subset $X$ of $\R^n$.

\begin{prop}
\label{prop:vec-dim}
Suppose $\Sa R$ expands $\rvec$ and let $X \subseteq \R^n$ be definable.
Then the Hausdorff dimension of $X$ agrees with $\dim X$.
\end{prop}

\begin{proof}
Hausdorff dimension is monotone and bounded below by topological dimension, so $\dim X \leq \dimH X \leq \dimH \cl(X)$.
The proposition follows by combining Fact~\ref{fact:HD} and Proposition~\ref{prop:dim-closure}.
\end{proof}

\noindent It is an open question whether Proposition~\ref{prop:vec-dim} holds for $\nip$ expansions of $(\R,<,+)$ by closed sets.
As as countable set has Hausdorff dimension zero this is closely related to Question~\ref{qst:cantor-rank}.

\subsection{Theorem E}
\label{section:zilber}
We continue to suppose that $\Sa R$ is a generically locally o-minimal expansion of $(\R,<,+)$.
We apply the tools developed above to develop a ``Zil'ber dichotomy" between generically locally o-minimal expansions that define a field structure on an interval and those in which definable sets and functions are generically affine.
The o-minimal case follows from the Peterzil-Starchenko trichotomy~\cite{PS-Tri}.
A $\DSig$-version of Theorem~\ref{thm:zilber} is proven in \cite{HW-continuous} for type A expansions.
Our argument follows that of \cite{HW-continuous} but uses different tools. \newline

\noindent
We say that $\Sa R$ is \textbf{field-type} if there is an open interval $I$ and definable $\oplus,\otimes : I^2 \to I$ such that $(I,<,\oplus,\otimes)$ is an ordered field isomorphic to $(\R,<,+,\times)$.
This definition implies $\oplus,\otimes$ are continuous.
It is easy to see that if $\Sa R$ is field-type then there is a bounded open interval $I$ and $\oplus,\otimes : I^2 \to I$ such that $(I,<,\oplus,\otimes)$ is isomorphic to $(\R,<,+,\times)$.
So we will always suppose that $I$ is bounded.
We recall a general theorem from \cite{HW-continuous}.
(The case when $m = 1$ is essentially due to Marker, Peterzil, and Pillay~\cite{MPP}.)

\begin{fact}
\label{fact:gen-field-type}
Suppose $U$ is a nonempty connected open subset of $\R^m$ and $f : U \to \R^n$ is $C^2$ and not affine.
Then $(\R,<,+,f)$ is field-type.
\end{fact}

\noindent A family $\Cal X$ of one-dimensional subsets of $\R^n$ is \textbf{normal} if $\dim X \cap X' = 0$ for all distinct $X,X' \in \Cal X$.
If $( X_a )_{a \in B}$ is a normal definable family of one-dimensional subsets of $\R^n$ then the dimension of $( X_a )_{a \in B}$ is defined to be the dimension of $B$.
Given a subset $X$ of $\R^n$ we say that $p \in X$ is an \textbf{affine point} of $X$ if there is an open neighbourhood $U$ of $p$ such that $U \cap X = U \cap H$ for some affine subspace $H$ of $\R^n$.
It is shown in \cite{HW-fractal} that the set of affine points of $X$ is definable in $(\R,<,+,X)$.

\begin{theorem}
\label{thm:zilber}
The following are equivalent
\begin{enumerate}
    \item $\Sa R$ is field-type,
    \item there is a definable field $(Z,\oplus,\otimes)$ such that $\dim Z > 0$,
    \item there is a definable normal family of one-dimensional subsets of $\R^n$ of dimension $\geq 2$,
    \item there is a definable $f : U \to \R^m$ where $U$ is an open subset of $\R^n$ and $f$ is nowhere locally affine,
    \item there is a definable subset $X$ of $\R^n$ such that the affine points of $X$ are not dense in $X$.
\end{enumerate}
\end{theorem}

\noindent
An o-minimal expansion of $(\R,<,+)$ is field-type if and only if it is not a reduct of $\rvec$ by the Peterzil-Starchenko trichotomy.
This dichotomy fails for $\nip$ expansions of $(\R,<,+)$ by closed subsets of Euclidean space.
Let $g : 2^{\Z} \times \R \to \R$ be given by $f(t,t') = tt'$.
Delon~\cite{Delon-powers} studied $(\R,<,+,g)$.
It may be deduced from \cite[Theorem 2]{Delon-powers} that $(\R,<,+,g)$ is not field-type.
However $(\R,<,+,g)$ is $\nip$ as it is a reduct of $(\R,<,+,\times,2^\Z)$.


\begin{proof}
A nonempty open interval is one-dimensonal so $(1)$ implies $(2)$.\newline

\noindent $(2) \Rightarrow (3)$:
Let $(Z,\oplus,\otimes)$ be a definable field with $\dim Z \geq 1$.
Suppose $Z$ is a subset of $\R^n$.
Applying Corollary~\ref{cor:inject} we obtain a nonempty open interval $I$ and a definable injection $g : I \to Z$.
Given $(a,a') \in I^2$ we let
$$ X_{(a,a')} := \{ (b,b') \in I \times Z : b' = [ g(a) \otimes g(b) ] \oplus g(a') \}. $$
Note that $X_{(a,a')}$ is the graph of the function $I \to Z$ given by
$$ x \mapsto [g(a) \otimes g(x)] \oplus g(a'). $$
This function is injective so Theorem~\ref{thm:space-filling} shows that $\dim X_{(a,a')} = 1$ for all $(a,a') \in I^2$.
We show that $(X_{(a,a')})_{(a,a') \in I^2}$ is normal.
Let $(a,a')$ and $(b,b')$ be distinct elements of $I^2$.
Then $(g(a),g(a))$ and $(g(b),g(b'))$ are distinct as $g$ is injective.
Let $L_{(a,a')}$ be the graph of the function $Z \to Z$ given by $x \mapsto [g(a) \otimes x] \oplus g(a')$, likewise define $L_{(b,b')}$.
Then $X_{(a,a')}$ and $X_{(b,b')}$ are contained in $L_{(a,a')}$ and $L_{(b,b')}$, respectively.
As $(Z,\oplus,\otimes)$ is a field the intersection of $L_{(a,a')}$ and $L_{(b,b')}$ contains at most one element.
So $X_{(a,a')} \cap X_{(b,b')}$ contains at most one element.\newline

\noindent $(3) \Rightarrow (4)$:
Suppose $(X_a)_{a \in B}$ is a definable normal family of one-dimensional subsets of $\R^n$ where $B$ is a definable subset of $\R^m$ of dimension $\geq 2$.
Applying Corollary~\ref{cor:inject} we obtain a nonempty definable open $U \subseteq \R^2$ and a definable injection $f : U \to B$.
After replacing $(X_a)_{a \in B}$ with $(X_{f(a)})_{a \in U}$ if necessary we suppose $B$ is an open subset of $\R^2$.\newline

\noindent Let $\pi_k : \R^n \to \R$ be the projection onto the $k$th coordinate for $1 \leq k \leq n$.
As $\dim X_a = 1$ for all $a \in B$ an application of Theorem~\ref{thm:dim-eq} shows that for any $a \in B$ there is $1 \leq k \leq n$ such that $\pi_k(X_a)$ has interior.
Given $1 \leq k \leq n$ let $B_k$ be the set of $a \in B$ such that $\pi_k(X_a)$ has interior.
Then there is $1 \leq k \leq n$ such that $B_k$ is somewhere dense in $B$.
We assume that $B_1$ is somewhere dense in $B$, the case when $B_k$ is somewhere dense for $2 \leq k \leq n$ follows in the same way.
So $B_1$ has interior in $B$.
After replacing $B$ with a smaller nonempty definable open subset of $\R^2$ if necessary we suppose $\pi_1(X_a)$ has interior for all $a \in B$.
Let $\rho : B \times \R^n \to B \times \R$ be the projection onto the first three coordinates.
Then $\rho(X)_a = \pi_1(X_a)$ for all $a \in B$.
As $\rho(X)_a$ has interior for all $a \in B$ an application of Lemma~\ref{lem:k-w} shows that $\rho(X)$ has interior.
Let $V$ be a definable open subset of $\R^3$ contained in $\rho(X)$.
Applying definable selection we obtain a definable $g : V \to \R^{n-1}$ such that $(c,g(a,b,c)) \in X_{(a,b)}$ for all $(a,b,c) \in V$.\newline

\noindent We show that $g$ is nowhere locally affine.
Suppose otherwise.
After shrinking $V$ if necessary we suppose $g$ is affine and $V$ is of the form  $I_1 \times I_2 \times I_3$ for nonempty open intervals $I_1,I_2,I_3$.
Fix linear $h_1,h_2,h_3 : \R \to \R^{n-1}$ and $\beta \in \R^{n-1}$ such that
$$ g(a_1,a_2,a_3) = h_1(a_1) + h_2(a_2) + h_3(a_3) + \beta \quad \text{for all}  \quad (a_1,a_2,a_3) \in I_1 \times I_2 \times I_3. $$
Fix $(u,v) \in I_1 \times I_2$ and $u' \in I_1$ such that $u' \neq u$.
Let $v' \in I_2$ satisfy 
\[
 h_1(u')  +h_2(v') = h_1(v) + h_2(v).
 \]
Then $f(u,v,t) = f(u',v',t)$ for all $t \in I_3$.
So $\{ (t,f(u,v,t)) : t \in I_3 \}$ is a subset of both $X_{(u,v)}$ and $X_{(u',v')}$.
As $\{ (t,f(u,v,t)) : t \in I_3 \}$ is one-dimensional we obtain a contradiction with  our assumption that $(X_a)_{a \in B}$ is normal.\newline

\noindent $(4) \Rightarrow (5)$:
First observe that if $V$ is an open subset of $\R^n$ and $f : V \to \R^m$ is continuous then $f$ is locally affine at $p \in V$ if and only if $(p,f(p))$ is an affine point of $\gr(f)$.
Now suppose $U$ is a definable open subset of $\R^n$ and $f : U \to \R^m$ is definable such that $f$ is nowhere locally affine.
After applying Theorem~\ref{thm:generic-smoothness} and shrinking $U$ if necessary we suppose $f$ is continuous.
Then $\gr(f)$ does not have any affine points.\newline

\noindent $(5) \Rightarrow (1)$:
Suppose $X \subseteq \R^n$ is definable and suppose that the affine points of $X$ are not dense in $X$.
Applying Theorem~\ref{thm:Ck-points} we obtain a definable open subset $U$ of $\R^n$ and $0 \leq d \leq n$ such that $U \cap X$ is a $d$-dimensional $C^2$-submanifold of $\R^n$.
After permuting coordinates and shrinking $U$ if necessary we obtain a nonempty open $V \subseteq \R^d$ and a $C^2$ definable $f : V \to \R^{n-d}$ such that $U \cap X = \gr(f)$.
We also may suppose that $U$ and $V$ are both connected.
Our assumption on $X$ implies that $f$ is not affine.
Fact~\ref{fact:gen-field-type} shows that $\Sa R$ is field-type.
\end{proof}

\noindent In the d-minimal case we get a sharper version of Theorem~\ref{thm:zilber}.

\begin{prop}
\label{prop:zilber-d-min}
Suppose $\Sa R$ is d-minimal.
Then the following are equivalent.
\begin{enumerate}
    \item $\Sa R$ does not interpret $(\R,<,+,\times)$,
    \item $\Sa R$ is not field type,
    \item if $X$ is a definable subset of $\R^n$ then there are definable locally affine subsets $Y_0,\ldots,Y_k$ of $\R^n$ such that $X = Y_0 \cup \ldots \cup Y_k$, $Y_0$ is open and dense in $X$, $Y_i$ is open and dense in $X \setminus (Y_0 \cup \ldots \cup Y_{i  - 1})$ for all $1 \leq i \leq k$.
    \item every definable subset of every $\R^n$ is a finite union of locally affine sets.
\end{enumerate}
\end{prop}

\noindent Note that Delon's structure $(\R,<,+,g)$ discussed after the statement of Theorem~\ref{thm:zilber} is d-minimal as it is a reduct of $(\R,<,+,\times,2^{\Z})$.

\begin{proof}
It is clear that $(1)$ implies $(2)$.
\newline

\noindent $(2) \Rightarrow (3):$
Suppose $X$ is a definable subset of $\R^n$.
We apply induction on the Pillay rank of $\cl(X)$.
If $\pr(\cl(X)) = -1$ then $X$ is empty and hence locally affine.
Suppose $\pr(\cl(X)) \geq 1$.
Let $Y_0$ be the set of affine points of $X$.
Then $Y_0$ is locally affine, definable, and open in $X$.
Theorem~\ref{thm:zilber} shows that $Y_0$ is dense in $X$ so $X \setminus Y_0$ is nowhere dense in $X$.
It follows that $\cl(X \setminus Y_0)$ is nowhere dense in $\cl(X)$ so $\pr(\cl(X \setminus Y_0)) < \pr(\cl(X))$.
Applying induction we obtain definable locally affine sets $Y_1,\ldots,Y_k$ such that $X \setminus Y_0 = \bigcup_{i = 1}^{k} Y_i$ and $Y_i$ is open and dense in $X \setminus (Y_0 \cup \ldots \cup Y_{i - 1})$ for all $1 \leq i \leq k$.\newline

\noindent 
It is clear that $(3)$ implies $(4)$.
We show that $(4)$ implies $(2)$.
Suppose $\Sa R$ is field type.
Applying Theorem~\ref{thm:zilber} we obtain a nonempty definable open subset $U$ of $\R^m$ and a definable $f : U \to \R^n$ which is nowhere locally affine.
After applying Theorem~\ref{thm:generic-smoothness} we suppose $f$ is continuous so $\gr(f)$ is homeomorphic to $U$.
We show that $\gr(f)$ is not a finite union of locally affine sets.
A locally affine subset of $\R^{m + n}$ is contained in a countable union of affine subspaces of $\R^{m + n}$, so a finite union of locally affine subsets is contained in a countable union of affine subspaces.
So it suffices to suppose $\{H_i\}_{i \in \N}$ is a sequence of affine subspaces of $\R^{m+n}$ and show that $\bigcup_{i \in \N} H_i$ does not contain $\gr(f)$.
If $(p,f(p)) \in \gr(f)$ lies in the interior of $H_i$ in $\gr(f)$ then $f$ is locally affine at $p$.
So each $H_i$ has empty interior in $\gr(f)$.
As $H_i$ is closed it follows that each $H_i$ is nowhere dense in $\gr(f)$.
As $\gr(f)$ is homeomorphic to $U$ an application of the Baire category theorem shows that $\bigcup_{i \in \N} H_i$ has empty interior in $\gr(f)$.
\newline

\noindent
$(2) \Rightarrow (1):$
Suppose $\Sa R$ is not field-type.
By Fact~\ref{fact:selection} it suffices to show that $\Sa R$ does not define an isomorphic copy of $(\R,<,+,\times)$.
Suppose $(Z,\oplus,\otimes)$ is a definable field.
Theorem~\ref{thm:zilber} shows that $\dim Z = 0$.
Applying Fact~\ref{fact:d-min-cntble} we see that $Z$ is countable so $(Z,\oplus,\otimes)$ is not isomorphic to $(\R,+,\times)$.
\end{proof}

\noindent
If $\Sa R$ is o-minimal then a definable set is zero-dimensional if and only if it is finite.
In this case condition $(2)$ of Theorem~\ref{thm:zilber} is equivalent to the assumption that $\Sa R$ does not define an infinite field.
The expansion of $(\R,<,+)$ by \textit{all} subsets of \textit{all} $\Z^n$ is locally o-minimal and not field-type and obviously defines an isomorphic copy of any countable field, see Fact~\ref{fact:ktt-converse}.

\begin{conj}
\label{conj:field}
Suppose $\Sa R$ is an $\nip$ expansion of $(\R,<,+)$ by closed subsets of Euclidean space.
If $\Sa R$ interprets an infinite field then $\Sa R$ is field-type.
\end{conj}

\noindent By Fact~\ref{fact:selection} we may replace ``interprets" with ``defines" in this conjecture.
Conjecture~\ref{conj:field} fails for noisey $\nip$ expansions.
Let $\Sa F$ be an $\nip$ (strongly dependent) field of cardinality $\leq 2^{\aleph_0}$.
Then there is a noisey $\nip$ (strongly dependent) expansion $\Sa R$ of $(\R,<,+)$ such that $\Sa R$ interprets $\Sa F$ and $\Sa R^\circ$ is interdefinable with $(\R,<,+)$ \cite{HNW}.
(Note that an expansion $\Sa R$ of $(\R,<,+)$ is field-type if and only if $\Sa R^\circ$ is field type.) \newline

\noindent Proposition~\ref{prop:char-p} handles the positive characteristic case of Conjecture~\ref{conj:field-1}.
This proposition was observed for noiseless $\nip$ expansions in \cite{HNW}, we include a proof for the sake of completeness.

\begin{prop}
\label{prop:char-p}
Suppose $\Sa R$ is an $\nip$ expansion of $(\R,<,+)$ by closed subsets of Euclidean space.
Then $\Sa R$ does not interpret an infinite vector space over a finite field.
In particular $\Sa R$ does not interpret an infinite field of positive characteristic.
\end{prop}

\begin{proof}
Fix a finite field $\mathbb{F}$.
Let $\Sa V$ be an infinite $\mathbb{F}$-vector space.
Suppose $\Sa V$ is interpretable in $\Sa R$.
By Fact~\ref{fact:selection} we may suppose that $\Sa V$ is definable in $\Sa R$.
Suppose the domain of $\Sa V$ is a subset of $\R^m$.
Let $\prec$ be the lexicographic order on $\R^m$.
Shelah and Simon~\cite[Theorem 2.1]{SS-linear} show that any expansion of an infinite $\mathbb{F}$-vector space by a linear order is $\mathrm{IP}$.
So $(\Sa V,\prec)$ is $\mathrm{IP}$, contradiction.
\end{proof}

\noindent We prove several more special cases of Conjecture~\ref{conj:field}.
In particular we prove the strongly dependent case in Proposition~\ref{prop:conj-case-strong}.
We do this by establishing special cases of a stronger conjecture, Conjecture~\ref{conj:field-1}.
Theorem~\ref{thm:zilber} shows that Conjecture~\ref{conj:field} follows from Conjecture~\ref{conj:field-1}.

\begin{conj}
\label{conj:field-1}
An $\nip$ expansion of $(\R,<,+)$ by closed subsets of Euclidean space cannot define a field $(Z,\oplus,\otimes)$ where $Z$ is infinite and zero-dimensional.
\end{conj}

\noindent We will need Proposition~\ref{prop:presburger}.

\begin{prop}
\label{prop:presburger}
No infinite field is interpretable in $(\Z,<,+)$.
\end{prop}

\begin{proof}
We are aware of two approaches to Proposition~\ref{prop:presburger}.
As $(\Z,<,+)$ eliminates imaginaries it suffices to show that $(\Z,<,+)$ does not define an infinite field. 
Onshuus and Vicaria~\cite{Onshuus-pres} show that any group definable in $(\Z,<,+)$ has an abelian subgroup of finite index.
It is basic fact from algebra that if $K$ is an infinite field then $\mathrm{Sl}_2(K)$ does not have a finite index abelian subgroup.
Alternatively, as $(\Z,<,+)$ is dp-minimal any structure definable in $(\Z,<,+)$ has finite dp rank, Dolich and Goodrick~\cite[Corollary 2.2]{DG} show that an expansion of a field which does not eliminate $\exists^\infty$ has infinite dp rank, and it is easy to see that the structure induced on any infinite subset of $\Z^n$ by $(\Z,<,+)$ does not eliminate $\exists^\infty$.
\end{proof}

\noindent Proposition~\ref{prop:dis-group} and Theorem~\ref{thm:zilber} yield Conjecture~\ref{conj:field} for reducts of $\slam$.
In particular the structure $(\R,<,+,g)$ discussed after Theorem~\ref{thm:zilber} does not interpret an infinite field.

\begin{prop}
\label{prop:dis-group}
Suppose $\Sa S$ is an o-minimal expansion of $(\R,<,+,\times)$ with rational exponents and suppose $\lambda > 1$.
Then $(\Sa S, \lambda^\Z)$ does not define an infinite zero-dimensional field.
\end{prop}

\begin{proof}
Tychonievich~\cite[4.1.10]{Tychon-thesis} shows that if $X \subseteq \R^n$ is zero-dimensional and definable in $(\Sa S, \lambda^\Z)$ then there is an $\slam$-definable surjection $(\lambda^{\Z})^m \to X$ for some $m$.
Tychonievich~\cite[4.1.2]{Tychon-thesis} also shows that the structure induced on $\lambda^\Z$ by $\slam$ is interdefinable with $(\lambda^\Z,<,\times)$.
As $(\lambda^\Z,<,\times)$ is isomorphic to $(\Z,<,+)$ the proposition follows from Proposition~\ref{prop:presburger}.
\end{proof}

\noindent Similar approaches should yield Conjecture~\ref{conj:field} for other expansions, we briefly describe one example.
Fix $\lambda > 1$, let $D$ be $\{ \lambda, \lambda^\lambda,\lambda^{\lambda^\lambda},\ldots \}$, and let $\Sa S$ be any o-minimal expansion of $(\R,<,+,\times)$ such that every $\Sa S$-definable function $\R \to \R$ is eventually bounded above by some compositional iterate of the exponential (every known o-minimal expansion satisfies this condition).
Miller and Tyne~\cite{Miller-iteration} show that $(\Sa S,D)$ admits quantifier elimination in a natural language and is d-minimal.
It should follow that $(\Sa S,D)$ is $\nip$, every zero-dimensional $(\Sa S,D)$-definable set is a definable image of some $D^n,$ and the induced structure on $D$ is interdefinable with $(D,<)$.
The proof of Proposition~\ref{prop:dis-group} would then show that $(\Sa S,D)$ does not define an infinite zero-dimensional field.\newline

\noindent We describe two closely related conjectures which should be more approachable.
Conjecture~\ref{conj:omega-order} is a special case of Conjecture~\ref{conj:directed}.






\begin{conj}
\label{conj:omega-order}
Suppose $(\N,\oplus,\otimes)$ is a field.
Then $(\N,<,\oplus,\otimes)$ is $\mathrm{IP}$.
\end{conj}

\noindent It is likely that if Conjecture~\ref{conj:omega-order} fails then Conjecture~\ref{conj:field} also fails.
Consider $(\R,<,+,2^\N)$.
We expect that the structure induced on $2^\N$ by $(\R,<,+,2^\N)$ is interdefinable with $(2^\N,<)$.
If $(2^\N,\oplus,\otimes)$ is a field then $(\R,<,+,2^\N,\oplus, \otimes)$ is d-minimal by \cite{FM-Sparse}.
If $(2^\N,\oplus,\otimes)$ is in addition $\nip$ then we expect $(\R,<,+,2^\N, \oplus,\otimes)$ to be $\nip$ as well.\newline

\noindent A natural weakening of Conjecture~\ref{conj:omega-order} is that $(\N,<,+,\oplus,\otimes)$ is $\mathrm{IP}$ when $(\N,\oplus,\otimes)$ is a field.
If this weaker conjecture fails then Fact~\ref{fact:ktt-converse} may be applied to construct a locally o-minimal expansion which is not field-type and defines an infinite field.
The theorem of Dolich and Goodrick described in the proof of Proposition~\ref{prop:presburger} shows that $(\N,<,\oplus,\otimes)$ has infinite dp-rank if $(\N,\oplus,\otimes)$ is a field.
However there are $\nip$ expansions of fields which do not eliminate $\exists^\infty$ such as $(\R,<,+,\times,\lambda^\Z)$.

\begin{conj}
\label{conj:directed}
If $\Sa F$ is a expansion of an infinite field which admits a definable directed family $\Cal X$ of finite subsets of $F$ such that $\bigcup \Cal X = F$ then $\Sa F$ is $\mathrm{IP}$.
\end{conj}

\noindent Conjecture~\ref{conj:directed} implies the d-minimal case of Conjecture~\ref{conj:field-1} as it is not difficult to show that if $\Sa R$ is d-minimal and $X$ is a zero-dimensional definable subset of $\R^n$ then there is a definable directed family $\Cal X$ of finite subsets of $\R^n$ such that $\bigcup \Cal X = X$.
It may also be worth observing that Conjecture~\ref{conj:directed} implies that an $\nip$ expansion of $(\N,<)$ cannot interpret an infinite field. \newline

\noindent 
The proof of Proposition~\ref{prop:char-p} is easily adapted to show that $(\N,<,\oplus,\otimes)$ is $\mathrm{IP}$ when $(\N,\oplus,\otimes)$ is a field of positive characteristic.
Proposition~\ref{prop:field-case} follows from the proof of Proposition~\ref{prop:interpret}.

\begin{prop}
\label{prop:field-case}
Suppose $\Sa F$ is a expansion of an infinite field which admits both a definable non-discrete field topology and a definable directed family $\Cal X$ of finite subsets of $F$ such that $\bigcup \Cal X = F$.
Then $\Sa F$ is $\tp$.
\end{prop}

\noindent
Every known unstable $\nip$ expansion of a field admits a definable non-discrete field topology.
Johnson~\cite{Johnson-top} shows that an unstable dp-minimal expansion of a field admits a definable non-discrete field topology.
Conjecture~\ref{conj:directed} would follow from Proposition~\ref{prop:field-case} and a generalization of Johnson's theorem.
This is probably not the right way to go about establishing Conjecture~\ref{conj:directed}.

\section{Strongly Dependent Expansions of $(\R,<,+)$}
\label{section:strongly-dependent}
\noindent 
We continue to suppose that $\Sa R$ is an expansion of $(\R,<,+)$.
In this section we analyse strongly dependent expansions of $(\R,<,+)$ by closed sets and in particular prove Theorem D.
We first discuss $(\R,<,+,\az)$-minimality.
We then show that a strongly dependent expansion of $(\R,<,+)$ by closed sets is either o-minimal or $(\R,<,+,\az)$-minimal for some $\alpha > 0$.
(This $\alpha$ is unique up to rational multiples.)
As an application we prove the strongly dependent case of Conjecture~\ref{conj:field-1}.
\newline

\noindent
We will make crucial use of the work of Dolich and Goodrick on strongly dependent expansions of ordered abelian groups~\cite{DG}.
Dolich and Goodrick work in the more general setting of strong theories.
So the reader may wonder if our results on strongly dependent expansions of $(\R,<,+)$ generalize to strong expansions.
This is true and will appear in forthcoming joint work with Hieronymi.
We do not know if an $\mathrm{NTP}_2$ expansion of $(\R,<,+)$ by closed sets is generically locally o-minimal.

\subsection{$\pmb{(\R,<,+,\Z)}$-minimality}
Recall that $\Sa R$ is $(\R,<,+,\Z)$-minimal if it expands $(\R,<,+,\Z)$ and every $\Sa R$-definable subset of $\R$ is $(\R,<,+,\Z)$-definable.
Equivalently: $\Sa R$ is $(\R,<,+,\Z)$-minimal if the $\Sa R$-definable subsets of $\R$ are all finite unions of sets of the form $A + I$ where $A  = s + t\N$ for some $s \in \R, t \in \Q$ and interval $I$.
Note that $(\R,<,+,\Z)$-minimality implies local o-minimality.
So we may apply the following result of Kawakami, Takeuchi, Tanaka, and Tsuboi~\cite{KTTT} on locally o-minimal expansions of $(\R,<,+)$ which define $\Z$.
\newline

\noindent
For each $n$ let $\iota_n$ be the bijection $[0,1)^n \times \Z^n \to \R^n$ given by $\iota_n(a,d) = a + d$ and let $+_1$ be the function $[0,1)^2 \to [0,1)$ given by declaring $t +_1 t' = t + t'$ when $t + t' < 1$ and $t +_1 t' = t + t' - 1$ otherwise.
So for any $A \subseteq [0,1)^n$ and $B \subseteq \Z^n$ we have
$$ \iota_n( A \times B ) = \bigcup_{b \in B} A + b. $$
Fact~\ref{fact:ktt} is \cite[Lemma 23, Theorem 24]{KTTT}.
Fact~\ref{fact:ktt} shows that any locally o-minimal expansion of $(\R,<,+)$ which defines $\Z$ is bi-interpretable with the disjoint union of an o-minimal expansion of $([0,1),<,+_1)$ and an expansion of $(\Z,<,+)$.
Fact~\ref{fact:ktt} is a converse to Fact~\ref{fact:ktt-converse}.
Informally Fact~\ref{fact:ktt} shows that a locally o-minimal expansion which defines $\Z$ may be decomposed into a ``fractional part" and an ``integer part".
This is a higher order analogue of the decomposition of a real number into a fractional part and an integer part.

\begin{fact}
\label{fact:ktt}
Suppose $\Sa R$ is locally o-minimal and defines $\Z$.
Every definable subset $[0,1)^n \times \Z^n$ is a finite union of sets of the form $A \times B$ for definable $A \subseteq [0,1)^n$ and $B \subseteq \Z^n$, hence every definable subset of $\R^n$ is a finite union of sets of the form
$$ \iota_n( A \times B ) = \bigcup_{ b \in B} A + b $$
for definable $A \subseteq [0,1)^n$ and $B \subseteq \Z^n$.
Therefore $\Sa R$ is bi-interpretable with the disjoint union of $\Sa I$ and $\Sa D$ where $\Sa I$ is the structure induced on $[0,1)$ by $\Sa R$ and $\Sa D$ is the structure induced on $\Z$ by $\Sa R$.
Note that $\Sa I$ is o-minimal as $\Sa R$ is locally o-minimal.
\end{fact}

\noindent
We also apply Fact~\ref{fact:mv}, a theorem of Michaux and Villemaire~\cite{MV}.

\begin{fact}
\label{fact:mv}
There are no proper $(\Z,+,<)$-minimal expansions of $(\Z,+,<)$.
\end{fact}

\noindent
We can now prove Proposition~\ref{prop:az-0}.

\begin{prop}
\label{prop:az-0}
Suppose $\Sa R$ is $(\R,<,+,\Z)$-minimal.
Then every $\Sa R$-definable subset of $\R^n$ is a finite union of sets of the form $\bigcup_{b \in B} b + A$ where $A$ is an $\Sa R$-definable subset of $[0,1)^n$ and $B$ is a $(\Z,<,+)$-definable subset of $\Z^n$.
In particular any subset of $\Z^n$ definable in $\Sa R$ is definable in $(\Z,<,+)$.
It follows that $\Sa R$ is bi-interpretable with the disjoint union of the induced structure on $[0,1)$ and $(\Z,<,+)$.
\end{prop}

\noindent
Strong dependence is preserved under disjoint unions and bi-interpretations, so Proposition~\ref{prop:az-0} shows that $(\R,<,+,\Z)$-minimality implies strong dependence.

\begin{proof}
By Fact~\ref{fact:ktt} it suffices to show that the structure induced on $\Z$ by $\Sa R$ is interdefinable with $(\Z,<,+)$.
It follows from the description of $(\R,<,+,\Z)$-definable sets that every subset $\Sa R$-definable of $\Z$ is $(\Z,<,+)$-definable.
Apply Fact~\ref{fact:mv}.
\end{proof}

\noindent
We now classify $(\R,<,+,\az)$-minimal structures.
Suppose $\Sa S$ is an o-minimal expansion of $(\R,<,+)$.
A \textbf{pole} is an $\Sa S$-definable surjection from a bounded interval to an unbounded interval.
We say that $\Sa S$ has \textbf{rational scalars} if the function $\R \to \R$ given by $t \mapsto \lambda t$ is only definable when $\lambda \in \Q$.

\begin{thm}
\label{thm:rz-min}
Fix $\alpha > 0$.
The following are equivalent:
\begin{enumerate}
    \item $\Sa R$ is $(\R,<,+,\az)$-minimal,
    \item There is a collection $\Cal B$ of bounded subsets of Euclidean space such that $(\R,<,+,\Cal B)$ is o-minimal and $\Sa R$ is interdefinable with $(\R,<,+,\Cal B,\az)$,
    \item There is an o-minimal expansion $\Sa S$ of $(\R,<,+)$ with has no poles and has rational scalars such that $\Sa R$ is interdefinable with $(\Sa S,\az)$.
    \item There is an o-minimal expansion $\Sa S$ of $(\R,<,+)$ such that $(\Sa S,\az)$ is locally o-minimal and $\Sa R$ is interdefinable with $(\Sa S,\az)$.
\end{enumerate}
\end{thm}

\noindent
As $(\R,<,+,\sin|_{[0,\pi]})$ is o-minimal and $(\R,<,+,\sin)$, $(\R,<,+,\sin|_{[0,\pi]},\pi\Z)$ are interdefinable, we see that $(\R,<,+,\sin)$ is $(\R,<,+,\pi\Z)$-minimal.
The equivalence of $(2)$ and $(3)$ follows from Fact~\ref{fact:edmundo}, a special case of a theorem of Edmundo~\cite{ed-str}.

\begin{fact}
\label{fact:edmundo}
The following are equivalent for an o-minimal expansion $\Sa S$ of $(\R,<,+)$,
\begin{itemize}
    \item $\Sa S$ has no poles and has rational scalars,
    \item $\Sa S$ is interdefinable with $(\R,<,+,\Cal B)$ for some collection $\Cal B$ of bounded subsets of Euclidean space.
\end{itemize}
\end{fact}

\noindent
We now prove Theorem~\ref{thm:rz-min}.

\begin{proof}
Rescaling reduces to the case $\alpha = 1$.
We first show $(1)$ and $(2)$ are equivalent.
\newline

\noindent
Suppose that $\Sa R$ is $(\R,<,+,\Z)$-minimal.
Let $\Cal B$ be the collection of all $\Sa R$-definable subsets of all $[0,1)^n$.
As $\Sa R$ is locally o-minimal Theorem~\ref{thm:equiv-1} shows that $(\R,<,+,\Cal B)$ is o-minimal.
It is immediate that $(\R,<,+,\Cal B,\Z)$ is a reduct of $\Sa R$.
Proposition~\ref{prop:az-0} shows that $\Sa R$ is a reduct of $(\R,<,+,\Cal B, \Z)$.
\newline

\noindent
Now suppose $(2)$.
Rescaling and translating reduces to the case when every element of $\Cal B$ is a subset of some $[0,1)^n$.
Let $\Sa I$ be the structure induced on $[0,1)$ by $(\R,<,+,\Cal B)$.
Note that $\Sa I$ is an o-minimal expansion of $([0,1),<,+_1)$.
Let $\Sa S$ be the expansion of $(\R,<,+)$ constructed from $\Sa I$ and $(\Z,<,+)$ as in Fact~\ref{fact:ktt-converse}.
The description of $\Sa S$-definable sets in Fact~\ref{fact:ktt-converse} shows that $\Sa S$ is $(\R,<,+,\Z)$-minimal.
It is immediate that $(\R,<,+,\Cal B,\Z)$ is a reduct of $\Sa S$ and the description of $\Sa S$-definable sets in  Fact~\ref{fact:ktt-converse} shows that $\Sa S$ is a reduct of $(\R,<,+,\Cal B, \Z)$.
So $\Sa S$ and $(\R,<,+,\Cal B, \Z)$ are interdefinable.
It follows that $(\R,<,+,\Cal B,\Z)$ is $(\R,<,+,\Z)$-minimal.
\newline

\noindent
Fact~\ref{fact:edmundo} shows that $(2)$ and $(3)$ are equivalent.
So $(1)$ and $(3)$ are equivalent, it follows that $(3)$ implies $(4)$ as $(\R,<,+,\az)$-minimality implies local o-minimality.
We show that $(4)$ implies $(3)$ by showing that if $\Sa S$ is an o-minimal expansion of $(\R,<,+)$ and $(\Sa S,\Z)$ is locally o-minimal then $\Sa S$ has no poles and has rational scalars.
\newline

\noindent
Suppose $\tau : I \to J$ is a pole.
Applying the monotonicity theorem for o-minimal structures, reflecting, and translating, we suppose that $J$ contains a final segment of $\R$, $\tau$ is strictly increasing, and $\tau$ is continuous.
Then $\tau^{-1}(\N)$ is an infinite discrete subset of a bounded interval so $(\Sa S,\Z)$ is not locally o-minimal.
Suppose $\lambda \in \R$ is irrational and the map $\R \to \R$, $t \mapsto \lambda t$ is $\Sa S$-definable.
Then $\Z + \lambda\Z$ is dense and co-dense in $\R$ so $(\Sa S,\Z)$ is not locally o-minimal.
\end{proof}

\subsection{Strong dependence}
We now show that strongly dependent expansions of $(\R,<,+)$ by closed sets are either o-minimal or $(\R,<,+,\az)$-minimal for some $\alpha > 0$ which is unique up to rational multiples.
We need (special cases of) several results of Dolich and Goodrick.
The first claim of Fact~\ref{fact:DG1} is a special case of \cite[Corollary 2.13]{DG}.
The second claim is a essentially a case of \cite[Theorem 2.18]{DG}.

\begin{fact}
\label{fact:DG1}
Suppose that $\Sa R$ is strongly dependent and $E \subseteq \R$ is definable and discrete.
Then $E$ has no accumulation points.
Furthermore there are $s_1,\ldots,s_n,t_1,\ldots,t_n$ such that $E = (s_1 + t_1\N) \cup \ldots \cup (s_n + t_n\N)$ and $t_i/t_j \in \Q$ for all $i,j$ such that $t_j \neq 0$.
\end{fact}

\noindent
We say that the second claim is ``essentially" a special case of \cite[Theorem 2.18]{DG} because that theorem is slightly incorrect as stated.
We explain.
Dolich and Goodrick assert, under the assumptions of Fact~\ref{fact:DG1}, that $E$ is a union of a finite set together with \textit{commensurable} progressions $s_1 + t_1 \N,\ldots,s_n + t_n \N$.
(Recall $s_1 + t_1 \N$ and $s_2 + t_2 \N$ are commensurable if  they both lie in $s + t\Z$ for some $s,t$.)
This need not be the case, for example $\N$ and $\alpha + \N$ are not commensurable for any $\alpha \in \R \setminus \Q$.
The mistake lies in \cite[Corollary 2.30]{DG}.
In the proof of that corollary it is asserted that if $s_1 + t_1 \N, s_2 + t_2 \N$ are both infinite and contained in $\R_{>0}$ and
$$ \{ t - \min\{ r \in s_1 + t_1 \N, s_2 + t_2 \N : r > t \} : t \in s_1 + t_1\N ,s_2 + t_2 \N \}$$
is disjoint from $[0,\varepsilon)$ for some $\varepsilon > 0$ then $s_1 + t_1 \N$ and $s_2 + t_2 \N$ are commensurable.
This is incorrect, but it does imply that $t_1/t_2 \in \Q$.
The rest of the proof is correct\footnote{Thanks to John Goodrick for discussions on this point.}.
\newline

\noindent
Fact~\ref{fact:bes-choffrut} is a theorem of B\`{e}s and Choffrut~\cite{bes-choffrut}.
We only need the case $n = 1$.
The case $n = 1$ was also proven independently by Miller and Speissegger~\cite{MS-correct}.

\begin{fact}
\label{fact:bes-choffrut}
Suppose $X$ is a subset of $\R^n$ which is definable in $(\R,<,+,\Z)$ and not definable in $(\R,<,+)$.
Then $(\R,<,+,X)$ defines $\Z$.
\end{fact}

\noindent
We can now describe strongly dependent expansions of $(\R,<,+)$ by closed sets.


\begin{theorem}
\label{thm:loc-o-min}
The following are equivalent:
\begin{enumerate}
\item $\Sa R$ is a strongly dependent expansion by closed sets.
\item $\Sa R$ is strongly dependent and noiseless.
\item $\Sa R$ is either o-minimal or $(\R,<,+,\az)$-minimal for some $\alpha > 0$.
\item $\Sa R$ is either o-minimal or interdefinable with $(\R,<,+,\Cal B,\az)$ for some $\alpha > 0$ and collection $\Cal B$ of bounded subsets of Euclidean space such that $(\R,<,+,\Cal B)$ is o-minimal.
\item $\Sa R$ is either o-minimal or locally o-minimal and interdefinable with $(\Sa S,\az)$ for some $\alpha > 0$ and o-minimal expansion $\Sa S$ of $(\R,<,+)$.
\end{enumerate}
\end{theorem}

\begin{proof}
Theorem~\ref{thm:main-reals} shows that $(1)$ implies $(2)$.
We show that $(2)$ implies $(3)$.
Suppose that $\Sa R$ is strongly dependent, noiseless, and not o-minimal.
The first claim of Fact~\ref{fact:DG1} and Theorem~\ref{thm:equiv-1} together show that $\Sa R$ is locally o-minimal.
We show that $\Sa R$ defines $\az$ for some $\alpha > 0$.
As $\Sa R$ is locally o-minimal and not o-minimal there is an infinite definable discrete subset $E$ of $\R$.
Let $s_1,\ldots,s_n,t_1,\ldots,t_n$ be as in Fact~\ref{fact:DG1}.
After possibly removing a finite subset of $E$ we suppose that $t_i \neq 0$ for all $i$.
Let $\alpha := |t_1|$.
As $t_i/\alpha \in \Q$ for all $2 \leq i \leq n$ each $s_i + t_i\N$ is $(\R,<,+,\az)$-definable.
So $E$ is $(\R,<,+,\az)$-definable.
As $E$ is not $(\R,<,+)$-definable, rescaling and applying Fact~\ref{fact:bes-choffrut} shows that $(\R,<,+,E)$ defines $\az$.
So $\Sa R$ defines $\az$.
\newline

\noindent
We may now apply Fact~\ref{fact:ktt}.
We get that any $\Sa R$-definable subset of $\R$ is a finite union of sets of the form $\bigcup_{b \in B} b + A$ where $A \subseteq [0,\alpha)$ is a finite union of intervals and singletons and $B \subseteq \az$ is $\Sa R$-definable.
So it suffices to show that an $\Sa R$-definable subset $B$ of $\az$ is $(\az,<,+)$-definable.
As $B$ is discrete this follows from Fact~\ref{fact:DG1}.
\newline

\noindent
The equivalence of $(3)$, $(4)$, and $(5)$ is Theorem~\ref{thm:rz-min}.
We show that $(3)$ implies $(1)$.
Suppose $\Sa R$ is $(\R,<,+,\az)$-minimal for fixed $\alpha > 0$.
Then $\Sa R$ is d-minimal so every definable set is constructible by Fact~\ref{fact:d-min-constructible}, so $\Sa R$ is an expansion by closed sets.
The remarks after Proposition~\ref{prop:az-0} show that $\Sa R$ is strongly dependent.
\end{proof}

\noindent
In the proof of Theorem~\ref{thm:loc-o-min} we could have applied the theorem of Dolich and Goodrick~\cite[Corollary 2.20]{DG} that any strongly dependent expansion of $(\Z,<,+)$ is interdefinable with $(\Z,<,+)$.
Dolich and Goodrick's theorem on $(\Z,<,+)$ is proven by combining that analogue of Fact~\ref{fact:DG1} over $\Z$ with Fact~\ref{fact:mv}.
\newline

\noindent
The following corollary is essentially observed in \cite[Corollary 2.4]{DG}.
It follows from Theorem~\ref{thm:loc-o-min} as a locally o-minimal expansion of $(\R,<,+,\times)$ is o-minimal.

\begin{theorem}
\label{thm:loc-o-min-field}
Suppose $\Sa R$ expands $(\R,<,+,\times)$.
Then the following are equivalent.
\begin{enumerate}
\item $\Sa R$ is a strongly dependent expansion by closed sets.
\item $\Sa R$ is strongly dependent and noiseless.
\item $\Sa R$ is o-minimal.
\end{enumerate}
\end{theorem}



\noindent
Theorem~\ref{thm:loc-o-min} shows that a strongly dependent expansion of $(\R,<,+)$ by closed sets is ``almost o-minimal".
This is reflected in the dp-rank.

\begin{cor}
\label{cor:dp-rank}
Suppose $\Sa R$ is an expansion of $(\R,<,+)$ by closed sets.
Then the dp-rank of $\Sa R$ is either $\geq \aleph_0$ or at most two.
\end{cor}

\noindent In contrast Dolich and Goodrick~\cite[Section 3.2]{DG} show that for any $n \geq 2$ there is a noisey expansion of $(\R,<,+)$ of dp-rank $n$.

\begin{proof}
Suppose $\Sa R$ has finite dp-rank.
Then $\Sa R$ is strongly dependent.
If $\Sa R$ is o-minimal then $\Sa R$ has dp-rank one, so suppose that $\Sa R$ defines $\az$ for some $\alpha > 0$.
Rescaling, we suppose that $\Sa R$ defines $\Z$.
As $\Sa R$ defines an infinite nowhere dense subset of $\R$, \cite[Lemma 3.3]{goodrick} shows that $\Sa R$ has dp-rank at least two.
The induced structure on $[0,1)$ is o-minimal and so has dp-rank one.
The induced structure on $\Z$ is interdefinable with $(\Z,<,+)$ and hence has dp-rank one.
Fact~\ref{fact:subadditive} shows that $[0,1) \times \Z$ has dp-rank at most two.
Dp-rank is preserved by definable bijections, and $[0,1) \times \Z \to \R$, $(t,k) \mapsto t + k$ is a definable bijection, so $\Sa R$ has dp-rank two.
\end{proof}

\subsection{Definable fields} The results above allow us to prove the strongly dependent case of Conjecture~\ref{conj:field-1}.
The strongly dependent case of Conjecture~\ref{conj:field} follows.
The proof is very similar to that of Proposition~\ref{prop:dis-group}.

\begin{lemma}
\label{lem:strong-field}
Suppose $\Sa R$ is a locally o-minimal expansion of $(\R,<,+)$ which defines $\az$ for some $\alpha > 0$.
Suppose $X$ is a nonempty definable zero-dimensional subset of $\R^n$.
Then there is a definable surjection $f : (\az)^m \to X$ for some $m$.
\end{lemma}

\begin{proof}
Rescaling reduces to the case $\alpha = 1$.
Let $\iota_n$ be as defined above.
Let $\Cal D$ be the collection of all images of cartesian powers of $\Z$ under definable functions.
An application of Fact~\ref{fact:ktt} shows that $X$ is a finite union of sets of the form $\iota_n(A \times B)$ for nonempty definable $A \subseteq [0,\alpha)^n$ and $B \subseteq \Z^n$.
It is easy to see that $\Cal D$ is closed under finite unions and subsets.
So we suppose $X = \iota_n(A \times \Z^n)$.
Suppose $A$ is infinite.
As the structure induced on $[0,\alpha)$ by $\Sa R$ is o-minimal we have $\dim A \geq 1$.
Fix $b \in \Z^n$.
Then $A + b \subseteq X$ hence
$$ \dim X  \geq \dim (A + b) = \dim A \geq  1. $$
So we may suppose that $A$ is finite.
Then $X$ is a finite union of images of $\Z^n$ under functions of the form $b \mapsto \iota_n(a,b)$.
So $X \in \Cal D$.
\end{proof}

\begin{prop}
\label{prop:conj-case-strong}
Suppose $\Sa R$ is a strongly dependent expansion of $(\R,<,+)$ by closed subsets of Euclidean space.
Then $\Sa R$ does not define an infinite zero-dimensional field.
If $\Sa R$ interprets an infinite field then $\Sa R$ is field-type.
\end{prop}

\begin{proof}
The second claim follows from the first claim by Theorem~\ref{thm:zilber}.
If $\Sa R$ is o-minimal then every zero-dimensional definable set is finite so we may assume that $\Sa R$ is not o-minimal.
Combining Theorem~\ref{thm:loc-o-min},  Lemma~\ref{lem:strong-field}, and Proposition~\ref{prop:az-0} it suffices to show that $(\Z,<,+)$ does not interpret an infinite field.
This is Proposition~\ref{prop:presburger}.
\end{proof}

\noindent We now classify strongly dependent expansions of $(\R,<,+)$ by closed sets which are not field-type.
Recall that we let $\rvec$ be the ordered vector space $(\R,<,+,(t \mapsto \lambda t)_{\lambda \in \R})$ of real numbers.
We let $\bvec$ be the expansion of $(\R,<,+)$ by the restriction of $t \mapsto \lambda t$ to $[0,1]$ for all $\lambda \in \R$.
Fact~\ref{fact:edmundo} shows that $\bvec$ has rational scalars so $\bvec$ is a proper reduct of $\rvec$.

\begin{prop}
\label{prop:strong-linear}
The following are equivalent:
\begin{enumerate}
    \item $\Sa R$ is a strongly dependent expansion by closed sets which does not interpret an infinite field.
    \item $\Sa R$ is either a reduct of $\rvec$ or of $(\bvec,\az)$ for some $\alpha > 0$.
\end{enumerate}
\end{prop}

\begin{proof}
$(1) \Rightarrow (2):$ If $\Sa R$ is o-minimal then $\Sa R$ is a reduct of $\rvec$ by the Peterzil-Starchenko trichotomy~\cite{PS-Tri}.
Suppose $\Sa R$ is not o-minimal.
Theorem~\ref{thm:loc-o-min} shows that $\Sa R$ is interdefinable with $(\R,<,+,\Cal B,\az)$ for some $\alpha > 0$ and collection $\Cal B$ of bounded subsets of Euclidean space such that $(\R,<,+,\Cal B)$ is o-minimal.
As $(\R,<,+,\Cal B)$ is not field-type it is a reduct of $\rvec$.
It follows easily from the semilinear cell decomposition \cite[Corollary 7.6]{Lou} that every bounded subset of $\R^m$ definable in $\rvec$ is already definable in $\bvec$.
So $(\R,<,+,\Cal B)$ is a reduct of $\bvec$. \newline

\noindent $(2) \Rightarrow (1):$ Note that $\rvec$ and $(\bvec,\az)$ are both expansions by closed subsets of Euclidean space.
If $\Sa R$ is a reduct of $\rvec$ then $\Sa R$ does not interpret an infinite field by the Peterzil-Starchenko trichotomy or (see also Theorem~\ref{thm:zilber}).
Fix $\alpha > 0$.
It suffices to show that $(\bvec,\az)$ does not interpret an infinite field.
Theorem~\ref{thm:rz-min} and Proposition~\ref{prop:az-0} together show that any bounded subset of $\R^n$ definable in $(\bvec,\az)$ is $\bvec$-definable, so $(\bvec,\az)$ is not field type.
Now apply Proposition~\ref{prop:conj-case-strong}.
\end{proof}

\subsection{Distal expansions}
\label{section:distal}
We finish this section with some observations concerning distality.
(See Simon~\cite{Simon-Book} for an account of distality.)
Any o-minimal structure is distal and $(\Z,<,+)$ is distal.
Distality is preserved under disjoint unions and bi-interpretations.
Corollary~\ref{cor:distal} follows.

\begin{cor}
\label{cor:distal}
Suppose $\Sa R$ is an expansion of $(\R,<,+)$ by closed sets.
If $\Sa R$ is strongly dependent then $\Sa R$ is distal.
\end{cor}

\noindent
Every $\nip$ expansion of $(\R,<,+)$ by closed subsets of Euclidean space of which we are aware is distal.
In particular $\slam$ is distal when $\Sa S$ is an o-minimal expansion of $(\R,<,+,\times)$ with rational exponents and $\lambda > 1$, see Hieronymi and Nell~\cite{HN-Distal}.

\begin{qst}
\label{qst:distal}
Is every $\nip$ expansion of $(\R,<,+)$ by closed sets distal?
\end{qst}

\noindent
Suppose $\Sa Z$ is a $\nip$ expansion of $(\Z,<,+)$.
Let $(\R,<,+,\Sa Z)$ be the expansion of $(\R,<,+)$ by all $\Sa Z$-definable subsets of $\Z^n$.
It follows from Fact~\ref{fact:ktt-converse} that $(\R,<,+,\Sa Z)$ is locally o-minimal and $\nip$ and the structure induced on $\Z$ by $(\R,<,+,\Sa Z)$ is interdefinable with $\Sa Z$.
So if $\Sa Z$ is not distal then $(\R,<,+,\Sa Z)$ is also not distal.
Thus if there is a non-distal $\nip$ expansion of $(\Z,<,+)$ then there is a non-distal $\nip$ expansion of $(\R,<,+)$ by closed subsets of Euclidean space.

\begin{qst}
\label{qst:distal1}
Is there is a non-distal $\nip$ expansion of $(\N,<)$?
\end{qst}

\section{Archimedean quotients}
\label{section:arch-quot}
\noindent
\textbf{Through this section $\Sa N$ is a highly saturated expansion of a dense ordered abelian group $(N,<,+)$.}
Given a positive element $a$ of $N$ we let $\mfin_a$ be the convex hull of $a\Z$ and $\minf_a$ be the set of $b \in N$ such that $|nb| < a$ for all $n$.
Then $\mfin_a$ and $\minf_a$ are convex subgroups of $(N,<,+)$.
We order $\om$ by declaring $a + \minf_a < b + \minf_a$ if every element of $a + \minf_a$ is strictly less then every element of $b + \minf_a$.
Then $\om$ is an ordered abelian group.
As $\mfin_a$ and $\minf_a$ are convex Fact~\ref{fact:convex} shows that they are externally definable.
So we consider $\mfin_a/\minf_a$ to be an imaginary sort of $\Sh N$.\newline

\noindent
Let $\st_a : \mfin_a \to \R$ be given by $\st_a(b) := \sup \{ \frac{m}{n} \in \Q : m a \leq nb \}$.
Then $\st_a$ is an ordered group homomorphism with kernal $\minf_a$.
Note that $\st_a(\alpha) = 1$.
Density of $(N,<,+)$ and saturation together imply that $\st_a$ is surjective.
So $\om$ is isomorphic to $(\R,<,+)$. \newline

\noindent Suppose $\Sa N$ expands an ordered field $(N,<,+,\times,0,1)$.
Let $\mfin = \mfin_1$, $\minf = \minf_1$, and $\st = \st_1$.
Then $\st$ is the residue map of the archimedean valuation on $\Sa N$.
Let $a$ be a positive element of $N$.
Then $x \mapsto a^{-1} x$ gives ordered group isomorphisms $\mfin_a \to \mfin$ and $\minf_a \to \minf$ and hence induces an isomorphism $\mfin_a/\minf_a \to \mfin/\minf$.
This isomorphism is definable in $\Sh N$.
So if $\Sa N$ expands an ordered field then we obtain nothing more then the residue field of the archimedean valuation.\newline

\noindent
If $\Sa N$ is o-minimal then $\Sh N$ is weakly o-minimal by the theorem of Baisalov and Poizat~\cite{BP-weak} (this also follows from Fact~\ref{fact:shelah}).
In this case it easily follows that the induced structure on $\mfin_a/\minf_a$ is o-minimal.
(A weakly o-minimal expansion of $(\R,<)$ is o-minimal.)
We view Proposition~\ref{prop:arch-quotient} as a generalization of this fact.

\begin{prop}
\label{prop:arch-quotient}
Suppose $\Sa N$ is $\nip$ and fix a positive element $a$ of $N$.
Then the structure induced on $\mfin_a/\minf_a$ by $\Sh N$ is isomorphic to a generically locally o-minimal expansion of $(\R,<,+)$.
If $\Sa N$ is strongly dependent then the induced structure is either isomorphic to an o-minimal expansion of $(\R,<,+)$ or an $(\R,<,+,\az)$-minimal expansion of $(\R,<,+,\az)$ for some $\alpha > 0$.
\end{prop}

\begin{proof}
For the sake of simplicity we identify $\mfin_a/\minf_a$ with $\R$.
We show that the structure induced on $\R$ by $\Sh N$ is strongly noiseless.
By Theorem~\ref{thm:equiv} it suffices to show that the structure induced on $[0,1]$ by $\Sh N$ is strongly noiseless.
Let $I = [0,a]$.
Note that $\st_a(I) = [0,1]$.
Let $E$ be the equivalence relation of equality modulo $\minf_a$ on $I$.
Applying density we select for each $n \geq 1$ an $a_n \in N$ satisfying $(n - 1)a_n < a \leq na_n$.
Let $F_n$ be the set of $(a,b) \in I^2$ such that $| a - b | < a_n$.
Then each $F_n$ is $\Sa N$-definable, $(F_n)_{n \in \N}$ is nested, and $\bigcap_{n \geq 1} F_n = E$.
So $E$ is $\bigwedge$-definable and Lemma~\ref{lem:basis} shows that  $(F_n)_{n \geq 1}$ is a subdefinable basis for $F$. \newline

\noindent By Theorem~\ref{thm:equiv} is suffices to show that the logic topology on $[0,1]$ agrees with the usual order topology.
As both topologies are compact Hausdorff it suffices to show that any subset of $[0,1]$ which is closed in the order topology is also closed in the logic topology.
For this purpose it is enough to fix elements $t < t'$ of $[0,1]$ and show that $[t,t']$ is closed in the logic topology.
For each $m,n,m',n'$ such that $\frac{m}{n} \leq t, t' \leq \frac{m'}{n'}$ we let $X^{m,n}_{m',n'}$ be the set of $a \in I$ such that $m \leq na$ and $n'a \leq m'$.
So each $X^{m,n}_{m',n'}$ is $\Sa N$-definable.
Then $\st^{-1}([a,b]) \cap I$ is the intersection of all $X^{m,n}_{m',n'}$ and is hence $\bigwedge$-definable.
So $[a,b]$ is closed in the logic topology.
The proposition now follows by Theorems~\ref{thm:main-noise}, \ref{thm:equiv}, and \ref{thm:loc-o-min}.
\end{proof}

\noindent We record Proposition~\ref{prop:quot-dp} for future use.

\begin{prop}
\label{prop:quot-dp}
Suppose $\Sa N$ is $\nip$ and fix a positive element $a$ of $N$.
Then the dp-rank of the induced structure on $\om$ does not exceed the dp-rank of $\Sa N$.
\end{prop}

\begin{proof}
As $\om$ is an $\Sh N$-definable image of $N$ the dp-rank of the induced structure does not exceed that of $\Sh N$.
Now apply Fact~\ref{fact:onus}.
\end{proof}




\section{Archimedean structures}
\noindent
In this section we prove Theorem F, which is broken up into several theorems.
\newline

\noindent It is natural to try to extend Theorem~\ref{thm:main-reals} to expansions of archimedean ordered abelian groups.
We describe counterexamples showing that the naive extension fails and then describe what we believe to be the correct extension.
Recall that the classical Hahn embedding theorem states that any archimedean ordered abelian group has a unique up to rescaling embedding into $(\R,<,+)$.
So we only consider substructures of $(\R,<,+)$.
Any discrete archimedean ordered abelian group is isomorphic to $(\Z,<,+)$, so we only consider the dense case.
\textbf{We suppose throughout this section that $(R,<,+)$ is an archimededan dense ordered abelian group, which we take to be a substructure of $(\R,<,+)$, and let $\Sa R$ be a structure expanding $(R,<,+)$.}

\subsection{Two counterexamples} Theorem~\ref{thm:main-reals} can fail for $(R,<,+)$.
Recall that $(R,<,+)$, like any ordered abelian group, is $\nip$.
Suppose $R = \Z + \lambda \Z$ for some $\lambda \in \R \setminus \Q$.
Then $(R,<,+)$ is noisey as $nR$ is dense and co-dense for any $n \geq 2$. \newline

\noindent We give a divisible counterexample.
Let $\ralg$ be the set of real algebraic numbers.
Then $(\R,<,+,\times,\ralg)$ is $\nip$~\cite{GH-Dependent}.
Fix $\lambda \in \R \setminus \ralg$.
Let $R := \ralg + \lambda \ralg$.
Note that $(R,+)$ is divisible.
Consider $(R,<,+,t \mapsto \lambda t)$, this is an $\nip$ expansion by a continuous function.
Elementary algebra yields $\lambda R = \lambda \ralg$, so $\lambda R$ is dense and co-dense.
This example is sharp in the following sense.
Simon and Walsberg~\cite{SW-dp} show that a dp-minimal expansion of a divisible archimedean ordered abelian group is weakly o-minimal.
The structure induced on $\ralg$ by $(\R,<,+,\times,\ralg)$ is weakly o-minimal hence dp-minimal.
So by Fact~\ref{fact:subadditive} $(R,<,+,t \mapsto \lambda t)$ has dp-rank two. \newline

\subsection{The completion of an archimedean structure}
\label{section:completion} \noindent
The correct generalization of Theorem~\ref{thm:main-reals} concerns the completion of $\Sa R$.
This construction is motivated by work in the o-minimal and weakly o-minimal settings, which we describe below.

\subsubsection{The o-minimal case}
\label{section:o-minimal-case}
Suppose $\Sa R$ is o-minimal.
Laskowski and Steinhorn~\cite{LasStein} show that there is a unique o-minimal expansion $\Sa S$ of $(\R,<,+)$ such that $\Sa R$ is an elementary submodel of $\Sa S$.
It follows immediately that the structure induced on $R$ by $\Sa S$ is a reduct of $\Sh R$.
The Marker-Steinhorn theorem~\cite{Marker-Steinhorn} shows that if $\Sa N$ is an elementary extension of $\Sa S$, and $Z \subseteq N^n$ is $\Sa N$-definable, then $Z \cap \R^n$ is definable in $\Sa S$.
So any subset of $R^n$ externally definable in $\Sa R$ is of the form $Y \cap R^n$ for some $\Sa S$-definable $Y \subseteq \R^n$.
So the structure induced on $R$ by $\Sa S$ is interdefinable with $\Sh R$.

\begin{prop}
\label{prop:LS}
Let $\Sa R$ be o-minimal and $\Sa S$ be the unique elementary extension of $\Sa R$ which expands $(\R,<,+)$.
Then $\Sa S$ is interdefinable with
\begin{enumerate}
    \item the expansion $\Sa S_1$ of $(\R,<,+)$ whose primitive $n$-ary relations are all sets of the form $\cl(X)$ for $\Sa R$-definable $X \subseteq R^n$, and
    \item the expansion $\Sa S_2$ of $(\R,<,+)$ whose primitive $n$-ary relations are all sets of the form $\cl(X)$ for $\Sh R$-definable $X \subseteq R^n$.
\end{enumerate}
\end{prop}

\noindent We believe that $(2)$ is the ``right" definition in the general $\nip$ setting.

\begin{proof}
\noindent Note that $\Sa S_1$ is a reduct of $\Sa S_2$.
It follows easily from o-minimal cell decomposition that any subset of $\R^n$ which is $0$-definable in $\Sa S$ is a boolean combination of closed subsets of $\R^n$ which are $0$-definable in $\Sa S$.
If $X \subseteq \R^n$ is closed and $0$-definable in $\Sa S$ then $X$ is the closure of $X'$ in $\R^n$ where $X'$ is the $\Sa R$-definable set defined by any parameter-free formula defining $X$.
So $\Sa S$ is a reduct of $\Sa S_1$. \newline

\noindent Work of Dolich, Miller, and Steinhorn~\cite[Section 5]{DMS1} shows that $(\Sa S,R)^\circ$ is interdefinable with $\Sa S$.
So if $X \subseteq R^n$ is definable in the structure induced on $R$ by $\Sa S$ then $\cl(X)$ is definable in $\Sa S$.
It now follows from the observations above that if $X \subseteq R^n$ is externally definable in $\Sa R$ then $\cl(X)$ is $\Sa S$-definable.
So $\Sa S_2$ is a reduct of $\Sa S$.
\end{proof}

\subsubsection{The weakly o-minimal case}
\label{section:weak-o-min}
A weakly o-minimal structure may not have weakly o-minimal theory.
However, a result of Marker, Macpherson, and Steinhorn~\cite[Theorem 6.7]{MMS-weak} shows that a weakly o-minimal expansion of an archimedean ordered abelian group has weakly o-minimal theory, so we ignore this distinction. \newline

\noindent Suppose $\Sa R$ is weakly o-minimal.
We describe the o-minimal completion of $\Sa R$.
This completion was first constructed by Wencel~\cite{Wencel-1,Wencel-2}.
The construction we use is due to Keren~\cite{Keren-thesis}, see also \cite{EHK-weak}.
They define this completion in the more general setting of non-valuational structures.
We specialize to the setting of archimedean structures.
(Any archimedean structure is non-valuational.) \newline

\noindent Let $\Sa C(\Sa R)$ be set of all $t \in \R$ such that $(-\infty,t)$ is definable in $\Sa R$.
Note that $R$ is a subset of $\Sa C(\Sa R)$ and $\Sa C(\Sa R)$ is a subgroup of $(\R,+)$.
The completion $\overline{\Sa R}$ of $\Sa R$ is the expansion of $(\Sa C(\Sa R),<,+)$ by all closures in $\Sa C(\Sa R)^n$ of $\Sa R$-definable subsets of $R^n$.
Then $\overline{\Sa R}$ is o-minimal and the structure induced on $R$ by $\overline{\Sa R}$ is interdefinable with $\Sa R$.\newline

\noindent Fact~\ref{fact:shelah} may be applied to show that the Shelah expansion of a weakly o-minimal structure is weakly o-minimal, so $\Sa R^{\text{Sh}}$ is weakly o-minimal.
Fact~\ref{fact:convex} shows that every convex subset of $R$ is definable in $\Sh R$.
So $\Sa C(\Sh R) = \R$ and $\overline{\Sh R}$ is an o-minimal expansion of $(\R,<,+)$.
Note that the structure induced on $R$ by $\overline{\Sh R}$ is interdefinable with $\Sh R$.

\subsubsection{The $\nip$ case}
Generalizing what is above we define the completion of $\Sa R$.
Let $\Sq R$ be the structure on $\R$ whose primitive $n$-ary relations are all sets of the form $\cl(X)$ for $\Sh R$-definable $X \subseteq R^n$.
Note that $\Sa R$ expands $(\R,<,+)$.
If $\Sa R$ is o-minimal then $\Sq R$ is interdefinable with the unique elementary extension of $\Sa R$ with domain $\R$.
If $\Sa R$ is weakly o-minimal then $\Sq R$ is $\overline{\Sh R}$.
If $R = \R$ then $\Sq R$ is the open core of $\Sh R$. \newline

\noindent
Proposition~\ref{prop:induced-1} follows as $\cl(X) \cap R^n = X$ for any $X \subseteq R^n$ which is closed in $R^n$.

\begin{prop}
\label{prop:induced-1}
Suppose $\Sa R$ is an expansion of $(R,<,+)$ by closed sets.
Then $\Sa R$ is a reduct of the structure induced on $R$ by $\Sq R$.
\end{prop}

\noindent
Let $\Sa N$ be a highly saturated elementary extension of $\Sa R$.
Let $\mfin$ be the convex hull of $R$ in $N$ and let $\minf$ be the set of $a \in N$ such that $|a| < b$ for all positive $b \in R$.
Following Section~\ref{section:arch-quot} we identify $\mfin/\minf$ with $\R$ and consider $\R$ to be an imaginary sort of $\Sh N$.
We let $\st : \mfin \to \R$ be the quotient map and, abusing notation, let $\st : \mfin^n \to \R^n$ be given by
$$ \st(a_1,\ldots,a_n) = (\st(a_1),\ldots,\st(a_n)) .$$
Given an $\Sa R$-definable subset $Y$ of $R^n$ we let $Y'$ be the subset of $N^n$ defined by any formula which defines $Y$.

\begin{prop}
\label{prop:arch-complete}
Suppose $\Sa R$ is $\nip$. The the following are interdefinable
\begin{enumerate}
    \item $\Sq R$,
\item The expansion of $(\R,<,+)$ by all $\st(Y \cap \mfin^n)$ for $\Sa N$-definable $Y \subseteq N^n$,
\item The open core of the structure induced on $\R$ by $\Sh N$.
\end{enumerate}
In particular $\Sq R$ is a reduct of the structure induced on $\R$ by $\Sh N$.
So $\Sa R^{\square}$ is $\nip$ and generically locally o-minimal and if $\Sa R$ is strongly dependent then $\Sq R$ is either o-minimal or $(\R,<,+,\az)$-minimal for some $\alpha > 0$.
\end{prop}

\noindent
We do not know if $\Sq R$ is always interdefinable with the structure induced on $\R$ by $\Sh N$.
We need the next two propositions to prove the first claim of Proposition~\ref{prop:arch-complete}.

\begin{lemma}
\label{lem:honest-apply}
Suppose $X \subseteq R^n$ is externally definable in $\Sa R$ and $Y \subseteq N^n$ is an honest definition of $X$.
Then $\cl(X) = \st(Y \cap \mfin^n)$.
\end{lemma}

\begin{proof}
As $X \subseteq Y \cap \mfin^n$ and $\st$ is the identity on $R^n$ we see that $X$ is contained in $\st(Y \cap \mfin^n)$.
An easy saturation argument shows that $\st(Y \cap \mfin^n)$ is closed so $\cl(X)$ is contained in $\st(Y \cap \mfin^n)$.
We prove the other inclusion.
Let $I_1,\ldots,I_n$ be nonempty open intervals in $\R$ with endpoints in $R$.
Let $U$ be $I_1 \times \ldots \times I_n$.
Note that the collection of such open boxes forms a basis for $\R^n$.
It suffices to suppose $\cl(X)$ is disjoint from $U$ and show that $\st(Y \cap \mfin^n)$ is disjoint from $U$.
It is enough to show that $Y$ is disjoint from $\st^{-1}(U)$.
Note that $V := R^n \cap U$ is definable in $\Sa R$.
Then $X$ is disjoint from $V$, so $Y$ is disjoint from $V'$ by honestness.
It is easy to see that $\st^{-1}(U)$ is a subset of $V'$.
So $Y$ is disjoint from $\st^{-1}(U)$.
\end{proof}

\begin{lemma}
\label{lem:closed-interdef}
Suppose $\Sa R$ is $\nip$.
Suppose that $X$ is a closed subset of $\R^n$ which is definable in $\Sh N$.
Then $X$ is definable in $\Sq R$.
\end{lemma}

\noindent
Given $a = (a_1,\ldots,a_m) \in N^m$ we let $\|a\| = \max\{ |a_1|,\ldots,|a_m|\}$.
We will apply the fact that if $p,p' \in N^k$ and $q,q' \in N^l$ then $ \| (p,q) - (p',q') \| = \max \{ \| p - p' \|, \| q - q' \| \}. $

\begin{proof}
Let $Y := \st^{-1}(X)$, so $Y$ is $\Sh N$-definable.
So 
$$ \{ (\delta,a) \in N_{>0} \times N^n : \|a - a'\| < \delta \quad \text{for some   } a' \in Y \} $$
is $\Sh N$-definable, hence externally definable in $\Sa N$.
Applying Fact~\ref{fact:extension-external} we see that
$$ W := \{ (\delta,a) \in R_{>0} \times R^n : \| a - a ' \| < \delta \quad \text{for some  } a' \in Y \} $$
is $\Sh R$-definable.
Let $Z := \cl(W) \cap (\R_{>0} \times \R^n)$.
So $Z$ is $\Sq R$-definable.
We show that $X = \bigcap_{t > 0} Z_t$.
\newline

\noindent
To prove the left to right inclusion we fix $p \in X$ and $t \in \R_{>0}$ and show that $p \in Z_t$.
As $Z_t$ is closed it suffices to fix $\varepsilon \in \R_{>0}$ and produce $p' \in Z_t$ such that $\|p - p'\| < \varepsilon$.
We may suppose $\varepsilon < t$.
Fix $q \in Y$ such that $\st(q) = p$ and $p' \in R^n$ such that $\|p - p'\| < \varepsilon$.
As $\st \| p - q \| = 0$ we have $\| p' - q \| < \varepsilon < t$.
So $p' \in Z_t$.
\newline

\noindent
We prove the other inclusion.
Suppose $p \in Z_t$ for all $t \in \R_{>0}$.
We show $p \in X$.
As $X$ is closed it suffices to fix $\varepsilon \in \R_{>0}$ and produce $p' \in X$ such that $\| p - p' \| < \varepsilon$.
As $p \in Z_{\frac{1}{4}\varepsilon}$ we fix $(t,p'') \in W$ such that $\| (\frac{1}{4}\varepsilon,p) - (t,p'') \| < \frac{1}{4}\varepsilon$.
So $| t - \frac{1}{4}\varepsilon | < \frac{1}{4}\varepsilon$ and $\| p - p'' \| < \frac{1}{4}\varepsilon$.
As $(t,p'') \in W$ there is $q \in Y$ such that $\| p'' - q \| < t < \frac{1}{2}\varepsilon$.
Set $p' := \st(q)$, so $p' \in X$.
As $\st \| q - p' \| = 0$ we have $\|p'' - p'\| < \frac{1}{2}\varepsilon$, so $\|p - p'\| < \varepsilon$.
\end{proof}

\noindent
We prove Proposition~\ref{prop:arch-complete}.
As observed above it suffices to prove the first claim.

\begin{proof}
The second claim of Proposition~\ref{prop:arch-complete} follows immediately from the first claim and the third claim follows from the second claim and Proposition~\ref{prop:arch-quotient}.
So we prove the first claim.
Fact~\ref{fact:Honest} and Lemma~\ref{lem:honest-apply} together show that $(1)$ is a reduct of $(2)$.
An application of saturation shows that if $Y \subseteq N^n$ is $\Sa N$-definable then $\st(Y \cap \mfin^n)$ is closed.
So $(2)$ is a reduct of $(3)$.
Lemma~\ref{lem:closed-interdef} shows that $(3)$ is a reduct of $(1)$.
\end{proof}

\begin{cor}
\label{cor:strong-interdef}
If the structure induced on $\R$ by $\Sh N$ is d-minimal then the induced structure is interdefinable with $\Sq R$.
It follows that if $\Sa R$ is strongly dependent then $\Sq R$ is interdefinable with the structure induced on $\R$ by $\Sh N$.
\end{cor}

\begin{proof}
Fact~\ref{fact:d-min-constructible} shows that if the induced structure is d-minimal then the induced structure is interdefinable with its open core.
If $\Sa R$ is strongly dependent the by Proposition~\ref{prop:arch-quotient} the structure induced on $\R$ by $\Sh N$ is d-minimal.
\end{proof}

\noindent
We now investigate the structure induced on $R$ by $\Sq R$.
\begin{prop}
\label{prop:induced-external}
Suppose $\Sa R$ is $\nip$.
Then the structure induced on $R$ by $\Sq R$ is a reduct of $\Sh R$.
\end{prop}

\noindent
Propositions~\ref{prop:induced-1} and \ref{prop:induced-external} together show that if $\Sa R$ is an $\nip$ expansion by closed sets then the structure induced on $R$ by $\Sq R$ is between $\Sa R$ and $\Sh R$.

\begin{proof}
Suppose $X$ is a subset of $\R^n$ definable in $\Sq R$.
By Proposition~\ref{prop:arch-complete} $\st^{-1}(X)$ is definable in $\Sh N$, hence externally definable in $\Sa N$.
As $\st$ is the identity on $R^n$ we have $X \cap R^n = \st^{-1}(X) \cap R^n$.
By Fact~\ref{fact:extension-external} $X \cap R^n$ is externally definable in $\Sa R$.
\end{proof}

\noindent
If $\Sa R$ is weakly o-minimal then by \ref{section:weak-o-min} $\Sh R$ is interdefinable with the structure induced on $R$ by $\Sh R$.
The examples constructed in \cite{HNW} may be adapted to show that this fails for arbitrary noisey strongly dependent expansions of archimedean ordered abelian groups.

\begin{qst}
\label{qst:induced}
Suppose that $\Sa R$ is noiseless and $\nip$.
Must the structure induced on $R$ by $\Sq R$ be interdefinable with $\Sh R$?
\end{qst}

\noindent We give a positive answer to Question~\ref{qst:induced} when $\Sq R$ is d-minimal.
We expect natural examples of $\nip$ expansions of $(\R,<,+)$ by closed sets to be d-minimal, so Proposition~\ref{prop:d-min-induced} should cover most cases of Question~\ref{qst:induced} of interest.

\begin{prop}
\label{prop:d-min-induced}
Suppose $\Sa R$ is noiseless and $\nip$ and suppose $\Sq R$ is d-minimal.
For every $\Sh R$-definable $X \subseteq R^n$ there is an $\Sq R$-definable $Y \subseteq \R^n$ such that $X = Y \cap R^n$.
It follows that the structure induced on $R$ by $\Sq R$ eliminates quantifiers and is interdefinable with $\Sh R$.
\end{prop}

\noindent
The proof of Proposition~\ref{prop:d-min-induced} requires Lemma~\ref{lem:equiv-noise}.

\begin{lemma}
\label{lem:equiv-noise}
Suppose $\Sa R$ is $\nip$.
Then $\Sa R$ is noiseless if and only if $\Sa R$ is strongly noiseless.
\end{lemma}

\begin{proof}
Suppose $\Sa R$ is not strongly noiseless.
Let $X,Y$ be $\Sa R$-definable subsets of $R^n$ such that $X$ is somewhere dense in $Y$ and has empty interior in $Y$.
Let $U$ be a definable open subset of $Y$ such that $X$ is dense in $U$.
After replacing $Y$ with $U$ if necessary we suppose $X$ is dense in $Y$.
As $\Sq R$ is generically locally o-minimal an application of Fact~\ref{fact:c0-smooth} yields $0 \leq d \leq n$, a point $p \in Y$, an open neighbourhood $V$ of $p$, and a coordinate projection $\pi : \R^n \to \R^d$ such that $\pi$ gives a homeomorphism $\cl(Y) \cap U \to \pi(\cl(Y) \cap U)$ and $\pi(\cl(Y) \cap U)$ is an open subset of $\R^d$.
Then $\pi(X \cap U)$ is dense and co-dense in an open subset of $R^d$.
So $\Sa R$ is noisey.
\end{proof}

\noindent 
We now prove Proposition~\ref{prop:d-min-induced}.
The proof makes use of the Pillay rank defined in Section~\ref{section:d-min}.
We let $\cl_R(X)$ be the closure in $R^n$ of $X \subseteq R^n$.

\begin{proof}
The second claim follows from the first claim by Proposition~\ref{prop:induced-external}, so we only prove the first claim.
Lemma~\ref{lem:equiv-noise} shows that $\Sa R$ is strongly noiseless.
Proposition~\ref{prop:shelah-noise} shows that $\Sh R$ is strongly noiseless. \newline

\noindent Let $X$ be an $\Sh R$-definable subset of $R^n$.
We apply induction on the Pillay rank of $\cl(X)$.
If $\pr(\cl(X)) = -1$ then $X$ is empty, so we take $Y = \emptyset$.
Suppose $\pr(\cl(X)) \geq 0$.
Let $U$ be the interior of $X$ in $\cl_R(X)$ and $Z := \cl_R(X) \setminus U$.
Then $U,Z$ are $\Sh R$-definable and $Z$ is closed in $R^n$.
By strong noiselessness $U$ is dense in $\cl_R(X)$ so $Z$ is nowhere dense in $\cl_R(X)$ hence $\cl(Z)$ is nowhere dense in $\cl(X)$.
So $\pr(\cl(Z)) < \pr(\cl(X))$.
Proposition~\ref{prop:induced-external} shows that $Z \cap X$ is definable in $\Sh R$ so an application of induction shows that $Z \cap X = Y' \cap R^n$ for some $\Sq R$-definable $Y' \subseteq \R^n$.
We have $X \setminus Z = U$ and $U = [\cl(X) \setminus \cl(Z)] \cap R^n$.
So 
$$ X = [ Y' \cup (\cl(X) \setminus \cl(Z) )] \cap R^n $$
and
$$Y := Y' \cup (\cl(X) \setminus \cl(Z)) $$
is $\Sq R$-definable.
\end{proof}

\noindent Proposition~\ref{prop:arch-complete} shows that if $\Sa R$ is strongly dependent then $\Sq R$ is d-minimal.
So Proposition~\ref{prop:conj-case-strong} follows from Proposition~\ref{prop:d-min-induced}.

\begin{prop}
\label{prop:constructible}
Suppose $\Sa R$ is strongly dependent and noiseless.
For every $\Sh R$-definable $X \subseteq R^n$ there is an $\Sq R$-definable $Y \subseteq \R^n$ such that $X = Y \cap R^n$.
In follows that the structure induced on $R$ by $\Sq R$ eliminates quantifiers and is interdefinable with $\Sh R$. 
\end{prop}

\noindent Proposition~\ref{prop:constructible} shows that if $\Sa R$ is strongly dependent and noiseless then $\Sa R$ is ``locally weakly o-minimal".
One example of such a structure is the structure induced on $\Q$ by $(\R,<,+,\Z,\Q)$, which is strongly dependent by \cite[Proposition 3.1]{DG}. \newline

\noindent We show how the weakly o-minimal case may be recovered from the general case.
Suppose $\Sa R$ is weakly o-minimal.
Then $\Sa N$ is weakly o-minimal and so $\Sh N$ is weakly o-minimal by Fact~\ref{fact:shelah}.
It easily follows that the structure induced on $\R$ by $\Sh N$ is weakly o-minimal, hence o-minimal.
So $\Sq R$ is o-minimal.
The weak o-minimal cell decomposition~\cite[Theorem 4.11]{MMS-weak} shows that a weakly o-minimal structure is noiseless.
So the structure induced on $\R$ by $\Sq R$ is interdefinable with $\Sh R$. \newline

\noindent
As an application we describe the dp-ranks of noiseless expansions of archimedean ordered abelian groups.
Corollary~\ref{cor:dp-arch} generalizes Corollary~\ref{cor:dp-rank}.

\begin{cor}
\label{cor:dp-arch}
Suppose $\Sa R$ is noiseless.
Then the dp-rank of $\Sa R$ is either $\geq \aleph_0$ or at most two.
If $\Sa R$ is strongly dependent then the dp-rank of $\Sa R$ agrees with that of $\Sq R$.
\end{cor}

\noindent We will apply Proposition~\ref{prop:induced-dp}.
Onshuus and Usvyatsov's proof of Fact~\ref{fact:onus} immediately yields Proposition~\ref{prop:induced-dp}, we leave the details to the reader.

\begin{prop}
\label{prop:induced-dp}
Let $\Sa M$ be an $\nip$ structure.
Suppose $A$ is a subset of $M$ and suppose that the structure induced on $A$ by $\Sa M$ eliminates quantifiers.
Then the dp-rank of the induced structure on $A$ is no greater than the dp-rank of $\Sa M$.
\end{prop}

\noindent We now prove Corollary~\ref{cor:dp-arch}.

\begin{proof}
Suppose $\Sa R$ has finite dp-rank.
Then $\Sa R$ is strongly dependent.
Propositions~\ref{prop:constructible} and \ref{prop:induced-dp} together show that the dp-rank of $\Sa R$ does not exceed that of $\Sq R$.
Combining Proposition~\ref{prop:arch-complete} with Proposition~\ref{prop:quot-dp} we see that the dp-rank of $\Sq R$ does not exceed that of $\Sa R$.
Proposition~\ref{prop:arch-complete} and Corollary~\ref{cor:dp-rank} together show that $\Sq R$ has dp-rank at most two.
\end{proof}

\noindent Proposition~\ref{prop:arch-complete}, Proposition~\ref{prop:constructible}, and Corollary~\ref{cor:dp-arch} together show that $\Sa R$ is noiseless and dp-minimal if and only if $\Sa R$ is weakly o-minimal.
This also follows from \cite{SW-dp}. \newline

\noindent 
It now seems natural to study the ``dense-pair" $(\Sq R,\Sa R)$.
Hieronymi and G{\"u}naydin~\cite{GH-Dependent} show that $(\Sq R,\Sa R)$ is $\nip$ when $\Sa R$ is o-minimal.
Bar-Yehuda, Hasson, and Peterzil~\cite{BHP-pairs} (combined with \cite[Corollary 2.5]{CS-I}) shows that $(\overline{\Sa R},\Sa R)$ is $\nip$ when $\Sa R$ is weakly o-minimal, it follows that $(\Sq R, \Sa R)$ is $\nip$ when $\Sa R$ is weakly o-minimal.
Conjecture~\ref{conj:complete} will probably require more sophisticated tools then are currently available.

\begin{conj}
\label{conj:complete}
If $\Sa R$ is $\nip$ then $(\Sq R, \Sa R)$ is $\nip$.
\end{conj}

\noindent Conjecture~\ref{conj:complete} would reduce the study of $\nip$ expansions of dense archimedean ordered groups to the study of $\nip$ expansions of $(\R,<,+)$.

\subsection{Basic examples}
Finally, we discuss a few basic examples.
It will be shown in \cite{dp-embedd} that if $R$ is a dense subgroup of $(\R,<,+)$ then $(R,<,+)^\square$ is interdefinable with $(\R,<,+)$ and $(R,<,+)^{\mathrm{Sh}}$ is interdefinable with the structure induced on $R$ by $(\R,<,+)$.
So the completion of a dense archimedean ordered abelian group in our sense agrees with the usual notion of completion.
\newline

\noindent
Suppose that $R$ is an $\nip$ subfield of $(\R,<,+,\times)$.
One hopes that $(R,<,+,\times)^\square$ agrees with the usual completion, i.e. is interdefinable with $(\R,<,+,\times)$.
We show that this follows from the main conjecture on $\nip$ fields.
\newline

\noindent
It is a well-known conjecture that an infinite $\nip$ field is either separably closed, real closed, or admits a non-trivial Henselian valuation (see for example \cite{HHJ-conj}).
An ordered field cannot be separably closed, a Henselian valuation on an ordered field has a convex valuation ring~\cite[Lemma 4.3.6]{EP-value}, and an archimedean ordered field does not admit a non-trivial convex subring.
So the conjecture implies that an $\nip$ archimedean ordered field is real closed.
Suppose $R$ is real closed.
Then $(R,<,+,\times)$ is o-minimal and $(\R,<,+,\times)$ is an elementary extension of $(R,<,+,\times)$.
So it follows from Section~\ref{section:o-minimal-case} that $(R,<,+,\times)^\square$ is interdefinable with $(\R,<,+,\times)$ and $(R,<,+,\times)^{\mathrm{Sh}}$ is interdefinable with the structure induced on $R$ by $(\R,<,+,\times)$.

\section{Definable groups}
\label{section:pillay}
\noindent In this section we presume the reader is familiar with the theory of definable groups in $\nip$ structures and basic facts about definable groups in o-minimal structures.
Our reference for definable groups in $\nip$ structures is \cite{Simon-Book}.
\textbf{Throughout this section $\Sa N$ is a highly saturated $\nip$ structure, $G$ is an $\Sa N$-definable group, and $\pi$ is the quotient map $G \to G/G^{00}$.} \newline

\noindent Fact~\ref{fact:pillays} follows from the proof of \cite[Theorem 10.1]{HP-central-extensions}.
Fact~\ref{fact:pillays} depends on a long line of work on definable groups in o-minimal structures due to may authors.

\begin{fact}
\label{fact:pillays}
Suppose $\Sa N$ is an o-minimal expansion of an ordered field and $G$ is definably compact.
Then
\begin{enumerate}
    \item $G^{00}$ is externally definable in $\Sa N$,
    \item $G$ is compactly dominated by $\pi$, and
    \item $G/G^{00}$ is ``semi-o-minimal", i.e., $G/G^{00}$ is definable in a disjoint union of finitely many o-minimal expansions of $(\R,<)$, each of which is interpretable in $\Sh N$.
\end{enumerate}
\end{fact}
\noindent It is natural to ask if there are more general assumptions under which the conclusions of Fact~\ref{fact:pillays} hold.
There are satisfactory generalizations of $(1)$ and $(2)$.
Fact~\ref{fact:def-amenable}, proven by Hrushovski and Pillay~\cite[Lemma 8.2, Remark 8.3]{HP-invariant}, generalizes $(1)$.
Recall that a definably compact group in an o-minimal structure is definably amenable.

\begin{fact}
\label{fact:def-amenable}
Suppose that $G$ is definably amenable and $G/G^{00}$ is a Lie group.
Then $G^{00}$ is definable in $\Sh N$.
\end{fact}

\noindent Recall that $G/G^{00}$ is a Lie group if and only if there is a subgroup $H$ of  $G$ definable in $\Sa N$ such that $H$ has finite index in $G$ and $H$ has no $\Sa N$-definable subgroups of finite index.
If such $H$ exists then we view $H$ as a ``definably connected component of the identity".
Note also that $G/G^{00}$ is a connected Lie group if and only if $G$ has no $\Sa N$-definable subgroups of finite index, a condition we view as a form of ``definable connectedness". \newline

\noindent Fact~\ref{fact:distal-domination} generalizes $(2)$.
Fact~\ref{fact:distal-domination} follows from the measure-theoretic definition of distality~\cite[Proposition 9.26]{Simon-Book} and \cite[Proposition 8.32, Theorem 8.37, Lemma 8.36]{Simon-Book}.
Recall that an o-minimal structure is distal and a definable group in an o-minimal structure is fsg if and only if it is definably compact.

\begin{fact}
\label{fact:distal-domination}
If $\Sa N$ is distal and $G$ is fsg then $G$ is compactly dominated by $\pi$.
\end{fact}

\noindent The question remains if we can generalize $(3)$.
At a minimum one must assume $G/G^{00}$ is a Lie group.
It is also natural to assume that $G$ is definably amenable as to ensure $G^{00}$ is externally definable.

\begin{qst}
\label{qst:semi-o-minimal}
If $\Sa N$ is strongly dependent and $G$ is definably amenable then must $G/G^{00}$ be definable in a disjoint union of finitely many o-minimal expansions of $(\R,<)$, each of which is interpretable in $\Sh N$?
\end{qst}

\noindent Suppose Question~\ref{qst:semi-o-minimal} admits an affirmative answer.
It follows that if $\Sa N$ is distal and strongly dependant, $G$ is fsg, and $G$ has no $\Sa N$-definable subgroups of finite index, then $G^{00}$ is definable in $\Sh N$, $G$ is compactly dominated by $\pi$, and $G/G^{00}$ is semi-o-minimal.
Informally: if $\Sa N$ satisfies two key abstract properties of o-minimal structures, $G$ is ``definably compact" and ``definably connected", then $\Sh N$ defines an ``intrinsic standard part map" from $G$ to a compact connected Lie group $G/G^{00}$, the structure on $G$ is closely related to the induced structure on $G/G^{00},$ and the induced structure on $G/G^{00}$ is ``semi-o-minimal". \newline

\noindent We describe an example which should witness the necessity of strong dependence.
Let $\Sa S$ be an o-minimal expansion of $(\R,<,+,\times)$ with rational exponents.
Suppose $\Sa N$ is a highly saturated elementary expansion of $(\Sa S,\lambda^{\Z})$ for some $\lambda > 1$.
Let $G$ be the unit circle in $\Sa N$ equipped with complex multiplication.
We expect that $G$ is fsg hence definably amenable.
We further expect that $G/G^{00}$ can be identified with the unit circle in $\C$ and that the structure induced on $G/G^{00}$ by $\Sh N$ is interdefinable with the structure induced on the unit circle by $(\Sa S,\lambda^{\Z})$.
So in this case $G/G^{00}$ should not be semi-o-minimal. \newline

\noindent Theorem~\ref{thm:main-lie} handles the case when $G/G^{00}$ is a simple centerless Lie group.  
Suppose $H$ is a Lie group.
Then $H$ is said to be simple if every closed normal subgroup of $H$ is discrete.
Recall that any discrete subgroup of $H$ is contained in the center of $H$ and the center of $H$ is a normal subgroup.
So $H$ is simple and centerless if and only if $H$ has no non-trivial closed normal subgroups.

\begin{theorem}
\label{thm:main-lie}
Suppose $G$ is definably amenable and $G/G^{00}$ is a simple centerless Lie group.
Then the structure induced on $G/G^{00}$ by $\Sh N$ is bi-interpretable with a generically locally o-minimal expansion of $(\R,<,+,\times)$.
If $\Sa N$ is in addition strongly dependent then the induced structure on $G/G^{00}$ is bi-interpretable with an o-minimal expansion of $(\R,<,+,\times)$.
\end{theorem}

\noindent Let $H$ be a simple centerless compact Lie group.
Nesin and Pillay~\cite{NP-lie} show that $H$ (as an abstract group) is bi-interpretable with $(\R,<,+,\times)$.
It follows from their proof that $H$ (an as abstract group) defines a basis for the group topology on $H$.
So we do not distinguish between $H$ as an abstract group and $H$ as a Lie group.
(This idea is essentially already present in classical Lie theory.) \newline

\noindent Theorem~\ref{thm:main-lie} is proven in the following way.
Applying \cite{NP-lie} we see that the induced structure on $G/G^{00}$ is bi-interpretable with an expansion $\Sa R$ of $(\R,<,+,\times)$.
Theorem~\ref{thm:main-noise} shows $\Sa R$ is strongly noiseless, hence generically locally o-minimal.
Theorem~\ref{thm:loc-o-min-field} (essentially due to Dolich and Goodrick) shows that $\Sa R$ is o-minimal when $\Sa N$ is strongly dependent.

\begin{prop}
\label{prop:lie-bi-interp}
Suppose $H$ is a simple centerless compact Lie group.
Suppose $\Sa H$ is a strongly noiseless expansion of $H$.
Then $\Sa H$ is bi-interpretable with a generically locally o-minimal expansion of $(\R,<,+,\times)$.
If $\Sa H$ is in addition strongly dependent then $\Sa H$ is bi-interpretable with an o-minimal expansion of $(\R,<,+,\times)$.
\end{prop}


\begin{proof}
Note first that $\Sa H$ is bi-interpretable with an expansion of $(\R,<,+,\times)$ as $H$ is bi-interpretable with $(\R,<,+,\times)$.
We therefore regard $\R$ as an imaginary sort of $\Sa H$ and show that the structure induced on $\R$ by $\Sa H$ is generically locally o-minimal.
By \cite[Theorem 0.1, Corollary 2.3]{NP-lie} there is a subgroup $G$ of $\mathrm{Gl}_n(\R)$ and an $\Sa H$-definable continuous group isomorphism $u : G \to H$. 
O-minimal cell decomposition is easily applied to obtain a semialgebraic homeomorphic embedding $v : \R \to G$.
So $u \circ v$ is an $\Sa H$-definable homeomorphic embedding $\R \to H$.
It follows that the induced structure on $\R$ is strongly noiseless, hence generically locally o-minimal by Theorem~\ref{thm:equiv}. \newline

\noindent If $\Sa H$ is strongly dependent then the structure induced on $\R$ by $\Sa H$ is strongly dependent and noiseless, hence o-minimal by Theorem~\ref{thm:loc-o-min-field}.
\end{proof}

\noindent
We can now prove Theorem~\ref{thm:main-lie}.

\begin{proof}
Fact~\ref{fact:def-amenable} shows that $G^{00}$ is definable in $\Sh N$.
As the group topology on $G/G^{00}$ is definable in $\Sh N$, Theorem~\ref{thm:main-noise} shows that the structure induced on $G/G^{00}$ is strongly noiseless.
An application of Proposition~\ref{prop:lie-bi-interp} now shows that the structure induced on $G/G^{00}$ by $\Sh N$ is bi-interpretable with a generically locally o-minimal expansion of $(\R,<,+,\times)$ in general and bi-interpretable with an o-minimal expansion of $(\R,<,+,\times)$ when $\Sa N$ is strongly dependent.
\end{proof}

\noindent  Question~\ref{qst:semi-o-minimal} naturally splits into two questions.
Does $\Sh N$ define a basis for the logic topology on $G/G^{00}$?
If $H$ is a compact Lie group (considered as an abstract group), $\Sa H$ is a first order expansion of $H$ which defines a basis for the group topology on $H$, and $\Sa H$ is strongly dependent and strongly noiseless, then must $\Sa H$ be interpretable in a finite disjoint union of o-minimal expansions of $(\R,<)$?

\section{Modular speculations}
\label{section:modular}
\noindent It is a known open question to define a good notion of ``modularity" or ``one-basedness" for $\nip$ structures.
The Peterzil-Starchenko trichotomy~\cite{PS-Tri} seems to show that an o-minimal structure is ``modular" if and only if it does not define an infinite field.
In particular an o-minimal expansion $\Sa R$ of $(\R,<,+)$ satisfies exactly one of the following
\begin{enumerate}
    \item $\Sa R$ defines an infinite field,
    \item $\Sa R$ is a reduct of $\rvec$.
\end{enumerate}
So an o-minimal structure should be modular if and only if $(2)$ holds.
It seems likely that one might arrive at the right general definition by proving similar dichotomies for more general classes of structures.
We discuss the case of weakly o-minimal expansions of archimedean ordered abelian groups, more generally noiseless strongly dependent expansions of archimedean densely ordered abelian groups.
We make use of the notation and results of Section~\ref{section:completion}. \newline

\noindent We believe that a good notion of modularity for $\nip$ structures should satisfy the following:
\begin{enumerate}[label=(A\arabic*)]
    \item A modular structure cannot interpret an infinite field.
    \item Any ordered abelian group and any ordered vector space is modular.
    \item Modularity is preserved under disjoint unions and bi-interpretations.
    \item Suppose $\Sa M$ is a modular $\nip$ structure, $A$ is a subset of $M^x$, and the structure induced on $A$ by $\Sa M$ eliminates quantifiers.
    Then the induced structure on $A$ is modular.
    (So in particular the Shelah expansion of a modular $\nip$ structure is modular.)
\end{enumerate}

\noindent (A$1)$ is a well-known property of one-based stable structures, see for example \cite{HP-weakly-normal}.
We do not know if (A$4)$ has been observed in the literature for one-based stable structures, so we briefly indicate why it is true.
A formula $\phi(x,y)$ is weakly normal if there is $n$ such that whenever $a_1,\ldots,a_n \in M^{x}$ satisfy $\bigcap_{i = 1}^{n} \phi(a_i,M^y) \neq \emptyset$ then $\phi(a_i, M^y) = \phi(a_j,M^y)$ for some $1 \leq i < j \leq n$.
The structure $\Sa M$ is said to be weakly normal if every formula is a boolean combination of weakly normal formulas.
Hrushovski and Pillay show that $\Sa M$ is weakly normal if and only if $\Sa M$ is stable and one-based~\cite{HP-weakly-normal}.
It is easy to see that if $\Sa M$ is weakly normal, $A$ is a subset of $M^x$, and the structure induced on $A$ by $\Sa M$ eliminates quantifiers then the induced structure is weakly normal. \newline

\noindent Following Section~\ref{section:completion} we let $(R,<,+)$ be a dense archimedean ordered abelian group, which we take to be a substructure of $(\R,<,+)$, let $\Sa R$ be an expansion of $(R,<,+)$, and let $\Sa N$ be a highly saturated elementary extension of $\Sa R$.

\begin{prop}
\label{prop:modular}
Suppose $\Sa R$ is strongly dependent and noiseless.
Then one of the following holds.
\begin{enumerate}
    \item $\Sh N$ interprets $(\R,<,+,\times)$,
    \item every $\Sa R$-definable subset of $R^n$ is of the form $X \cap \R^n$ for some $\rvec$-definable subset $X$ of $\R^n$,
    \item there is an $\alpha > 0$ such that every $\Sa R$-definable subset of $R^n$ is of the form $X \cap R^n$ for some $(\bvec,\az)$-definable subset $X$ of $\R^n$.
\end{enumerate}
\end{prop}

\begin{proof}
Recall that $\Sq R$ is interpretable in $\Sh N$, so $(1)$ holds when $\Sq R$ is field-type.
Suppose $\Sq R$ is not field-type.
Proposition~\ref{prop:arch-complete} and Proposition~\ref{prop:strong-linear} together show that $\Sq R$ is either a reduct of $\rvec$ or $\bvec$.
Proposition~\ref{prop:constructible} shows that if $\Sq R$ is a reduct of $\rvec$ then $(2)$ holds and if $\Sq R$ is a reduct of $(\bvec,\az)$ then $(3)$ holds.
\end{proof}

\noindent Suppose for the moment that we have a notion of modularity satisfying (A$1)$-(A$4)$.
Fix $\alpha > 0$ and let $\Sa I$ be the structure induced on $[0,\alpha)$ by $\rvec$.
(A$2)$ and (A$4)$ together imply that $\Sa I$ is modular.
As $(\bvec,\az)$ is bi-interpretable with the disjoint union of $\Sa I$ and $(\Z,<,+)$, (A$2)$ and (A$3)$ together imply that $(\bvec,\az)$ is modular.
So if $\Sa R$ satisfies $(2)$ or $(3)$ of Proposition~\ref{prop:modular} then (A$4)$ implies $\Sa R$ is modular.

\begin{conj}
\label{conj:modular}
Suppose $\Sa R$ is strongly dependent and noiseless.
If $\Sq R$ is not field-type then $\Sh N$ does not interpret an infinite field.
\end{conj}

\noindent If Conjecture~\ref{conj:modular} fails then (A$1)$ - (A$4)$ are inconsistent.
It is presumably possible to give a ``brute-force" proof.
The ``right" solution would be to define a notion of modularity satisfying (A$1)$ - (A$4)$.
\newline

\noindent Block Gorman, Hieronymi, and Kaplan~\cite[Section 5.2]{GoHi-Pairs} study $(\rvec,\Q)$.
It follows from their work that $(\rvec,\Q)$ is $\nip$, the induced structure on $\Q$ is weakly o-minimal, and $(\rvec,\Q)^\circ$ is interdefinable with $\rvec$.
We expect the structure induced on $\Q$ by $(\rvec,\Q)$ to be a typical example of a modular weakly o-minimal structure. \newline

\noindent We now give an example of a weakly o-minimal structure which does not interpret an infinite field, but whose Shelah expansion interprets $(\R,<,+,\times)$.
This example also shows that modularity cannot be defined in terms of algebraic closure. \newline

\noindent
We first recall the Mann property.
We refer to G\"{u}naydin and van den Dries~\cite{vdDG-small-mult} for more information and references.
Let $\R_{>0}$ be the set of positive real numbers.
A solution $(s_1,\ldots,s_n)$ of the equation $a_1 x_1 + \ldots a_n x_n = 1$ is non-degenerate if $\sum_{i \in I} a_i s_i \neq 0$ for all $I \subseteq \{1,\ldots,n\}$.
A subgroup $G$ of $\C^\times$ has the \textbf{Mann property} if for any $a_1,\ldots,a_n \in \C^\times$ the equation $a_1 x_1 + \ldots + a_n x_n = 1$ has only finitely many non-degenerate solutions in $G^n$.
It is a deep diophantine result that any finite rank subgroup of $\C^\times$ has the Mann property.
In particular the group of complex roots of unity has the Mann property and if $a_1,\ldots,a_m \in \C^\times$ then $ \{ a_1^{q_1} \ldots a_m^{q_m} : q_1,\ldots,q_m \in \Q \}$ has the Mann property.
If $G$ is a dense subgroup of $(\R_{>0},\times)$ with the Mann property then $(\R,<,+,\times,G)$ is $\nip$. \newline

\noindent Fix a divisible subgroup $G$ of $(\R_{>0},\times)$ with the Mann property.
Note that $(G,<,\times)$ is archimedean.
Let $\Sa G$ be the structure induced on $G$ by $(\R,<,+,\times,G)$.
By \cite[Theorem 7.2]{vdDG-small-mult} any $\Sa G$-definable subset of $G^n$ is of the form $Y \cap G^n$ for semialgebraic $Y \subseteq \R^n$ (this requires divisibility).
In particular $\Sa G$ is weakly o-minimal and the structure induced on $G$ by $(\R,<,+,\times)$ eliminates quantifiers.
Given polynomials $f_1,\ldots,f_m \in \R[x_1,\ldots,x_n]$ it follows from \cite[Proposition 5.8]{vdDG-small-mult} that
$$ \{ g \in G^n : f_1(g) = \ldots = f_m(g) = 0 \} $$
is a finite union of cosets of subgroups of $G^n$ of the form
$$ \{ (g_1,\ldots,g_n) \in G^n : g_1^{i_1} \ldots g_n^{i_n} = 1\}$$
for integer $i_1,\ldots,i_n$
(they only prove this for algebraically closed fields, but the same proof goes through over $\R$).
A semialgebraic subset of $\R^n$ has empty interior if and only if its Zariski closure has empty interior (see for example \cite[2.8]{real-algebraic-geometry}), it follows that any $\Sa G$-definable subset of $G^n$ either has interior or is contained in a nowhere dense $(G,\times)$-definable subset of $G^n$.
It is now easy to see that algebraic closure in $\Sa G$ agrees with algebraic closure in $(G,\times)$.

\begin{prop}
\label{prop:mann}
Let $\Sa H$ be a highly saturated elementary extension of $\Sa G$.
Then $\Sa H$ does not interpret an infinite field but $\Sh H$ interprets $(\R,<,+,\times)$.
\end{prop}

\noindent A similar argument shows that if $U$ is the set of complex roots of unity, $\Sa U$ is the structure induced on $U$ by $(\R,<,+,\times,U)$, and $\Sa V$ is a highly saturated elementary expansion of $\Sa U$, then $\Sa V$ does not interpret an infinite field but $\Sh V$ interprets $(\R,<,+,\times)$.
We let $\cl(X)$ be the closure in $(\R_{>0})^n$ of $X \subseteq (\R_{>0})^n$.
Recall that an open set $U$ is \textbf{regular} if $U$ agrees with the interior of the closure of $U$.

\begin{proof}
Eleftheriou~\cite{E-small} shows that $\Sa G$, and hence $\Sa H$, eliminates imaginaries.
So it suffices to show that $\Sa H$ does not define an infinite field.
This may be deduced from the observations on algebraic closure above.
It also follows from Berenstein and Vassiliev~\cite[Proposition 2.11, Proposition 3.16]{BV-one-based}.
(It is easy to see that a structure which defines an infinite field defines a partial almost quasidesign in the sense of \cite{BV-one-based}.) \newline

\noindent We show that $\Sh H$ interprets $(\R,<,+,\times)$.
Let $\Sq G$ be the expansion of $(\R_{>0},<,\times)$ by $\cl(X)$ for every $\Sh G$-definable $X \subseteq G^n$.
The arguments of Section~\ref{section:completion} show that $\Sq G$ is interpretable in $\Sh H$.
We show that $\Sq G$ defines $(\R,<,+,\times)$.
It suffices to show that
$$ \{ (t,t',s) \in (\R_{>0})^3 : tt' < s \} \quad \text{and} \quad \{ (t,t',s) \in (\R_{>0})^3 : t + t' < s \} $$
are both definable in $\Sq G$.
Both of these are regular open subsets of $(\R_{>0})^3$.
So let $U$ be any regular open semialgebraic subset of $(\R_{>0})^n$.
Then $U \cap G^n$ is $\Sa G$-definable so $\cl(U \cap G^n)$ is $\Sq G$-definable.
As $U$ is open and $U \cap G^n$ is dense in $U$ we have $\cl(U \cap G^n) = \cl(U)$.
So $U$ is $\Sq G$-definable as it is the interior of $\cl(U)$.
\end{proof}

\bibliographystyle{abbrv}
\bibliography{NIP}

\end{document}